\documentclass[12pt,a4paper]{amsart}
\usepackage{amscd,amsmath,amssymb,amsfonts,times}
\usepackage{mathrsfs}
\usepackage[dvips]{color} 
\usepackage[all]{xy}
\usepackage{amsthm,miyamath,tikz-cd}
\mathchardef\hyphen="2D

\numberwithin{equation}{section}
    \newtheorem{thm}{Theorem}[section]
    \newtheorem{lem}[thm]{Lemma}
    
    \newtheorem{prop}[thm]{Proposition}
    \newtheorem{cor}[thm]{Corollary}

    \newtheorem{defn}[thm]{Definition}
    
    \newtheorem{rem}[thm]{Remark}
\DeclareMathAlphabet{\mathpzc}{OT1}{pzc}{m}{it}
\newenvironment{pf}{\par\smallskip\noindent\emph{Proof.}}{\hfill\qed\par\smallskip}
\newenvironment{pf*}[1]{\par\smallskip\noindent\emph{#1.}}{\hfill\qed\par\smallskip}
\begin{document}
\title{$F$-isocrystal and syntomic regulators via hypergeometric functions}
\author{Masanori Asakura}
\address{Department of Mathematics, Hokkaido University, Sapporo, 060-0810 Japan}
\email{asakura@math.sci.hokudai.ac.jp}
\author{Kazuaki Miyatani}
\address{Department of Mathematics, Hiroshima University, Hiroshima, 739-8526 Japan}
\email{miyatani@hiroshima-u.ac.jp}
\date\today
\maketitle

\def\can{\mathrm{can}}
\def\ch{{\mathrm{ch}}}
\def\Coker{\mathrm{Coker}}
\def\crys{\mathrm{crys}}
\def\dlog{d{\mathrm{log}}}
\def\dR{{\mathrm{d\hspace{-0.2pt}R}}}            
\def\et{{\mathrm{\acute{e}t}}}  
\def\Frac{{\mathrm{Frac}}}
\def\phami{\phantom{-}}
\def\id{{\mathrm{id}}}              
\def\Image{{\mathrm{Im}}}        
\def\Hom{{\mathrm{Hom}}}  
\def\Ext{{\mathrm{Ext}}}
\def\MHS{{\mathrm{MHS}}}  
  
\def\Ker{{\mathrm{Ker}}}          
\def\rig{{\mathrm{rig}}}
\def\Pic{{\mathrm{Pic}}}
\def\CH{{\mathrm{CH}}}
\def\NS{{\mathrm{NS}}}
\def\Fil{{\mathrm{Fil}}}
\def\End{{\mathrm{End}}}
\def\pr{{\mathrm{pr}}}
\def\Proj{{\mathrm{Proj}}}
\def\ord{{\mathrm{ord}}}
\def\reg{{\mathrm{reg}}}          %
\def\res{{\mathrm{res}}}          %
\def\Res{\mathrm{Res}}
\def\Spec{{\mathrm{Spec}}}     
\def\syn{{\mathrm{syn}}}
\def\cont{{\mathrm{cont}}}
\def\zar{{\mathrm{zar}}}
\def\bA{{\mathbb A}}
\def\bC{{\mathbb C}}
\def\C{{\mathbb C}}
\def\G{{\mathbb G}}
\def\bE{{\mathbb E}}
\def\bF{{\mathbb F}}
\def\F{{\mathbb F}}
\def\bG{{\mathbb G}}
\def\bH{{\mathbb H}}
\def\bJ{{\mathbb J}}
\def\bL{{\mathbb L}}
\def\cL{{\mathscr L}}
\def\bN{{\mathbb N}}
\def\bP{{\mathbb P}}
\def\P{{\mathbb P}}
\def\bQ{{\mathbb Q}}
\def\Q{{\mathbb Q}}
\def\bR{{\mathbb R}}
\def\R{{\mathbb R}}
\def\bZ{{\mathbb Z}}
\def\Z{{\mathbb Z}}
\def\cH{{\mathscr H}}
\def\cD{{\mathscr D}}
\def\cE{{\mathscr E}}
\def\cO{{\mathscr O}}
\def\O{{\mathscr O}}
\def\cJ{{\mathscr J}}
\def\cK{{\mathscr K}}
\def\cR{{\mathscr R}}
\def\cS{{\mathscr S}}
\def\cX{{\mathscr X}}
\def\cY{{\mathscr Y}}
\def\cZ{{\mathscr Z}}

%
\def\ve{\varepsilon}
\def\vG{\varGamma}
\def\vg{\varGamma}
%
%
%
%
\def\lra{\longrightarrow}
\def\lla{\longleftarrow}
\def\Lra{\Longrightarrow}
\def\hra{\hookrightarrow}
\def\lmt{\longmapsto}
\def\ot{\otimes}
\def\op{\oplus}
\def\l{\lambda}
\def\Isoc{{\mathrm{Isoc}}}
\def\FIsoc{{F\text{-}\mathrm{Isoc}^\dag}}
\def\FMIC{{F\text{-}\mathrm{MIC}}}
\def\FilFMIC{{\mathrm{Fil}\text{-}F\text{-}\mathrm{MIC}}}
\def\Log{{\mathscr{L}{og}}}
\newcommand{\sPol}{\mathop{\mathscr{P}\!\mathit{ol}}\nolimits}
\def\uq{\underline{q}}
\def\wt#1{\widetilde{#1}}
\def\wh#1{\widehat{#1}}
\def\spt{\sptilde}
\def\ol#1{\overline{#1}}
\def\ul#1{\underline{#1}}
\def\us#1#2{\underset{#1}{#2}}
\def\os#1#2{\overset{#1}{#2}}

\begin{abstract}
We give an explicit description of a syntomic regulator of a certain class of fibrations which we call hypergeometric fibrations. The description involves hypergeometric functions.
\end{abstract}

\section{Introduction}
In a preprint \cite{a-k2}, the first author discussed Beilinson's regulators
of a {\it hypergeometric fibration} $f:Y\to \P^1$ given by an equation
\[
f^{-1}(\l):y^N=x^a(1-x)^{N-b}(1-\l x)^b,\quad 0<a,b<N.
\]
$f$ is smooth over $\P^1\setminus\{0,1,\infty\}$, and the genus
of the general fiber is $N+1-\gcd(N,a)-\gcd(N,b)$.
It is endowed with multiplication by the group $\mu_N$ of $N$-th roots of $1$
by $[\zeta]\cdot(x,y,\l)=(x,\zeta^{-1}y,\l)$.
The cohomology group $H^1_B(f^{-1}(\l),\Q(\zeta_N))$ decomposes into 
the product of 2-dimensional
eigenspaces.
In \cite{a-k2}, we construct an element $\xi$
in Quillen's $K_2$ of $X:=f^{-1}(\P^1\setminus\{0,1,\infty\})$ (this is explicitly given as a 
$K_2$-symbol in case $(a,b)=(1,1)$, but otherwise we have only the existence), and 
give an explicit description of the Beilinson regulator $\reg(\xi)$
in terms of the {\it generalized hypergeometric functions} ${}_3F_2$ or ${}_4F_3$.
In case of $(N,a,b)=(2,1,1)$, $f$ is the Legendre family of elliptic curves, and then
one can expect the Beilinson conjecture on non-critical values of $L$-functions
and regulator. In loc.cit. we also give a numerical verification on the 
conjecture for our $\reg(\xi)$.

\medskip

The purpose of this paper is to give a $p$-adic counter part of \cite{a-k2}.
We consider the hypergeometric fibration given by an equation
\[
y^N=x(1-x)^{N-1}(1-\l x)
\]
which is defined over the Witt ring $W=W(\ol\F_p)$. There is an element $\xi\in K_2(X)$
constructed in \cite{a-k2}.
The main theorem is to give an explicit description of $\reg_\syn(\xi)$
in terms of hypergeometric function ${}_2F_1$, where 
$\reg_\syn$ is the {\it syntomic regulator map}
\[
\reg_\syn:K_2(X_\alpha)\lra H^2_\syn(X_\alpha,\Q_p(2))\cong H^1_\dR(X_\alpha/W)\ot\Q,
\quad X_\alpha:=f^{-1}(\alpha)
\]
to the syntomic cohomology of Fontaine-Messing (\cite{Ka2})
or the rigid syntomic cohomology of Besser (\cite{Be1}).
See Theorem \ref{main-thm} for the precise statement.
Besser showed that the syntomic regulator on $K_2$ of curves are
the Coleman integral (\cite{Be2} Theorem 3).
Thus our main theorem allows us to obtain a description of the Coleman
integral in terms of hypergeometric functions (Theorem \ref{coleman-thm})

To prove the main theorem, we make use of the theory of {\it $F$-isocrystals} which are, simply saying,
integrable connections with Frobenius action.
The strategy is as follows.
To an element $\xi\in K_2(X)$,
we associate an extension 
\[
0\lra H^1(X/S)(2)\lra M_\xi(X/S)\lra \O_S\lra 0
\]
of filtered $F$-isocrystals on $S:=\P^1\setminus\{0,1,\infty\}$
(\S \ref{Milnor-Ext-sect}).
Then its extension data gives the syntomic regulator (\S \ref{1-ext-sect}).
To compute the extension data, we take two steps.
First we derive ``differential equations'' of $\reg_\syn(\xi)$ from the 
commutativity of the Frobenius and connection (\S \ref{mainpf1-sect}).
To do this
we need to compute the matrix of the Frobenius on $H^1(X/S)$ completely, 
not only the eigenvalues but also non-diagonal entries. This is done in \S \ref{frob-sect}.
The second step is to determine the ``constant terms''.
To do this we compute the syntomic regulator
on $K_2$ of the degenerating curve $f^{-1}(0)$, which is essentially the regulator on 
$H^1_{\mathscr M}(\Spec W,\Q_p(2))$ and hence the $p$-adic dilogarithmic function
$\ln^{(p)}_2(z)$ appears (\S \ref{mainpf2-sect}).


\subsection*{Notation.}
For a integral domain $V$ and a $V$-algebra $R$ (resp. $V$-scheme $X$),
we denote by $R_K$ (resp. $X_K$) the tensoring $R\otimes_VK$ (resp. $X\times_VK$)
with the fractional field $K$.

Suppose that $V$ is a complete valuation ring $V$ endowed with a non-archimedian 
valuation $|\cdot|$.
For a $V$-algebra $B$ of finite type,
we denote by $B^{\dag}$ the weak completion of $B$.
Namely if $B=V[T_1,\cdots,T_n]/I$, then 
$B^\dag=V[T_1,\cdots,T_n]^\dag/I$ where $V[T_1,\cdots,T_n]^\dag$
is the ring of power series $\sum a_\alpha T^\alpha$ such that
for some $r>1$,
$|a_\alpha|r^{|\alpha|}\to0$ as $|\alpha|\to\infty$. 
We simply write $B^\dag_K=K\ot _V B^{\dag}$.
Let $X$ be a $V$-scheme of finite type.
Let $X_K\defeq X\times_{\Spec(V)}\Spec(K)$ and let $(X_K)^{\an}$ denote the rigid analytic space
associated to $X_K$.
Moreover, we let $\widehat{X}$ be the formal $V$-scheme obtained as the completion of $X$,
and let $\widehat{X}_K$ denote its generic fiber, which is a rigid analytic space.
Note that we have a canonical inclusion $\widehat{X}_K\hookrightarrow (X_K)^{\an}$
and that it is the identity if $X$ is proper.
If $Y$ is a sub-$k$-scheme of $X_k$, then
the tube of $Y$ in $\widehat{X}_K$ is denoted by $]Y[_{\widehat{X}}$ or
simply by $]Y[$ if no confusion will occur.
The natural inclusion is denoted by $j_Y\colon ]Y[\hookrightarrow\widehat{X}_K$.
If $W$ is a strict neighborhood of $]Y[$ in $\widehat{X}_K$,
then we also denote by $j_W\colon W\hookrightarrow\widehat{X}_K$ the inclusion.

\section{Milnor's $K$-theory and Extensions of filtered $F$-isocrystals}
\label{Milnor-Ext-sect}
Throughout this section we work over a complete discrete valuation ring $V$
such that the residue field $k$ is a perfect field of characteristic $p>0$.
We suppose that $V$ has a $p$-th Frobenius $\sigma$, namely an endomorphism
on $V$ such that $\sigma(x)\equiv x^p$ mod $pV$. 
Let $K=\Frac(V)$ be the fractional field.
The extension of $\sigma$ to $K$ is also denoted by $\sigma$.

\subsection{The category $\Fil$-$F$-$\MIC(S)$.}
Let $S=\Spec(B)$ be an affine smooth variety on $V$.
We define the category $\Fil$-$F$-$\MIC(S)$ as follows.
Let $\sigma_B\colon B^{\dag}\to B^{\dag}$ be a $p$-th Frobenius compatible with
$\sigma$ on $V$.
The induced endomorphism $\sigma_B\otimes_{\bZ}\bQ\colon B_K^{\dag}\to B_K^{\dag}$ 
is also denoted by $\sigma_B$. 
    An object of $\FilFMIC(S,\sigma_B)$ 
is a datum $H=(H_{\dR}, H_{\rig}, c, \Phi, \nabla, \Fil^{\bullet})$, where
    \begin{itemize}
            \setlength{\itemsep}{0pt}
        \item $H_{\dR}$ is a coherent $B_K$-module,
        \item $H_{\rig}$ is a coherent $B^{\dag}_K$-module,
        \item $c\colon H_{\dR}\otimes_{B_K}B^{\dag}_K\xrightarrow{\,\,\cong\,\,} H_{\rig}$ is a $B^{\dag}_K$-linear isomorphism,
        \item $\Phi\colon \sigma^{\ast}_BH_{\rig}\xrightarrow{\,\,\cong\,\,} H_{\rig}$ is an isomorphism of $B^{\dag}_K$-algebra,
        \item $\nabla\colon H_{\dR}\to \Omega_{B_K}^1\otimes H_{\dR}$ is an integrable connection and
        \item $\Fil^{\bullet}$ is a finite descending filtration on $H_{\dR}$ of locally free $B_K$-module (i.e. each graded piece is locally free),
    \end{itemize}
    that satisfies $\nabla(\Fil^i)\subset \Omega^1_{B_K}\ot \Fil^{i-1}$ and
    the compatibility of $\Phi$ and $\nabla$ in the following sense.
    Note first that $\nabla$ induces an integrable connection $\nabla_{\rig}\colon H_{\rig}\to\Omega^1_{B_K^{\dag}}\otimes H_{\rig}$,
    where $\Omega^1_{B_K^{\dag}}$ denotes the sheaf of continuous differentials.
    In fact, firstly regard $H_{\dR}$ as a coherent $\sO_{S_K}$-module.
        Then, by (rigid) analytification, we get an integrable connection $\nabla^{\an}$
        on a coherent $\sO_{(S_K)^{\an}}$-module $(H_{\dR})^{\an}$.
        Let $j\colon ]S_k[\hookrightarrow (S_K)^{\an}$ denote the inclusion and
        apply $j^{\dag}$ to $\nabla^{\an}$ to obtain
    an integrable connection on $\Gamma\left((S_K)^{\an}, j^{\dag}\big( (H_{\dR})^{\an}\big)\right)=H_{\dR}\otimes_{B_K}B_K^{\dag}$.
    This gives an integrable connection $\nabla_{\rig}$ on $H_{\rig}$
    via the isomorphism $c$.
Then the compatibility of $\Phi$ and $\nabla$ means
that $\Phi$ is horizontal with respect to $\nabla_{\rig}$, namely $\Phi\nabla_\rig=\nabla_\rig\Phi$.

    A morphism $H'\to H$ in $\FilFMIC(S,\sigma_B)$ 
    is a pair of homomorphisms $(h_{\dR}\colon H'_{\dR}\to H_{\dR}, h_{\rig}\colon H'_{\rig}\to H_{\rig})$,
    such that $h_{\rig}$ is compatible with $\Phi$'s, 
    $h_{\dR}$ is compatible with $\nabla$'s and $\Fil^{\bullet}$'s,
    and moreover they agree under the isomorphism $c$.

\medskip

Let $\sigma_B'$ be another $p$-th Frobenius on $B$. Then, as is well-known,
there is the natural equivalence $\FilFMIC(S,\sigma_B)\cong \FilFMIC(S,\sigma_B')$ of
categories. By virtue of this fact, we often drop ``$\sigma_B$'' in the notation.

\medskip

Let $n$ be an integer. The {\it Tate object} in $\Fil$-$F$-$\MIC(S)$, which
we denote by $B_K(n)$ or $\sO_{S_K}(n)$, is defined 
to be $(B_K, B_K^{\dag}, c, p^{-n}\sigma_B, 0, \Fil^{\bullet})$, where
$c$ is the natural isomorphism and where $\Fil^{\bullet}$ is defined by
$\Fil^{-n}B_K=B_K$ and $\Fil^{-n+1}B_K=0$.
We often abbreviate $B_K=B_K(0)$ and $\O_{S_K}=\O_{S_K}(0)$.
For an object $H$ of $\Fil$-$F$-$\MIC(S)$, we write $H(j)\defeq H\otimes B_K(j)$.

\medskip

A sequence
\[
 H_1\lra H_2\lra H_3
\]
in $\Fil$-$F$-$\MIC(S)$ is called {\it exact} if and only if
\[
 \Fil^i  H_{1,\dR}\lra \Fil^i H_{2,\dR}\lra \Fil^i H_{3,\dR}
\]
are exact for all $i$.
Then the category $\Fil$-$F$-$\MIC(S)$ forms an exact category which has
kernel and cokernel objects for any morphisms.
It also forms a tensor category in the usual manner with an unit object $B_K(0)$.
Thus the Yoneda extension groups
\[
\Ext^\bullet(H,H')=\Ext^\bullet_{\FilFMIC(S)}(H,H')
\]
in $\Fil$-$F$-$\MIC(S)$
are defined in the canonical way (or one can further define the derived category of
$\FilFMIC(S)$, \cite{BBDG}, 1.1).
An element of $\Ext^j(H,H')$ is represented by an exact sequence
\[
0\lra H'\lra M_1\lra\cdots\lra M_j\lra H\lra 0,
\]
subject to the equivalence relation generated by commutative diagrams
\[
\xymatrix{
0\ar[r]&H' \ar@{=}[d]\ar[r]&M_1\ar[d]\ar[r]&\cdots \ar[r]&M_j\ar[d]\ar[r]
&H\ar[r]\ar@{=}[d]&0\\
0\ar[r]&H' \ar[r]&N_1\ar[r]&\cdots \ar[r]&N_j\ar[r]&H\ar[r]&0.
}
\]
Note that $\Ext^\bullet(H,H')$ is uniquely divisible (i.e. $\Q$-module) as 
the multiplication by $m\in\Z_{>0}$ on $H$ or $H'$ is bijective.

\medskip

Let $S'=\Spec(B')$ be another affine smooth variety, and
$u\colon S'\to S$ a morphism of $V$-schemes.
Then, $u$ induces a pull-back functor
\[
    u^{\ast}\colon \Fil\hyphen F\hyphen\MIC(S)\to \Fil\hyphen F\hyphen\MIC(S')
\]
in a functorial way.
In fact, if $H=(H_{\dR}, H_{\rig}, c, \Phi, \nabla, \Fil^{\bullet})$ is an object
of $\Fil$-$F$-$\MIC(S)$, then we define
$u^{\ast}H=(H_{\dR}\otimes_{B_K}B'_K, H_{\rig}\otimes_{B^{\dag}_K}B'^{\dag}_K, c\otimes_{B^{\dag}_K}B'^{\dag}_K,
\Phi', \nabla', \Fil'^{\bullet})$,
where $\nabla'$ and $\Fil'^{\bullet}$ are natural pull-backs of $\nabla$ and $\Fil^{\bullet}$ respectively.
$\Phi'$ is the Frobenius structure obtained as follows;
regard $(H_{\rig}, \nabla_{\rig}, \Phi)$ as an overconvergent $F$-isocrystal on $S$
via the equivalence $F\hyphen\MIC(B_K^{\dag})\simeq F\hyphen\mathrm{Isoc}^{\dag}(B_k)$ \cite[Theorem 8.3.10]{LS}.
Then, its pull-back along $u_k\colon S'_k\to S_k$
is an overconvergent $F$-isocrystal on $S'_k$,
which is again identified with an object of $F\hyphen\MIC(B'^{\dag}_K)$;
it is of the form $(H'_{\rig}, \nabla'_{\rig}, \Phi')$, whose $\Phi'$ is the desired Frobenius structure.
\begin{rem}
Bannai constructed the category of syntomic coefficients \cite[1.8]{Bannai00}.
The reader will find that our $\FilFMIC(S)$ is adhoc and less polish than his, while
it is enough in the discussion in below.
\end{rem}

\subsection{$p$-adic Polylog functions}
For an integer $r$, we denote
the {\it $p$-adic polylog function} by
\begin{equation}\label{main-sect-eq1}
\ln^{(p)}_r(z):=\sum_{n\geq 1,\,p\not{\hspace{2pt}|}\,n}\frac{z^n}{n^r}
=
\lim_{s\to\infty}\frac{1}{1-z^{p^s}}
\sum_{1\leq n<p^s,\,p\not{\hspace{2pt}|}\,n}\frac{z^n}{n^r}
\in 
\Z_p\left\langle z,\frac{1}{1-z}\right\rangle
\end{equation}
where $\Z_p\langle z,(1-z)^{-1}\rangle$ denotes the $p$-adic completion of
the ring $\Z_p[z,(1-z)^{-1}]$.
As is easily seen, one has 
\[
\ln^{(p)}_r(z)=(-1)^{r+1}\ln^{(p)}_r(z^{-1}),\quad 
z\frac{d}{dz}\ln_{r+1}^{(p)}(z)=\mathrm{ln}_r^{(p)}(z).
\]
If $r\leq 0$, this is a rational function. Indeed
\[
\ln^{(p)}_0(z)=\frac{1}{1-z}-\frac{1}{1-z^p},\quad 
\ln^{(p)}_{-r}(z)=\left(z\frac{d}{dz}\right)^r\ln^{(p)}_0(z).
\]
If $r\geq 1$, it is no longer a rational function but an overconvergent function.
\begin{prop}\label{polylog-oc}
Let $r\geq 1$. Put $x:=(1-z)^{-1}$. Then $\ln^{(p)}_r(z)\in (x-x^2)\Z_p[x]^\dag$.
\end{prop}
\begin{pf}
Since $\ln^{(p)}_r(z)$ has $\Z_p$-coefficients, it is enough to check
$\ln^{(p)}_r(z)\in (x-x^2)\Q_p[x]^\dag$.
We first note
\begin{equation}\label{lnk-eq1}
(x^2-x)\frac{d}{dx}\mathrm{ln}_{k+1}^{(p)}(z)=\mathrm{ln}_k^{(p)}(z).
\end{equation}
The last term of \eqref{main-sect-eq1} is written as
\[
\lim_{s\to\infty}\frac{1}{x^{p^s}-(x-1)^{p^s}}
\sum_{1\leq n<p^s,\,p\not{\hspace{2pt}|}\,n}\frac{x^{p^s}(1-x^{-1})^n}{n^r}.
\]
This shows that $\ln^{(p)}_r(z)$ is divided by $x-x^2$.
Let $w(x)\in \Z_p[x]$ be defined by
\[
\frac{1-z^p}{(1-z)^p}=x^p-(x-1)^p=1-pw(x).
\]
Then
\[
\ln_1^{(p)}(z)=p^{-1}\log\left(\frac{1-z^p}{(1-z)^p}\right)
    =-p^{-1}\sum_{n=1}^\infty\frac{p^nw(x)^n}{n}\in \Z_p[x]^\dag.
\]
This shows $\ln^{(p)}_1(z)\in(x-x^2)\Q_p[x]^\dag$,
as required in case $r=1$. 
Let
\[
-(x-x^2)^{-1}\mathrm{ln}_1^{(p)}(z)
    =a_0+a_1x+a_2x^2+\cdots+a_nx^n+\cdots\in\Z_p[x]^\dag.
\]
By \eqref{lnk-eq1} one has
\[
\ln_2^{(p)}(z)=c+x+\frac{a_1}{2}x^2+\cdots
    +\frac{a_n}{n+1}x^{n+1}-\cdots\in \Q_p[x]^\dag
\]
and hence $\ln_2^{(p)}(z)\in(x-x^2)\Q_p[x]^\dag$,
as required in case $r=2$.
Continuing the same argument,
    we obtain $\ln_r^{(p)}(z)\in (x-x^2)\Q_p[x]^\dag$ for every $r$.
\end{pf}
\begin{rem}
The proof shows that $\ln_r^{(p)}(z)$ converges on an open disk $|x|<|1-\zeta_p|$.  
\end{rem}

\subsection{Polylog objects in $\FilFMIC$}\label{poly-obj-sect}
\begin{defn}\label{log-obj-def}
Let $B=V[T,(1-T)^{-1}]$ and $\sigma_B$ a $p$-th Frobenius such that $\sigma_B(T)=T^p$.
We define the {\it log object} $\sLog(1-T)$ in $\FilFMIC(S)$ as follows.
\begin{itemize}
    \setlength{\itemsep}{0pt}
    \item $\sLog(1-T)_{\dR}$ is a free $B_K$-module of rank two; $\sLog(1-T)_{\dR}=B_Ke_{-2}\oplus B_Ke_0$.
    \item $\sLog(1-T)_{\rig}\defeq\sLog(1-T)_{\dR}\otimes_{B_K}B_K^{\dag}=B_K^{\dag}e_{-2}\oplus B_K^{\dag}e_0$.
    \item $c$ is the natural isomorphism.
    \item $\Phi$ is the $B^{\dag}_K$-linear morphism defined by
        \[
            \Phi(e_{-2})=p^{-1}e_{-2},\quad \Phi(e_0)=
            \ln^{(p)}_1(T)e_{-2}+e_0.
            \]
    \item $\nabla$ is the connection defined by $\nabla(e_{-2})=0$ and $\nabla(e_0)=\frac{dT}{T-1}e_{-2}$.
    \item $\Fil^{\bullet}$ is defined by
        \[
            \Fil^{i}\sLog(1-T)_{\dR}=\begin{cases} \sLog(1-T)_{\dR} & \text{ if } i\leq -1,\\
                B_Ke_0 & \text{ if } i= 0,\\
               0 & \text{ if } i> 0.\end{cases}
        \]
\end{itemize}
\end{defn}
Note that $\ln^{(p)}_1(T)$ belongs to $B^{\dag}$ by 
Proposition \ref{polylog-oc}.
It is straightforward to check that these data define an object of $\Fil$-$F$-$\MIC(S)$.
It would be convenient here to remark that we have a situation where
we need to regard $\sLog(1-T)_{\dR}$ not only as usual connections on $\bA^1_K\setminus\{1\}$ but as a logarithmic connections on $\bP^1_K$ with logarithmic singularity along $\{1,\infty\}$
(It is easy to check that the connection $\nabla$ has logarithmic singularity along $\{1,\infty\}$).
Similarly, the overconvergent $F$-isocrystal $\sLog(1-T)_{\rig}$ with the connection $\nabla^{\an}$ and $\Phi$ can be
naturally considered as an log.\ $F$-isocrystal on the log.\ scheme $(\bP^1_k,\{1,\infty\})$
overconvergent along $\{1,\infty\}$.

\medskip

For a general $S=\Spec B$ and $f\in B^\times$, we define
the log object $\sLog(f)$ to be the pull-back $u^*\sLog(1-T)$ where $u:\Spec B\to 
\Spec V[T,(1-T)^{-1}]$ is given by $u(1-T)=f$.
This is fit into the exact sequence
\begin{equation}\label{log-obj-gp}
    0\lra B_K(1)\lra \sLog(f)\lra B_K\lra 0
\end{equation}
in $\FilFMIC(S)$.

\begin{defn}\label{poly-obj-def}
Let $C=V[T, T^{-1}, (1-T)^{-1}]$, and $\sigma_C$ a $p$-th Frobenius such that $\sigma_C(T)=T^p$.
Let $n\geq 1$ be an integer.
    We define the $n$-th polylog object $\sPol_n(T)$ of $\Fil\hyphen F\hyphen\MIC(\Spec C)$ as follows.
    \begin{itemize}
        \item $\sPol_n(T)_{\dR}$ is a free $C_K$-module of rank $n+1$;
            $\sPol_n(T)_\dR=\bigoplus_{j=0}^n C_K\, e_{-2j}.$
                 \item 
                 $\sPol_n(T)_{\rig}\defeq
                 \sPol_n(T)_{\dR}\otimes_{C_K}C_K^{\dag}$.
        \item $c$ is the natural isomorphism.
        \item $\Phi$ is the $C_K^{\dag}$-linear morphism defined by
            \[
                \Phi(e_0)=e_0-\sum_{j=1}^n(-1)^j\ln_j^{(p)}(T)e_{-2j},\quad
                \Phi(e_{-2j})=p^{-j}e_0,\quad(j\geq1).
                    \]
    \item $\nabla$ is the connection defined by 
    \[\nabla(e_0)=\frac{\rd T}{T-1}e_{-2},\quad
    \nabla(e_{-2j})=\frac{\rd T}{T}e_{-2j-2},\quad(j\geq1)
    \]
    where $e_{-2n-2}:=0$.
    \item $\Fil^{\bullet}$ is defined by 
 $\Fil^m\sPol_n(T)_{\dR}=\bigoplus_{0\leq j\leq -m} C_K\, e_{-2j}.$
 In particular, $\Fil^m\sPol_n(T)_{\dR}=\sPol_n(T)_{\dR}$ if $m\leq -n$ and $=0$ if $m\geq 1$.
    \end{itemize}
When $n=2$, we also write $\sDilog(T)=\sPol_2(T)$, and call it the {\it dilog object}.
\end{defn}
For a general $S=\Spec B$ and $f\in B$ satisfying $f,1-f\in B^\times$, 
we define
the polylog object $\sPol_n(f)$ to be the pull-back $u^*\sPol_n(T)$ where $u:\Spec B\to 
\Spec V[T,T^{-1},(1-T)^{-1}]$ is given by $u(T)=f$.
When $n=1$, $\sPol_1(T)$ coincides with 
the log object $\sLog(1-T)$ in $\FilFMIC(C)$.

\subsection{Relative cohomologies.}
In this subsection, we discuss objects in $\Fil$-$F$-$\MIC(S)$
arising as a relative cohomology of smooth morphisms (with a condition about compactification).
Firstly, we describe the situation for projective case.

\subsubsection*{Relative de Rham and rigid cohomology of proper smooth morphisms.}
Let $u\colon X\to S=\Spec(B)$ be a projective smooth morphism of smooth $V$-schemes.
Then, for each $i\geq 0$, it defines an object $H^i(X/S)$ of $\Fil$-$F$-$\MIC(S)$ in the following manner.
In fact, $H^i(X/S)=(H^i_{\dR}(X_k/S_k), H^i_{\rig}(X_K/S_K), c, \nabla, \Phi, \Fil^{\bullet})$ is defined as follows:
\begin{itemize}
    \setlength{\itemsep}{0pt}
    \item We define $H^i_{\dR}(X_K/S_K)$ to be the $i$-th relative de Rham cohomology $H^i_{\dR}(X_K/S_K)$
        and $\nabla$ \resp{$\Fil^{\bullet}$} to be the Gauss--Manin connection \resp{the Hodge filtration} on $H^i_{\dR}(X_K/S_K)$.
\item 
    Let $S\hookrightarrow\overline{S}$ be a projective completion of $S$ and
    let $\overline{u}\colon\overline{X}\to\overline{S}$ be a projective completion of the composite morphism $X\to S\hookrightarrow\overline{S}$
        Let $R^iu_{\rig,\ast}j_{X_k}^{\dag}\sO_{]\overline{X}_k[}$ denote the $i$-th relative rigid cohomology of $u_k\colon X_k\to S_k$.
    This is an overconvergent $F$-isocrystal on $S$ \cite[Th\'eor\`eme 5]{Berthelot81};
        equivalently (by taking the global section on $]\overline{S}_k[$), it is regarded as a coherent $B_K^{\dag}$-module with an integrable connection and a Frobenius structure.
        We define $H^i_{\rig}(X_k/S_k)$ to be this coherent $B_K^{\dag}$-module
        and $\Phi$ to be the Frobenius structure.
    \item $c$ is the comparison isomorphism between the rigid and the algebraic de Rham cohomology, which is constructed in the following manner.

We consider the natural morphism
        \begin{equation}
            \label{eq:comparisonisom}
            R^i]\overline{u}_k[_\ast\Omega^{\bullet}_{]\overline{X}_k[/]\overline{S}_k[}\otimes j_{S_k}^{\dag}\sO_{]\overline{S}_k[}
            \longrightarrow R^i]\overline{u}_k[_\ast j_{X_k}^{\dag}\Omega^{\bullet}_{]\overline{X}_k[/]\overline{S}_k[}
        \end{equation}
        and firstly  check that this is an isomorphism.
        Because it is a morphism of coherent $j^{\dag}_{S_k}\sO_{]\overline{S}_k[}$-modules
        (which follows from the properness of $]\overline{u}[_k$ and Kiehl's theorem),
        it suffices to show that its restriction to $]S_k[$ is an isomorphism.
        Moreover, it suffices to prove that it is an isomorphism at each fiber,
        and in fact this is true because of the base change theorem and
        the comparison of the (absolute) algebraic de Rham cohomology
        and the rigid cohomology \cite[(7)]{Gerkmann08}.

        Now, the desired comparison isomorphism $c$ is
        obtained by taking its global section on $(S_K)^{\an}$ of (\ref{eq:comparisonisom}).
        In fact, the global section of the target of (\ref{eq:comparisonisom}) on $(S_K)^{\an}$ is, being isomorphic
        to that on $]\overline{S}_k[$, nothing but the relative rigid cohomology $H^i_{\rig}(X_k/S_k)$.
        The global section of the source on $(S_K)^{\an}$ is
        \[
        \Gamma\left((S_K)^{\an}, R^i]\overline{u}_k[_{\ast}\Omega^{\bullet}_{]\overline{X}_k[/]\overline{S}_k[}\right)\otimes B^{\dag}_K,
    \]
        and the first tensorand is isomorphic to the algebraic de Rham cohomology $H^i_{\dR}(X_K/S_K)$ \cite[Theorem IV.4.1]{AndreBaldassarri}.
        We therefore get the comparison isomorphism
        \[
            c\colon H^i_{\dR}(X_K/S_K)\otimes_{B_K}B_K^{\dag}\xrightarrow{\quad\cong\quad} H^i_{\rig}(X_k/S_k).
        \]

        Note also that the Gauss-Manin connection on $H^i_{\rig}(X_k/S_k)$ (as a rigid cohomology)
        coincides with the connection induced by the Gauss-Manin connection $\nabla$ on $H^i_{\dR}(X_K/S_K)$ via $c$.

\end{itemize}
\subsubsection*{Relative cohomologies of non-projective cases.}
We will also need the ``relative cohomology object'' for families of open varieties.
Let $u\colon U\to S$ be a smooth morphism of smooth $V$-schemes of pure relative dimension,
and assume that there exists an open immersion $U\hookrightarrow X$
and a projective smooth morphism $u_X\colon X\to S$ that extends $u$.
Moreover, assume that $D\defeq X\setminus U$ is a relative normal crossing divisor with smooth components.
Let $S\hookrightarrow\overline{S}$ be a projective completion with $\overline{S}\setminus S$ a divisor of $\overline{S}$, and let $\overline{u}\colon\overline{X}\to\overline{S}$ be a projective completion 
of the composite morphism $X\xrightarrow{u_X}S\hookrightarrow\overline{S}$.

Let us construct an object
\[
    H^i(U/S)=(H^i_{\rig}(U_k/S_k), H^i_{\dR}(U_K/S_K), c, \nabla, \Phi, \Fil^{\bullet})
\]
of $\Fil$-$F$-$\MIC(S)$.
Let $H^i_{\rig}(U_k/S_k)$ be the $i$-th relative rigid cohomology of $u_k$,
namely, the global section of $R^i]\overline{u}_k[_{\ast}j^{\dag}_{U_k}\sO_{]\overline{X}_k[}$ on $]\overline{S}_k[$ (note that this is also its global section on $S_K^{\an}$);
the coherence of this module is explained later.
Let $H^i_{\dR}(U_K/S_K)$ be the relative algebraic de Rham cohomology of $u_K$.
The data $\nabla, \Phi$ and $\Fil^{\bullet}$ can be defined in the same way 
as in the projective smooth case.
The remaining task is to construct the comparison isomorphism $c$.

Put $T\defeq\overline{u}^{-1}(\overline{S}\setminus S)$; this is a divisor of $\overline{X}$.
Let $\overline{D}$ is the closure of $D$ in $\overline{X}$.
In order to construct the isomorphism,
we need the logarithmic theory of de Rham and rigid cohomology.
Let $\overline{X}^{\sharp}\defeq(\overline{X}, \overline{D})$ be the $V$-scheme $\overline{X}$ endowed with the logarithmic structure associated to $\overline{D}$;
$\overline{X}_K^{\an,\sharp}$ and $]\overline{X}_k[^{\sharp}$ are defined in a similar manner.

Then, we have a canonical comparison isomorphism
\begin{equation}
    R^i]\overline{u}_k[_{\ast}\left(j^{\dag}_{X_k}\Omega^{\bullet}_{]\overline{X}_k[^{\sharp}/]\overline{S}_k[}\right)
    \xrightarrow{\quad\cong\quad}
    R^i]\overline{u}_k[_{\ast}\left(j^{\dag}_{U_k}\Omega^{\bullet}_{]\overline{X}_k[/]\overline{S}_k[}\right)
    \label{eq:rigid-logrigid}
\end{equation}
between the relative log-rigid cohomology and the relative rigid cohomology;
in fact, this is the comparison isomorphism \cite[1.1.1]{CaroTsuzuki} in the case where the ambient morphism $\overline{u}$
is not necessarily smooth but smooth around $X_k$ \cite[1.1.2(4)]{CaroTsuzuki}.
The sheaf $R^i]\overline{u}_k[_{\ast}\left(j^{\dag}_{X_k}\Omega^{\bullet}_{]\overline{X}_k[^{\sharp}/]\overline{S}_k[}\right)$ is a coherent $j_{S_k}^{\dag}\sO_{]\overline{S}_k[}$-module by the argument as in the first half of the proof of \cite[4.1.1]{Tsuzuki03};
therefore, so is 
    $R^i]\overline{u}_k[_{\ast}\left(j^{\dag}_{U_k}\Omega^{\bullet}_{]\overline{X}_k[/]\overline{S}_k[}\right)$.

Now, consider the canonical morphism
\begin{equation}
    R^i]\overline{u}_k[_{\ast}\Omega^{\bullet}_{]\overline{X}_k[^{\sharp}/]\overline{S}_k[}\otimes j^{\dag}_{S_k}\sO_{]\overline{S}_k[}
\longrightarrow
    R^i]\overline{u}_k[_{\ast}\left(j^{\dag}_{X_k}\Omega^{\bullet}_{]\overline{X}_k^{\sharp}[/]\overline{S}_k[}\right).
    \label{eq:log-comparison}
\end{equation}
By the same argument as in the proper smooth case,
we may show that this morphism is an isomorphism.
Namely, since it is a morphism of coherent $j^{\dag}_{S_k}\sO_{]\overline{S}_k[}$-modules,
it suffices to prove that its restriction to $]S_k[$ is an isomorphism;
this can be checked fiber-wise and therefore we may reduce to the absolute case;
it follows from a comparison theorem of
Baldassarri and Chiarellotto \cite[Corollary 2.5]{BaldassarriChiarellotto}
together with the usual comparison of log.\ de Rham cohomology and de Rham cohomology.

Now, let us take the global section of the morphism (\ref{eq:log-comparison}) on $S_K^{\an}$.
The global section of the first tensorand of the source on $S_K^{\an}$ is isomorphic to the
algebraic log.\ de Rham $H^i_{\dR}(X_K^{\sharp}/S_K)$,
thus to the algebraic de Rham cohomology $H^i_{\dR}(U_K/S_K)$.
The global section of the target on $S_K^{\an}$ is the $i$-th rigid cohomology
$H^i_{\rig}(U_k/S_k)$ by (\ref{eq:rigid-logrigid}).
We thus get an isomorphism
\begin{equation}
    c\colon H^i_{\dR}(U_K/S_K)\otimes B_K^{\dag}\longrightarrow H^i_{\rig}(U_k/S_k)
    \label{eq:comparisonisom2}
\end{equation}
of coherent $B_K^{\dag}$-modules.

\begin{prop}[Gysin exact sequence]
    Let $\overline{u}\colon X\to S$ be a proper smooth morphism of relative dimension one,
    let $D$ be a divisor of $X$ such that $\overline{u}|_D$ is finite etale,
    and put $U\defeq X\setminus D$ and $u\defeq\overline{u}|_U\colon U\to S$.
    In the above notation, assume that $X$ is of relative dimension one over $S$
    and $D$ is a relative smooth divisor.
    Then, there exists an exact sequence
    \[
        0\to H^1(X/S)\to  H^1(U/S)\to H^0(D/S)(-1)\to H^2(X/S)\to 0.
    \]
    \label{prop:Gysin}
\end{prop}
\begin{pf}
    It is a standard fact about de Rham cohomology that we have an exact sequence 
    $0\to H^1_{\dR}(X_K/S_K)\to H^1_{\dR}(U_K/S_K)\to H^0_{\dR}(D_K/S_K)(-1)\to H^2_{\dR}(X_K/S_K)\to 0$
    the morphisms in which are horizontal and compatible with respect to $\Fil$.
    The relative rigid cohomology also have a Gysin exact sequence \cite[6.1.2]{Solomon}
    compatible with Frobenius morphisms.
    The construction shows that they are compatible with the comparison isomorphisms,
    which shows the proposition.
\end{pf}
\subsection{Extensions associated to $K_2$-symbols}
Let $f,g\in \sO(U)^\times$.
Recall from \S \ref{poly-obj-sect} the log objects $\sLog(f)$ and $\sLog(g)$ in 
$\FilFMIC(U)$.
The exact sequences \eqref{log-obj-gp} gives rise to a 2-extension
\[
    0\lra \sO_{U_K}(2)\lra \sLog(f)(1)\lra\sLog(g)\lra \sO_{U_K}\lra 0,
\]
which represents the cup-product $\sLog(f)(1)\cup\sLog(g)$ in
$\Ext^2(\sO_{U_K},\sO_{U_K}(2))$.

\begin{prop}   \label{2-ext-welldef}
$\sLog(f)(1)\cup\sLog(f)=0$ and $\sLog(f)(1)\cup\sLog(1-f)=0$.
Hence the homomorphism
\[
K^M_2(\O(U))\lra\Ext^2(\sO_{U_K},\sO_{U_K}(2)),\quad
\{f,g\}\mapsto\sLog(f)(1)\cup\sLog(g)
\]
is well-defined (note $\sLog(-f)=\sLog(f)$ by definition).
\end{prop}
To prove this, it is enough to show the following.
\begin{lem}
Let $C=V[T,T^{-1},(1-T)^{-1}]$. Then
    $\sLog(T)(1)\cup\sLog(T)=0$ and
    $\sLog(T)(1)\cup\sLog(1-T)=0$ in 
    $\Ext^2_{\Fil\hyphen F\hyphen\MIC(C)}(C_K, C_K(2))$.
\end{lem}

\begin{pf}
    $\sDilog(T)$ has a unique increasing filtration $W_{\bullet}$ (as an object of $\Fil$-$F$-$\MIC(C)$)
    that satisfies
    \[
        W_{j,\dR}=\begin{cases}
            0 & \text{if } j\leq -5,\\
            C_Ke_{-4} & \text{if } j=-4, -3,\\
            C_Ke_{-4}\oplus C_Ke_{-2} & \text{if } j=-2, -1,\\
            \sDilog(T)_{\dR} & \text{if } j\geq 0
        \end{cases}
    \]
and the filtration $\Fil^{\bullet}$ on $W_{j,\dR}$ is given to be
    $\Fil^iW_{j,\dR}=W_{j,\dR}\cap\Fil^i\sDilog(T)_{\dR}$.
    Then, it is straightforward to check that 
    $W_{-4}\cong C_K(2), W_{-2}\cong\sLog(T)(1), W_0/W_{-4}\cong\sLog(1-T)$ and 
    $W_0/W_{-2}\cong C_K$.
    In particular, $\sLog(T)(1)\cup\sLog(1-T)$ is realized by the extension
    $0\to W_{-4}\to W_{-2}\to W_0/W_{-4}\to W_0/W_{-2}\to 0$.
    The commutative diagram
   \[
        \xymatrix{
            0 \ar[r] & C_K(2) \ar@{=}[d]\ar[r]^{\iota} & C_K(2)\oplus W_{-2} \ar[r]^{\pi} 
            \ar[d]^{\mathrm{add}}& \ar[r] W_0 \ar[d]& C_K\ar@{=}[d] \ar[r] & 0\\
            0 \ar[r] & W_{-4} \ar[r] & W_{-2} \ar[r] & W_0/W_{-4} \ar[r] & C_K \ar[r] & 0,
        }
    \]
    where $\iota$ is the first inclusion,
    where $\pi$ is the composition with the second projection and the inclusion $W_{-2}\hookrightarrow W_0$ and where $\mathrm{add}\colon (x,y)\mapsto x+y$, shows that
    $\sLog(T)(1)\cup\sLog(1-T)$.
    The following commutative diagram, where $\pi_2$ is the second projection and where $\pi$ is the natural surjection, shows that this equals zero in the group $\Ext^2_{\Fil\hyphen F\hyphen\MIC(C)}(C_K, C_K(2))$;
    \[
   \xymatrix{
   0 \ar[r] & C_K(2) \ar[r]^{\iota} \ar@{=}[d] & C_K(2)\oplus W_{-2} 
   \ar[r]^{\pi} \ar[d]^{\pi_2}& \ar[r] \ar[d]^{\pi'} W_0 & C_K \ar@{=}[d] \ar[r] & 0\\
        0 \ar[r] & C_K(2) \ar[r]^{\mathrm{id}} & C_K(2) \ar[r]^{0} 
        & C_K \ar[r]^{\mathrm{id}} & C_K \ar[r] & 0
     }   \]
This completes the proof of the vanishing $\sLog(T)(1)\cup\sLog(1-T)=0$.
The proof of ``$\sLog(T)(1)\cup\sLog(-T)=0$'' goes in a smilar way by replacing
$\sDilog(T)$ with $\mathrm{Sym}^2\sLog(T)$.
\end{pf}
We define an object
\[
M_{f,g}(U/S)=(M_{f,g}(U/S)_\dR,M_{f,g}(U/S)_\rig,c,\Phi,\nabla,\Fil^\bullet)
\]
in $\FilFMIC(S)$
such that
\[
    M_{f,g}(U/S)_{\dR}\defeq H^1_{\dR}\left(U_K/S_K, \Cone[\sLog(f)_{\dR}\to\sLog(g)_{\dR}]\right),
\]
\[
    M_{f,g}(U/S)_{\rig}\defeq H^1_{\rig}\left(U_k/S_k, \Cone[\sLog(f)_{\rig}(1)\to\sLog(g)_{\rig}]\right),
\]
and $\nabla$ is the Gauss-Manin connection and
$\Fil^\bullet$ is the Hodge filtration on the de Rham cohomology with coefficients.
Here we 
note that the coefficient ``$\Cone[\cdots]$'' 
can be seen as a complex of mixed Hodge modules by M. Saito
\cite{msaito2}, 
so that one can define the connection and the Hodge filtration in the canonical way.
The comparison isomorphism
\[
    c\colon M_{f,g}(U/S)_{\dR}\otimes_{B_K}{B_K^{\dag}}\xrightarrow{\quad\cong\quad} M_{f,g}(U/S)_{\rig}
\]
is given in the following way.
Firstly, since $\sLog(f)$ and $\sLog(g)$ can be extended to logarithmic objects,
it suffices to construct the comparison isomorphism in the logarithmic situation.
Since the residue matrix of $\Cone[\sLog(f)(1)\to\sLog(g)]$ along every divisor is nilpotent
(because the connection matrix itself is nilpotent),
it satisfies the condition in \cite[1.1]{CaroTsuzuki} and hence we have an isomorphism
between the log.\ relative de Rham cohomology and the log.\ relative rigid cohomology.

\medskip

By the construction, $M_{f,g}(U/S)$ can be fit into the exact sequence
\[
    0\lra H^1(U/S)(2)\lra M_{f,g}(U/S)\lra H^0(U/S)=\cO_{S_K}
\os{\delta}{\lra} H^2(U/S)(2)
\]
in $\FilFMIC(S)$ where 
$\delta_{\dR}$ is the map $1\mapsto\dlog\{f,g\}=\frac{df}{f}\frac{dg}{g}
\in H_\dR^2(U_K/S_K)$.
Here we note that the strictness with respect to $\Fil^\bullet$ follows from
the fact that the above can be seen as an exact sequence of variations of
mixed Hodge structures, again by the theory of Hodge modules
\cite{msaito1}, \cite{msaito2}. 
\begin{prop}\label{prop:from_K2}
        The morphism
    \begin{equation}
        K_2^M(\sO(U))\cap\Ker(\dlog)\longrightarrow \Ext^1_{\Fil\hyphen F\hyphen\MIC(S)}\big(\sO_{S_K}, H^1(U/S)(2)\big);\quad \{f,g\}\mapsto M_{f,g}(U/S)
        \label{eq:fromK2}
    \end{equation}
    is well-defined, where
    $
    \dlog:K^M_2(\O(U))\to H^2_\dR(U/S)
    $
    is the dlog map.
\end{prop}
\begin{pf}
The proof is similar to that of Proposition \ref{2-ext-welldef} on noting that
the strictness with respect to $\Fil^\bullet$ on cohomology with coefficients
follows from the theory of Hodge modules
(details are left to the reader).
    \end{pf}

\begin{lem}\label{lem:tame}
    The following diagram is commutative:
    \[
 \xymatrix{K^M_2(\sO(U))\cap\Ker(\dlog)
  \ar[r]^{\hspace{35pt}\mathrm{tame}\ \mathrm{symbol}} 
  \ar[d]_{\eqref{eq:fromK2}} 
  & \sO(D)^{\times}\ar[d] \\
            \Ext^1(\sO_{S_K}, H^1(U/S)(2)) \ar[r]^{\mathrm{Res}} & 
            \Ext^1(\O_{S_K}, H^0(D/S)(1)).
}    \]
Here, the tame symbol is defined by
    \[
        \{f, g\}\mapsto (-1)^{\ord_D(f)\ord_D(g)}\frac{f^{\ord_D}(g)}{g^{\ord_D}(f)}\bigg|_D
    \]
    and the vertical arrow on the right sends $h\in\sO(D)^{\times}$ to $\sLog(h)$.
    \label{lem:tame_symbol}
\end{lem}

\begin{pf}
We first note that $K_2^M(\O(U))$ is generated by
symbols $\{f,g\}$ with $\ord_D(f)=0$.
Suppose $\ord_D(f)=0$, namely $f\in \O(U\cup D)^\times$.
Let $p$ be the projection $\Cone\left[\sLog f(1)\to\sLog (g)\right]\to \sLog f(1)$.
Then $p$ and the residue map induce the following commutative diagram 
\begin{equation}
\xymatrix{           
 0 \ar[r] & H^1(U/S)(2) \ar[r] \ar@{=}[d] 
 & M_{f,g}(U/S) \ar[d]^{p} \ar[r] & \O_{S_K} \ar[d]^{p'} \ar[r]^{\delta_{f,g}\qquad} 
 & H^2(U/S)(2)\ar@{=}[d]\\
 \O_S \ar[r]^{\delta_f\qquad} & H^1(U/S)(2)\ar[d]^{\Res} 
            \ar[r] & H^1\left(\sLog f(1)\right) \ar[d]^{\Res} \ar[r] 
            & H^1(U/S)(1) \ar[d]^{\Res} \ar[r] & H^2(U/S)(2)\ar[d]^\Res\\
           0 \ar[r] & H^0(D/S)(1) \ar[r] 
           & H^0\left(\sLog f|D\right) \ar[r] & H^0(D/S) \ar[r]^{\delta_{f|D}} & H^1(D/S)(1)
}       \label{eq:comm_diagram}
\end{equation}
with exact rows where 
$\delta_{f,g}(1)=\frac{df}{f}\frac{dg}{g}$,
$\delta_f(1)=\frac{df}{f}$ and 
$\delta_{f|D}(1)=\frac{df}{f}|_D$.
    The maps ``Res''s are residue maps appearing in the Gysin exact sequence;
    as for the middle one, we are considering the Gysin exact sequence with coefficient
    in $\sLog(f)(1)$ which is constructed in the same manner as Proposition \ref{prop:Gysin}.
Moreover $p'$ coincides with the map $\delta_g:1\mapsto\frac{dg}{g}$. To see this,
consider a commutative diagram
    \[
 \xymatrix{          \left[\sLog(f)_{\dR}\to\sLog(g)_{\dR}\right] \ar[r] \ar[d]^p & \left[\sLog(f)_{\dR}/\sO_{U_K}e_{-2,f}\to\sLog(g)_{\dR}\right] 
 \ar[d]^{p'} 
 \ar[r]^{\hspace{2cm}\cong} & \left[ 0\to \sO_{U_K}\right] \\
            \left[\sLog(f)_{\dR}\to 0\right] \ar[r] 
            & \left[\sLog(f)_{\dR}/\sO_{U_K}e_{-2,f}\to 0\right] 
            \ar[r]^{\hspace{1cm}\cong} & \left[ \sO_{U_K}\to 0\right],
 }   \]
    where the two horizontal morphisms on the left-hand side are the canonical surjections and the vertical ones are the projections.
    Via the identification $\sLog(f)_{\dR}/\sO_{U_K}e_{-2,f}\cong\sO_{U_K}$ (which sends $e_{0,f}$ to $1$), the homomorphism $p'$
    is nothing but the connecting morphism arising from the extension $0\to\sO_{U_K}\to \sLog(g)_{\dR}\to \sO_{U_K}\to 0$.

Since the composition $\Res\circ p'$ is multiplication by $\ord_D(g)$, 
the diagram \eqref{eq:comm_diagram}
induces
\begin{equation}\label{eq:comm_diagram2}
\xymatrix{           
 0 \ar[r] & H^1(U/S)(2) \ar[r] \ar[d]^\Res 
 & M_{f,g}(U/S) \ar[d] \ar[r] & \O_{S_K} \ar[d]^{\ord_D(g)} \ar[r]^{\delta_{f,g}\qquad} 
 & H^2(U/S)(2)\ar[d]^\Res\\
           0 \ar[r] & H^0(D/S)(1) \ar[r] 
           & H^0\left(\sLog f|D\right)' \ar[r] & \O_{S_K} \ar[r] & H^1(D/S)(1).
}\end{equation}
Now we show the lemma.
Let $\xi\in K^M_2(\O(U))\cap\Ker(\dlog)$ be arbitrary.
Fix $t\in \O(U)^\times$ such that $\ord_D(t)>0$. Replacing $\xi$ with $m\xi$ for some $m>0$, one can express 
\[
\xi=\{f,t\}+\sum_j\{u_j,v_j\}
\]
with $\ord_D(f)=\ord_D(u_j)=\ord_D(v_j)=0$.
Then \eqref{eq:comm_diagram2} induces a commutative diagram
\begin{equation}\label{eq:comm_diagram3}
\xymatrix{           
 0 \ar[r] & H^1(U/S)(2) \ar[r] \ar[d]^\Res 
 & M_\xi(U/S) \ar[d] \ar[r] & \O_{S_K} \ar[d]^{\ord_D(t)} \ar[r]
 & 0\\
           0 \ar[r] & H^0(D/S)(1) \ar[r] 
           & H^0\left(\sLog f|D\right)' \ar[r] & \O_{S_K} \ar[r] & 0.
}\end{equation}
Since the tame symbol of $\xi$ is $f|_D$, the assertion follows.
\end{pf}
Let $C$ be the kernel of the connection morphism $H^0(D/S)(-1)\to H^2(X/S)$
in the Gysin exact sequence.
Since $H^0_{\rig}(D/S)$ is an overconvergent $F$-isocrystal on $S$ of weight $0$,
the weight of $C_{\rig}$ is $-2$ and thus
we have $\Hom_{\Fil\hyphen F\hyphen\MIC(S)}\big(\sO_{S_K}, C(2)\big)=0$.
This shows that the Gysin exact sequence induces an exact sequence
    \[
        0\longrightarrow \Ext^1\left(\sO_{S_K}, H^1(X/S)(2)\right)\xrightarrow{\alpha}\Ext^1\left(\sO_{S_K}, H^1(U/S)(2)\right)\xrightarrow{\beta}\Ext^1\big(\sO_{S_K}, C(2)\big),
    \]
    in which each $\Ext^1$ means $\Ext^1_{\Fil\hyphen F\hyphen\MIC(S)}$.  
Notice that $\dlog:K_2(X)\to H^2_\dR(X/S)$ is zero where $K_i=K_i^Q$ denotes
Quillen's $K$-theory.
Therefore, Proposition \ref{prop:from_K2} and Lemma \ref{lem:tame} immediately
imply the following.
\begin{prop}\label{prop:from_K2X}
The map \eqref{eq:fromK2} uniquely extends to
\[
K_2(X)\ot\Q\lra \Ext^1\left(\sO_{S_K}, H^1(X/S)(2)\right).
\]
\end{prop}

\section{Syntomic regulators via hypergeometric functions}\label{main-sect}
For $a,b\in \Z_p$, we denote by
\[
{}_2F_1(a,b,1;\l):=\sum_{i=0}^\infty \frac{(a)_i(b)_i}{i!^2}\l^i\in \Z_p[[\l]],\quad
(x)_i:=x(x+1)\cdots(x+i-1)
\]
the {\it hypergeometric function}, regarded as a formal power series with
coefficients in $\Z_p$.

Let $A$ be a noetherian $\Z_p$-algebra complete with respect to the ideal $pA$, and
let $\sigma:A\to A$ be a $p$-th Frobenius (i.e. a lifting of the $p$-power map on $A/pA$).
We define
\begin{equation}\label{main-sect-eq2}
\log^{(\sigma)}(\alpha):=p^{-1}\log\left(\frac{\alpha^p}{\alpha^\sigma}\right)\in A
\end{equation}
for $\alpha\in A^\times$.
When $A=W(\ol\F_p)$ is the Witt ring and $\sigma$ is the $p$-th Frobenius, 
we also write it by $\log^{(p)}(\alpha)$.
\subsection{Hypergeometric fibration $y^N=x(1-x)^{N-1}(1-\l x)$ }\label{main-sect-1}
Let $N\geq 2$ be an integer.
Let $f:Y_\Q\to\P^1_\Q$ be a surjective morphism of projective smooth varieties
over $\Q$, such that a smooth general fiber $f^{-1}(\l)$ is a nonsingular
projective model of an affine curve \[y^N=x(1-x)^{N-1}(1-\l x).\]
Let $p\geq 3$ be an odd prime number prime to $N(N-1)$.
We shall give an integral model over $\Z_p$ in \S \ref{main-sect-3}, which we denote by
$f:Y\to \P^1_{\Z_p}$.
Put $S:=\Spec\Z_p[\l,(\l-\l^2)^{-1}]\hra \P^1_{\Z_p}$ and
$X:=f^{-1}(S)$. Then $f$ is smooth over $S$.
Moreover 
$f^{-1}(0)$, $f^{-1}(1)$ and $f^{-1}(\infty)$ are relative NCD's over $\Z_p$.

We denote by $\mu_m(F)\subset F^\times$
the group of all $m$-th roots of unity contained in a field $F$.
In what follows we assume $p\equiv 1$ mod $N$.
Then $\Z_p^\times$ contains a primitive $N$-th root of unity.
Define an action of $\zeta\in \mu_N:=\mu_N(\Q_p)$ on $Y$ by
\begin{equation}\label{main-sect-1-eq1}
[\zeta]\cdot(x,y,\l):=(x,\zeta^{-1} y,\l).
\end{equation}
By this action $H^\bullet_\rig(X_{\F_p}/S_{\F_p})$, $H^\bullet_\dR(X/S)$
are $\Z_p[\mu_N]$-modules.
For a $\Z_p[\mu_N]$-module $M$ and an integer $n$  
we denote by
$M^{(n)}$ the subspace on which each $\zeta\in \mu_N$
acts by multiplication by $\zeta^n$.
It follows from \cite{a-o-2} Prop. 3.1 that
there is a decomposition 
\begin{equation}\label{main-sect-1-eq2}
H^1_\dR(X_{\Q_p}/S_{\Q_p})=\bigoplus_{n=1}^{N-1}
H^1_\dR(X_{\Q_p}/S_{\Q_p})^{(n)}
\end{equation}
and each $H^1_\dR(X_{\Q_p}/S_{\Q_p})^{(n)}$ is free of rank $2$ over $\O(S_{\Q_p})$
with a free basis
\begin{equation}\label{main-sect-1-eq3}
\omega_n:=\frac{(1-x)^{n-1}dx}{y^n},\quad
\eta_n:=\frac{x(1-x)^{n-1}dx}{(1-\l x)y^n}\in H^1_\dR(X_{\Q_p}/S_{\Q_p})^{(n)}.
\end{equation}
For the later use, we introduce another basis
\begin{equation}\label{main-thm-wt}
\wt\omega_n:=F_n(\l)^{-1}\omega_n,\quad
\wt\eta_n:=(\l-\l^2)\left(F'_n(\l)\omega_n-\frac{n}{N}F_n(\l)\eta_n\right)
\end{equation}
of $(\Q_p\ot_{\Z_p}\Z_p((\l)))\ot_{\O(S)} H^1_\dR(X/S)^{(n)}$ where $\Z_p((\l)):=\Z_p[[\l]][\l^{-1}]$ and
\begin{equation}\label{main-sect-1-eq4}
F_n(\l):={}_2F_1\left(\frac{n}{N},1-\frac{n}{N},1;\l\right)\in \Z_p[[\l]],\quad (1\leq n\leq N-1).
\end{equation}

\subsection{Main Theorem}\label{main-sect-2}

For a smooth projective scheme $V$ over a complete discrete valuation ring
with a perfect reside field,
we denote by $H^\bullet_\syn(V,\Q_p(j))$ the {\it syntomic cohomology groups}
of Fontaine-Messing (cf. \cite{Ka2}).
We also use the {\it rigid syntomic cohomology groups} $H^\bullet_{\rig\text{-}\syn}(V,\Q_p(j))$
by Besser \cite{Be1}. They agree for projective smooth varieties.
However, for non-complete
varieties, the former does not give the ``right'' cohomology theory, while so does
the latter.

\medskip

Let $f:X\to S$ be as in \S \ref{main-sect-1}.
Let $W=W(\ol\F_p)$ denotes the Witt ring of the algebraic closure of $\F_p$, and
$K:=\Frac(W)$.
For $\alpha\in W\setminus\{0,1\}$,
we denote by
$X_\alpha=f^{-1}(\alpha)$ 
the fiber over $\Spec W[\l]/(\l-\alpha)\hra S\times_{\Z_p}\Spec W$.
Our main theorem gives an explicit description of the {\it syntomic regulator map}
\[
\reg_\syn:K_2(X_\alpha)\lra H^2_\syn(X_\alpha,\Q_p(2))\cong H^1_\dR(X_\alpha/K),
\]
on Quillen's $K_2$.

\medskip

To give the precise statement, we prepare notation.
Let $\wt\tau_n(\l)\in\Q_p[[\l]]$ be characterized by
\begin{equation}\label{main-sect-1-eq5}
\frac{d}{d\l}\wt\tau_n(\l)=\frac{1}{\l}\left(1-\frac{1}{(1-\l)F_n(\l)^2}\right),\quad 
\wt\tau_n(0)=0.
\end{equation}
Put
\begin{align}
\kappa^{(p)}_n:&=2\log^{(p)}N-\frac{1}{2}\sum_{\ve\in \mu_N(\Q_p)\setminus\{1\}}
(\ve^n+\ve^{-n})\log^{(p)}(1-\ve)(1-\ve^{-1})
\label{main-sect-1-eq6}\\
&=2\sum_{\ve\in \mu_N(\Q_p)\setminus\{1\}}
(1-\ve^{-n})\log^{(p)}(1-\ve),
\end{align}
Let $\sigma:W[[\l]]\to W[[\l]]$ be a $p$-th Frobenius
such that $\l^\sigma=a \l^p$ for some $a\in 1+pW$ and
$\sigma|_W$ agrees with the Frobenius on $W$.
Put 
\begin{equation}\label{main-sect-1-eq7}
\tau^{(\sigma)}_n(\l):=-\log^{(\sigma)}(\l)+
\kappa^{(p)}_n+\wt\tau_n(\l)-\frac{1}{p}\wt\tau_n(\l^\sigma)
\end{equation}
where 
$\log^{(\sigma)}(z)$ is the function defined in \eqref{main-sect-eq2}.
Let $E_i^{(n)}(\l)\in\Q_p[[\l]]$ be characterized by
\begin{equation}\label{main-diffeq-1-p}
\frac{d}{d\l}E_1^{(n)}(\l)=
\frac{F_n(\l)}{1-\l}-(-1)^{\frac{(p-1)n}{N}}p^{-1}
\frac{F_n(\l^\sigma)}{1-\l^\sigma}\frac{d\l^\sigma}{d\l}
\end{equation}
\begin{equation}\label{main-diffeq-2-p}
\frac{d}{d\l}E_2^{(n)}(\l)=
    \frac{E_1^{(n)}(\l)F_n(\l)^{-2}}{\l-\l^2}+
(-1)^{\frac{(p-1)n}{N}}p^{-1}\tau^{(\sigma)}_n(\l)
\frac{F_n(\l^\sigma)}{1-\l^\sigma}\frac{d\l^\sigma}{d\l}
\end{equation}
with initial conditions
\begin{equation}\label{main-diffeq-3-p}
E^{(n)}_1(0)=0,\quad E^{(n)}_2(0)=2N\sum_{\nu^N=-1}
\nu^{-n}\mathrm{ln}^{(p)}_2(\nu).
\end{equation}
Let 
$\ve^{(n)}_i(\l)$ be defined by
$\ve^{(n)}_1(\l)\omega_n+\ve^{(n)}_2(\l)\eta_n=
E^{(n)}_1(\l)\wt\omega_n+E^{(n)}_2(\l)\wt\eta_n$ by
the basis \eqref{main-sect-1-eq3} and
\eqref{main-thm-wt}.
In a down-to-earth manner, they are given as follows
\begin{align}
\ve^{(n)}_1(\l)&:=
F_n(\l)^{-1}E^{(n)}_1(\l)+(\l-\l^2)F'_n(\l)E^{(n)}_2(\l),\label{main-eq3}\\
\ve^{(n)}_2(\l)&:=
-\frac{n}{N}(\l-\l^2)F_n(\l)E^{(n)}_2(\l).\label{main-eq4}
\end{align}
The main theorem of this paper is the following.

\begin{thm}\label{main-thm}
Let $p\geq 3$ be a prime such that $p\equiv 1$ mod $N$.
Let $W=W(\ol\F_p)$ and $K=\Frac W$.
    Fix a $p$-th Frobenius $\sigma$ on $W[\l,(1-\l)^{-1}]^\dag$
    compatible with the Frobenius on $W$ and such that $\sigma(\l)=a\l^p$ for some 
    $a\in 1+pW$.
Let $\zeta_1,\zeta_2\in\mu_N=\mu_N(\Q_p)$ be arbitrary such that $\zeta_1\ne\zeta_2$. 
Let
\begin{equation}\label{main-thm-eqfg}
h_1:=\frac{y-\zeta_1(1-x)}{y-\zeta_2(1-x)},\quad h_2:=\frac{(1-\l)x^2}{(1-x)^2}
\end{equation}
and let $\{h_1,h_2\}\in K_2(X)$ be a $K_2$-symbol in Quillen's $K_2$.
    Then \[\ve^{(n)}_i(\l)\in K[\l,(1-\l)^{-1}]^\dag\] 
are overconvergent for each $i,n$, and the syntomic regulator is described as
\begin{equation}\label{main-thm-eq1}
\reg_\syn
(\{h_1,h_2\}|_{\l=\alpha})=\sum_{n=1}^{N-1}\frac{\zeta^n_1-\zeta^n_2}{N}
(\ve^{(n)}_1(\alpha)\omega_n+\ve^{(n)}_2(\alpha)\eta_n)\in H^1_\dR(X_\alpha/K)
\end{equation}
for all $\alpha\in W^\times$ satisfying $\alpha\not\equiv0, 1$ mod $p$
and $\sigma(\l)|_{\l=\alpha}=\alpha^\sigma$ (namely $a\alpha^p=\alpha^\sigma$).
\end{thm}
\begin{rem}
By the proof below (especially 
\S \ref{mainpf2-1-sect} \eqref{mainpf2-1-eq3}), one further sees that
the coefficients of $E_i^{(n)}(\l)$ belong to $W$ and hence
$\ve^{(n)}_i(\l)\in W[\l,(1-\l)^{-1}]^\dag$.
\end{rem}

\subsection{Integral model $Y$ over $\Z_p$}\label{main-sect-3}
Let $Y_\Q\to \P^1_\Q$ be the hypergeometric fibration over $\Q$ in \S \ref{main-sect-1}.
In this section, for $p\geq 3$ a prime number prime to $N(N-1)$
we construct an integral model $f:Y\to \P^1_{\Z_p}$ over $\Z_p$.
Moreover we shall show that $f$ is smooth over
$S=\Spec\Z_p[\l,(\l-\l^2)^{-1}]$, and the divisors
$f^{-1}(0)$, $f^{-1}(1)$ and $f^{-1}(\infty)$ are relative NCD's over $\Z_p$.

\medskip

Put
\[
U':=\Spec \Z_p[x,y,\l]/(y^N-x(1-x)^{N-1}(1-\l x))
\]
\[
U^{\prime\prime}:=\Spec \Z_p[z,w,\l]/(w^N-z^{N-1}(z-1)^{N-1}(z-\l))
\]
and glue them by $z=x^{-1}$ and $w=yx^{-2}$.
Let $U_1\to U'\cup U^{\prime\prime}$ 
be the blowing-up along the closed set $(x,y)=(1,0)$ $\Leftrightarrow$
$(z,w)=(1,0)$, and
$U'_1$ and $U^{\prime\prime}_1$
the inverse images of $U'$ and $U^{\prime\prime}$.
Then $U'_1$ is smooth over $\Z_p$ while $U^{\prime\prime}_1$ has 
a singular locus $(z,w)=(0,0)$.
Let $U_2\to U_1$ be the blowing up along the locus $(z,w)=(0,0)$.
Then $U_2$ is smooth over $\Z_p$ and we have a projective flat morphism
\[
f_0:U_2\lra \Spec \Z_p[\l].
\]
The fibers $f_0^{-1}(0)$ and
$f_0^{-1}(1)$ are relative NCD's over $\Z_p$, and all components are rational curves,
i.e. $f_0$ has {\it totally degenerate semistable reductions} at $\l=0,1$.

For the later use we write down the singular fiber 
\begin{equation}\label{int-locus1}
Z:=f_0^{-1}(0)\subset U_2.
\end{equation}
The inverse image of $U^{\prime\prime}\cap\{z\ne1\}$ is covered by two affine open sets
$W_1$ and $W_2$
\begin{equation}\label{int-locus2}
W_1=\Spec \Z_p[z_1,w,\l,(z_1w-1)^{-1}]/(w-z_1^{N-1}(z_1w-1)^{N-1}(z_1w-\l))\subset U_2,
\end{equation}
\begin{equation}\label{int-locus3}
W_2=\Spec \Z_p[z,w_1,\l,(z-1)^{-1}]/(zw_1^N-(z-1)^{N-1}(z-\l))\subset U_2
\end{equation}
where $z_1=w_1^{-1}=zw^{-1}$.
Hence $Z$ has two irreducible components $Z_1$ and $Z_2$
such that $Z_1\cap W_2=\{z=\l=0\}$ and  $Z_2\cap W_2=\{w_1^N-(z-1)^{N-1}=\l=0\}$.
They are isomorphic to $\P^1_{\Z_p}$ and 
intersect transversally. The intersection locus $Z_1\cap Z_2\subset W_2$ is given by 
\begin{equation}\label{int-locus4}
z=(-w_1)^N+1=\l=0.
\end{equation}

Next we construct a compactification of $U_2\to\Spec A$ over $\l=\infty$.
Let $B=\Z_p[t]$ and
\[
V':=\Spec B[x,y^*,t]/(y^{*N}-t^{N-1}x(1-x)^{N-1}(t-x))
\]
\[
V^{\prime\prime}:=\Spec B[z,w^*,t]/(w^{*N}-t^{N-1}z^{N-1}(z-1)^{N-1}(tz-1))
\]
be glued with $U'\cup U^{\prime\prime}$ by $t=\l^{-1}$, $y^*=ty$ and $w^*=tw$.
For simplicity of notation we write $y=y^*$ and $w=w^*$.
Let $V_1\to V'\cup V^{\prime\prime}$ be the blowing-up along $(x,y)=(1,0)$, and 
$V'_1$ and $V^{\prime\prime}_1$ the inverse images of $V'$ and $V^{\prime\prime}$.

We first resolve the singularities of $V'_1$.
The singular locus of $V'_1$ is contained in an affine open set
\[
P':=\Spec B[x,y_1,t]/((1-x)y_1^N-t^{N-1}x(t-x)),\quad y_1(1-x)=y
\]
and it is given as $(y_1,t)=(0,0)$.
Let $V'_2\to V'_1$ be the blowing-up along $(y_1,t)=(0,0)$.
Then the inverse image of $P'$ is covered by two affine open sets
\[
P_1=\Spec B[x,y_2,t]/((1-x)ty_2^N-x(t-x)),\quad ty_2=y_1,
\]
\[
P_2=\Spec B[x,y_1,t_1]/((1-x)y_1-t_1^{N-1}x(t-x)),\quad t_1y_1=t.
\]
Then 
$P_1$ has a singular point $(x,y_2,t)=(0,0,0)$ and
$P_2$ has a singular point $(x,y_1,t_1)=(1,0,0)$.
The latter is a rational singularity of type $A_{N-1}$, and there is a resolution
such that the reduced part of the fiber over $t=0$ is a relative NCD over $\Z_p$. 
The former is a singular point of type $ty^N=x(t-x)$.
If we take the blowing up along $(x,t)=(0,0)$, then
three rational singular points of type $A_N$ appear, and hence we can again
resolve them. 
We thus have a resolution $V'_2\to V'_1$.

Next we resolve the singularities of $V_1^{\prime\prime}$.
The singular loci of $V_1^{\prime\prime}$ are contained in an affine open set
\[
Q^{\prime\prime}:=\Spec B[z,w_1,t]/((z-1)w_1^N-t^{N-1}z^{N-1}(tz-1)),\quad w_1(z-1)=w,
\]
and they are given as $(w_1,t)=(0,0)$ or $(z,w_1)=(0,0)$.
Let $V_2^{\prime\prime}\to V_1^{\prime\prime}$ be the blowing-up along $(z,w_1)=(0,0)$.
The inverse image of $Q^{\prime\prime}$ is covered by
\[
Q_1=\Spec \Z_p[z,w_2,t]/(z(z-1)w_2^N-t^{N-1}(tz-1)),\quad zw_2=w_1,
\]
\[
Q_2=\Spec \Z_p[z_1,w_1,t]/((z_1w_1-1)w_1-t^{N-1}z_1^{N-1}(tz_1w_1-1)),\quad z_1w_1=z.
\]
$Q_2$ is smooth over $\Z_p$ while $Q_1$ has a singular locus $(w,t)=(0,0)$.
Again take a blowing up $Q^*_1\to Q_1$ along the locus $(w_2,t)=(0,0)$.
It is covered by
\[
Q_3=\Spec \Z_p[z,w_2,t_1]/(z(z-1)w_2-t_1^{N-1}(t_1w_2z-1)),\quad t_1w_2=t,
\]
\[
Q_4=\Spec \Z_p[z,w_3,t]/(z(z-1)tw_3^N-(tz-1)),\quad tw_3=w_2.
\]
$Q_4$ is smooth over $\Z_p$ while $Q_3$ has two singular points $(z,w_2,t_1)=(0,0,0),
(1,0,0)$.
Both are rational singularities of type $A_N$.
There is a resolution
such that the reduced part of the fiber over $t=0$ is a relative NCD over $\Z_p$. 
We thus have a resolution $V^{\prime\prime}_2\to V^{\prime\prime}_1$.

Summing up the above we have
a projective flat morphism
\[
f_\infty:V_2=V'_2\cup V^{\prime\prime}_2\lra \Spec \Z_p[t]
\]
which is smooth over $\Spec\Z_p[t,t^{-1}]$ and
$V_2$ is smooth over $\Z_p$.
One can check that the redcued part of the fiber over $t=0$ ($\Leftrightarrow$ $\l=\infty$)
is a relative NCD over $\Z_p$.
We thus have the desired integral model
\[
f:Y:=U_2\cup V_2\lra \P^1_{\Z_p}.
\]

\section{Periods of Jacobian variety $J(X/S)$}\label{gm-sect}
\subsection{De Rham symplectic basis for totally degenerating abelian varieties}
\label{gm-mt-sect}
Let $R$ be a regular noetherian domain, and $I$
a reduced ideal of $R$. Let $L:=\Frac(R)$ be the fractional field.
Let $J/R$ be a $g$-dimensional 
commutative group scheme such that
the generic fiber $J_\eta$ is a principally polarized abelian variety over $L$.
If the fiber $T$ over $\Spec R/I$ is an algebraic torus,
we call $J$ a {\it totally degenerating abelian schemes over $(R,I)$}
(cf. \cite{FC} Chapter II, 4).
Assume that the algebraic torus $T$ is split.
Assume further that $R$ is complete with respect to $I$.
Then there is the uniformization $\rho:\bG_m^g\to J$ in the rigid analytic sense.
We fix $\rho$ and the coordinates $(u_1,\ldots,u_g)$ of $\bG_m^g$.
Then a matrix
\begin{equation}\label{gm-sect-3-eq1}
\underline{q}=\begin{pmatrix}
q_{11}&\cdots&q_{1g}\\
\vdots&&\vdots\\
q_{g1}&\cdots&q_{gg}
\end{pmatrix},\quad q_{ij}=q_{ji}\in L
\end{equation}
of multiplicative periods is determined up to $\mathrm{GL}_g(\Z)$, and this 
yields an isomorphism
\[
J\cong \bG_m^g/\underline{q}^\Z
\]
of abelian schemes over $R$ where $\bG_m^g/\underline{q}^\Z$ denotes 
Mumford's construction of the quotient scheme (\cite{FC} Chapter III, 4.4).

\medskip

In what follows, we suppose that the characteristic of $L$ is zero.
The morphism $\rho$ induces 
\[
\rho^*:\Omega^1_{J/R}\lra \bigoplus_{i=1}^g\wh\Omega^1_{\bG_m,i},\quad
\wh\Omega^1_{\bG_m,i}:=\varprojlim_n\Omega^1_{R/I^n[u_i,u_i^{-1}]/R}.
\]
Let 
\[
\Res_i:\wh\Omega^1_{\bG_m,i}\lra R,\quad 
\Res_i\left(\sum_{m\in\Z} a_mu^m_i\frac{du_i}{u_i}\right)=a_0
\]
be the residue map. The composition of $\rho^*$ 
and the residue map induces 
a morphism
$\Omega^\bullet_{J/R}\lra R^g[-1]$
of complexes, and hence a map \[
\tau:H^1_\dR(J/R):=H^1_\zar(J,\Omega^\bullet_{J/R})\lra R^g.
\]
Let $U$ be defined by
\[
0\lra U\lra H^1_\dR(J_\eta/L)\os{\tau}{\lra} L^g\lra0.
\]
Note that the composition
$\vg(J_\eta,\Omega^1_{J_\eta/L})\os{\subset}{\to} H^1_\dR(J_\eta/L)\os{\tau}{\to} L^g
$
is bijective.
Let $\langle x,y\rangle$ denotes the symplectic pairing on $H^1_\dR(J_\eta/L)$
with respect to the principal polarization on $J_\eta$.
We call an $L$-basis
\[
\wh\omega_i,\, \wh\eta_j\in H^1_\dR(J_\eta/L),\quad 1\leq i,\,j\leq g
\]
a {\it de Rham symplectic basis} if the following conditions are satisfied. 
\begin{description}
\item[(DS1)]
$\wh\omega_i\in \vg(J_\eta,\Omega^1_{J_\eta/L})$ and
$\tau(\wh\omega_i)\in (0,\ldots,1,\ldots,0)$ where ``$1$'' is placed in the $i$-th component.
Equivalently, \[
\rho^*(\wh\omega_i)=\frac{du_i}{u_i}.
\]
\item[(DS2)]
$\wh\eta_j\in U$ and
$\langle\wh\omega_i ,\wh\eta_j\rangle=\delta_{ij}$ where
$\delta_{ij}$ is the Kronecker delta.
\end{description}
If we fix the coordinates $(u_1,\ldots,u_g)$ of $\bG_m^g$, then 
$\wh\omega_i$ are uniquely determined by {\bf(DS1)}.
Since the symplectic pairing $\langle x,y\rangle$ is annihilated on $U\ot U$ and
$\vg(J_\eta,\Omega^1_{J_\eta/L})\ot \vg(J_\eta,\Omega^1_{J_\eta/L})$,
the basis $\wh\eta_j$ are uniquely determined as well by {\bf(DS2)}.
\begin{prop}\label{gm-mt-sect-prop1}
Let $V$ be a subring of $R$.
Suppose that $(R,I)$ and $V$ satisfy the following.
\begin{description}
\item[(C)]
There is a regular integral noetherian 
$\C$-algebra $\wt R$ complete with respect to a reduced ideal $\wt I$ and an
injective homomorphism $i:R\to\wt R$ such that $i(V)\subset \C$ and
$i(I)\subset \wt I$ and 
\[
\wh\Omega^1_{L/V}\lra \wh\Omega^1_{\wt L/\C}
\]
is injective where we put $\wt L:=\Frac(\wt R)$ and 
\[
\wh\Omega^1_{L/V}:=L\ot_R\left(\varprojlim_n \Omega^1_{R_n/V}\right),\quad \wh\Omega^1_{\wt L/\C}:=\wt L\ot_{\wt R}\left(\varprojlim_n \Omega^1_{\wt R_n/\C}\right),\quad R_n:=R/I^n,\, \wt R_n:=\wt R/\wt I^n.
\]
\end{description}
Let
\[
\nabla:H^1_\dR(J_\eta/L)\lra \wh\Omega^1_{L/V}\ot_L
H^1_\dR(J_\eta/L)
\]
be the Gauss-Manin connection. Then
\begin{equation}\label{gm-mt-sect-eq3}
\nabla(\wh\omega_i)=\sum_{j=1}^g\frac{dq_{ij}}{q_{ij}}\ot \wh\eta_j,
\quad \nabla(\wh\eta_i)=0.
\end{equation}
\end{prop}
\begin{pf}
By the assumption {\bf(C)}, we may replace $J_\eta/L$ with $(J_\eta\ot_L\wt L)/\wt L$.
We may assume $R=\wt R$, $V=\C$ and $J=J_\eta\ot_R\wt R$.
There is a smooth scheme $J_S/S$ over a
connected smooth affine variety $S=\Spec A$ over $\C$
with a Cartesian square
\[
\xymatrix{
J\ar[d]\ar[r]\ar@{}[rd]|{\square}&J_S\ar[d]\\
\Spec R\ar[r]&S
}
\]
such that $\Spec R\to S$ is dominant.
Let $D\subset S$ be a closed subset such that $J_S$ is proper over $U:=S\setminus D$.
Then the image of $\Spec R/I$ is contained in $D$ since
$J$ has a totally degeneration over $\Spec R/I$. 
Thus we may replace $R$ with the completion $\wh A_D$ of $A$ by the ideal
of $D$. Let $L_D=\Frac \wh A_D$ then
$\wh \Omega^1_{L_D/\C}\cong L_D\ot_A\Omega^1_{A/\C}$.
Let ${\frak m}\subset A$ be a maximal ideal containing $I$, and $\wh A_{\frak m}$ the completion
by $\frak m$. Then $\wh A_D\subset \wh A_{\frak m}$ and $\wh\Omega_{\wh A_D/\C}
\subset \wh\Omega_{\wh A_{\frak m}/\C}$.
Therefore we may further replace $\wh A_D$ with $\wh A_{\frak m}$.
Summing up the above, it is enough work in the following situation.
\[
R=\wh A_{\frak m}\cong \C[[x_1,\ldots, x_n]]\supset I=(x_1,\ldots,x_n),
\quad V=\C,\quad L=\Frac R
\]
\[
J=\Spec R\times_AJ_S\lra \Spec R.
\]
Note $\wh \Omega^1_{L/\C}\cong L\ot_A\Omega^1_{A/\C}$.
Let $h:\cJ\to S^{\an}$ be the analytic fibration associated to $J_S/S$.
Write $J_\l=h^{-1}(\l)$ a smooth fiber over $\l\in U^{\an}:=(S\setminus D)^{\an}$.
Let
\[
\nabla:\O_{U^{\an}}\ot R^1h_*\C\lra 
\Omega^1_{U^{\an}}\ot R^1h_*\C.
\] 
be the flat connection compatible with
the Gauss-Manin connection on 
$H^1_\dR(J_S/S)|_U$ under the comparison
\[
\O_{U^{\an}}\ot R^1h_*\C\cong \O_{U^{\an}}\ot_AH^1_\dR(J_S/S).
\]
We describe a de Rham symplectic basis in $\wh\omega^{\an}_i,\wh\eta^{\an}_j
\in \O_{U^{\an}}\ot_AH^1_\dR(J_S/S)$
and prove \eqref{gm-mt-sect-eq3} for the above flat connection.
Write $J_\l=(\C^\times)^g/\underline{q}^\Z$ for $\l\in U^{\an}$.
Let $(u_1,\ldots,u_g)$ denotes the coordinates of $(\C^\times)^g$.
Let $\delta_i\in H_1(J_\l,\Z)$ be the homology cycle defined by the circle 
$|u_i|=\ve$ with $0<\ve\ll1$.
Let $\gamma_j\in H_1(J_\l,\Z)$ be the homology cycle defined by
the path from $(1,\ldots,1)$ to $(q_{j1},\ldots,q_{jg})$.
As is well-known, the dual basis $\delta_i^*,\gamma^*_j\in H^1(J_\l,\Z)$
is a symplectic basis, namely
\[
\langle\delta^*_i,\delta^*_{i'}\rangle=
\langle\gamma^*_j,\gamma^*_{j'}\rangle=0,\quad
\langle\delta^*_i,\gamma^*_j\rangle=\frac{1}{2\pi\sqrt{-1}}\delta_{ij}
\]
where $\delta_{ij}$ denotes the Kronecker delta.
We have $\wh\omega^{\an}_i=du_i/u_i$ by {\bf (DS1)},  and then 
\begin{equation}\label{gm-mt-sect-eq5}
\wh\omega_i^{\an}=2\pi\sqrt{-1}\delta_i^*+\sum_{j=1}^g\log q_{ij}\gamma^*_j.
\end{equation}
Let $\tau^B:R^1h_*\Z\to \Z(-1)^g$ be the associated map to $\tau$.
An alternative description of $\tau^B$ is
\[\tau^B(x)=\frac{1}{2\pi\sqrt{-1}}((x,\delta_1),\ldots,(x,\delta_g))\]
where $(x,\delta)$ denotes the natural pairing
on $H^1(J_\l,\Z)\ot H_1(J_\l,\Z)$.
Obviously $\tau^B(\gamma_j^*)=0$.
This implies that
$\wh\eta_j$ is a linear combination
of $\gamma_1^*,\ldots,\gamma_g^*$ by {\bf(DS2)}.
Since $\langle \wh\omega^{\an}_i,\wh\eta_j\rangle=\delta_{ij}=\langle \wh\omega^{\an}_i,
\gamma^*_j\rangle$,
one concludes 
\[
\wh\eta_j^{\an}=\gamma^*_j.
\]
Now \eqref{gm-mt-sect-eq3} is immediate from this and \eqref{gm-mt-sect-eq5}.

Let $\wh\omega_i,\wh\eta_j\in H^1_\dR(J_\eta/L)$ be the de Rham symplectic basis.
Let $x\in D^{\an}$ be the point associated to $\frak m$.
Let $V^{\an}$ be a small neighborhood of $x$ and $j:V^{\an}\setminus D^{\an}\hra S^{\an}$ an open immersion .
Obviously 
$\wh\omega^{\an}_i=du_i/u_i\in\vg(V^{\an},j_*\O^{\an})\ot_A H^1_\dR(J_S/S)$.
Thanks to the uniquness property, this implies 
$\wh\eta^{\an}_j=\gamma_j^*\in\vg(V^{\an},j_*\O^{\an})\ot_A H^1_\dR(J_S/S)$,
in other words, $\gamma_j^*\in \vg(V^{\an}\setminus D^{\an},j^{-1}R^1h_*\Q)$.
Let $\wh S_{\frak m}$ be the ring of power series over $\C$ containing
$\wh A_m$ and $\vg(V^{\an},j_*\O^{\an})$. 
There is a commutative diagram
\[
\xymatrix{
\wh A_{\frak m}\ot_A\Omega^1_{A/\C}\ot_AH^1_\dR(J_S/S)\ar[r]&
\wh S_{\frak m}\ot_A\Omega^1_{A/\C}\ot_AH^1_\dR(J_S/S)\\
\Omega^1_{A/\C}\ot_AH^1_\dR(J_S/S)\ar[r]\ar[u]
&\vg(V^{\an},j_*\O^{\an})\ot_A\Omega^1_{A/\C}\ot_A H^1_\dR(J_S/S)\ar[u]
}
\]
with all arrows injective.
Hence the desired assertion for
$\wh\omega_i,\wh\eta_j$ can be reduced
to that of $\wh\omega_i^{\an},\wh\eta_j^{\an}$.
This completes the proof.
\end{pf}

\subsection{Periods of Jacobian scheme $J(\cX/W((\l)))$}\label{gm-jac-sect}
Let $f:Y\to \P^1_{\Z_p}$ be the HG fibration in \S \ref{main-sect-3}.
Let $W=W(\ol\F_p)$ be the Witt ring.
Recall that 
$f$ has a totally degenerate semistable reduction at $\l=0$.
The fiber $Z=f^{-1}(0)=Z_1\cup Z_2$ is a relative SNCD over $\Z_p$
and the intersection points $Z_1\cap Z_2$ consists of $W$-rational points 
\begin{equation}\label{sing-point}
P_\nu:\l_1=z=0,\,w_1=-\nu,\quad(\nu\in W^\times,\,\nu^N=-1)
\end{equation}
in the affine open set \eqref{int-locus3}.
Write $W((\l)):=W[[\l]][\l^{-1}]$.
Put $\cY:=Y\times_{\Z_p[\l]} \Spec W[[\l]]$ and $\cX:=\cY\times _{W[[\l]]}W((\l))$:
\begin{equation}\label{gm-sect-eq1}
\xymatrix{
Z\ar[r]\ar[d]&\cY\ar[d]&\cX\ar[l]\ar[d]\\
\Spec W\ar[r]&\Spec W[[\l]]&\Spec W((\l)).\ar[l]
}
\end{equation}
Let $J(\cY/W[[\l]])$ the Picard scheme (Jacobian scheme) associated to $\cY/W[[\l]]$.
Then $J(\cY/W[[\l]])$ is a totally degenerating abelian scheme over $(W[[\l]],(\l))$
in the sense of \S \ref{gm-mt-sect}.
Moreover the torus $T\cong \Pic^0(Z)$ is split
since $Z_1$, $Z_2$ and $P_\nu$ are geometrically irreducible over $W$.

\begin{thm}\label{gm-jac-sect-prop1}
Put 
\[
\kappa_n:=2\log N-\frac{1}{2}\sum_{\ve^N=1,\ve\ne1}(\ve^n+\ve^{-n})\log
(1-\ve)(1-\ve^{-1}),
\]
\[
K_{n,i}:=\sum_{k=1}^i\left(\frac{2}{k}-\frac{1}{k-n/N}-\frac{1}{k-1+n/N}\right).
\]
Put for $1\leq n\leq N-1$
\begin{equation}\label{gm-jac-sect-prop1-tau}
\tau_n(\l):=-\log\l+\kappa_n+
\frac{1}{F_n(\l)}\left(\sum_{i=1}^\infty K_{n,i}
\frac{\left(\frac{n}{N}\right)_i\left(1-\frac{n}{N}\right)_i}{i!^2}\l^i
\right).
\end{equation}
Fix a primitive $N$-th root $\zeta\in \Z_p^\times$ of unity, and put
$P:=(\zeta^{-ij})_{1\leq i,j\leq N-1}\in\mathrm{GL}_{N-1}(\Z_p)$.
Let
\begin{equation}\label{gm-jac-sect-prop1-eq0}
\underline{\tau}=(\tau_{ij}):=
P\begin{pmatrix}
(\zeta+\zeta^{-1}-2)\tau_1(\l)&&\text{\rm\Large 0}\\
&\ddots\\
\text{\rm\Large 0}&&(\zeta^{N-1}+\zeta^{1-N}-2)\tau_{N-1}(\l)
\end{pmatrix}P^{-1},
\end{equation}
be a $(N-1)\times(N-1)$-matrix and put
\[
\underline{q}:=(q_{ij})=(\exp(\tau_{ij})),
\]
where we regard $q_{ij}\in \Q_p((\l))$ (see Remark \ref{gm-jac-sect-rem} below).
Then $q_{ij}\in \Z_p((\l))$ and there is a surjective homomorphism
\[\bG_{m,W((\l))}^{N-1}/\underline{q}^\Z\lra J(\cX/W((\l)))\]
as abelian schemes over $W((\l))$.
Moreover, let $[\zeta]\in\mathrm{Aut}(X/S)$ denotes the automorphism 
given by $[\zeta](x,y)=(x,\zeta^{-1}y)$.
Then the induced automorphism on $J(\cX/W((\l)))$ is compatible with
the automorphism of $\bG_m^{N-1}/\underline{q}^\Z$ induced from 
a multiplication by matrix
\begin{equation}\label{gm-jac-sect-prop1-eq00}
P\begin{pmatrix}
\zeta&&\text{\rm\Large 0}\\
&\ddots\\
\text{\rm\Large 0}&&\zeta^{N-1}
\end{pmatrix}P^{-1}=\begin{pmatrix}
-1&-1&\cdots&-1&-1\\
1&0&\cdots&0&0\\
0&1&\cdots&0&0\\
\vdots&&\cdots&&\vdots\\
0&&\cdots&1&0
\end{pmatrix}
\end{equation}
on $\mathrm{Lie}(\bG_m^{N-1})$ (i.e. $(u_1,u_2,\ldots,u_{N-1})\mapsto
((u_1\cdots u_{N-1})^{-1},u_1,\ldots,u_{N-2})$ on $\bG_m^{N-1}$).
\end{thm}
\begin{rem}\label{gm-jac-sect-rem}
The definition \eqref{gm-jac-sect-prop1-eq0} is rewritten as
\[
\underline{\tau}=\frac{1}{N}\sum_{n=1}^{N-1}
(\zeta^n+\zeta^{-n}-2)\tau_n(\l)P_n ,\quad P_n:=(\zeta^{-in}(\zeta^{jn}-1))_{1\leq i,j\leq N-1}.
\]
Therefore 
the coefficients of ``$\log\l$'', ``$\log N$'' and ``$\log(1-\ve)(1-\ve^{-1})$''
are rational integers, and hence one can regard $q_{ij}\in \Q(\zeta)((\l))\subset \Q_p((\l))$.
\end{rem}
\begin{pf}
Fix an embedding $\iota:W\hra \C$ such that $\iota(\zeta)=\zeta_N:=
\exp(2\pi i/N)\in \C^\times$ and consider the analytic fibration
$f:X^{\an}\to S^{\an}$. Write $X_\l:=f^{-1}(\l)$ for $\l\in S^{\an}$.
We first compute the period integrals
\[\int_{\gamma_i}\omega_n
\]
where $\omega_n$ is the holomorphic 1-form 
in \eqref{main-sect-1-eq3}, and $\gamma_i$'s are the basis
of $H_1(X_\l,\Q)$ which we shall give in below.

Let $X_\l\to \P^1$ be the $N$-th cyclic covering
given by $(x,y)\mapsto x$.
This is ramified at $x=0,1$ with the ramification degree $N$, 
so that there are unique points $P_0$, $P_1\in X_\l$ above $x=0,1$ respectively.  
Let $|\l|<1$. Let $e$ be a path in $X_\l$ from $P_0$ to $P_1$ such that $x\in [0,1]$ and
$y=x^{1/N}(1-x)^{1-1/N}(1-\l x)^{1/N}$ takes the principal values.
Then
\begin{equation}\label{gm-jac-sect-prop1-eq1}
\int_e\omega_n=B(a_n,1-a_n){}_2F_1(a_n,1-a_n,1;\l),\quad a_n:=\frac{n}{N}
\end{equation}
by \cite{NIST} 15.6.1.
Let $[\zeta]:X_\l\to X_\l$ denotes the automorphism given by $[\zeta](x,y)=(x,\zeta_N^{-1}y)$.
We then have homology cycles $\delta_i:=(1-[\zeta_N^i])e\in H_1(X_\l,\Z)$, and 
\begin{equation}\label{gm-jac-sect-prop1-eq2}
\int_{\delta_i}\omega_n=(1-\zeta_N^{in})\int_e\omega_n
=(1-\zeta_N^{in})B(a_n,1-a_n){}_2F_1(a_n,1-a_n,1;\l)
\end{equation}
as $[\zeta_N^i]\omega_n=\zeta_N^{in}\omega_n$.
Let $T_1\in\pi_1(S^{\an},\l)$ be the local monodromy at $\l=1$, and put $\gamma_i:=T_1\delta_i$.
Then one has
\begin{equation}\label{gm-jac-sect-prop1-eq3}
\int_{\gamma_i}\omega_n
=(1-\zeta_N^{in})\cdot B(a_n,1-a_n)T_1{}_2F_1(a_n,1-a_n,1;\l)
=(1-\zeta_N^{in})\cdot 2\pi \sqrt{-1}{}_2F_1(a_n,1-a_n,1;1-\l)
\end{equation}
where the second equality follows from \cite{NIST} 15.8.10.
Since \eqref{gm-jac-sect-prop1-eq2} and \eqref{gm-jac-sect-prop1-eq3} are not
proportional and $\det(1-\zeta_N^{in})_{1\leq i,n\leq N-1}\ne0$, the homology cycles
$\delta_i,\gamma_j$ are linearly independent in $H_1(X_\l,\Q)$, and hence they give a basis
over $\Q$.
Put $f_n:=B(a_n,1-a_n){}_2F_1(a_n,1-a_n,1;\l)$, $g_n=2\pi\sqrt{-1}{}_2F_1(a_n,1-a_n,1;1-\l)$
and a $(N-1)\times(N-1)$-matrix
\[
Q:=\begin{pmatrix}
1-\zeta_N&1-\zeta^2_N&\cdots&1-\zeta^{N-1}_N\\
\vdots&\vdots&&\vdots\\
1-\zeta^{N-1}_N&1-\zeta^{2(N-1)}_N&\cdots&1-\zeta^{(N-1)^2}_N\\
\end{pmatrix}
=-N
\begin{pmatrix}
\zeta^{-1}_N&\cdots&\zeta^{-N+1}_N\\
\vdots&&\vdots\\
\zeta^{-N+1}_N&\cdots&\zeta^{-(N-1)^2}_N
\end{pmatrix}^{-1}.
\]
The period matrix is given as follows
\begin{align*}
&\left(\int_{\delta_i}\omega_n;\int_{\gamma_j}\omega_n\right)\\
&=
\begin{pmatrix}
(1-\zeta_N)f_1&\cdots&(1-\zeta^{N-1}_N)f_1&
(1-\zeta_N)g_1&\cdots&(1-\zeta^{N-1}_N)g_1\\
(1-\zeta^2_N)f_2&\cdots&(1-\zeta^{2(N-1)}_N)f_2&
(1-\zeta^2_N)g_2&\cdots&(1-\zeta^{2(N-1)}_N)g_2\\
&\vdots&&&\vdots&\\
(1-\zeta^{N-1}_N)f_{N-1}&\cdots&(1-\zeta^{(N-1)^2}_N)f_{N-1}&
(1-\zeta^{N-1}_N)g_{N-1}&\cdots&(1-\zeta^{(N-1)^2}_N)g_{N-1}\\
\end{pmatrix}\\
&=\begin{pmatrix}
f_1&&\\
&\ddots&\\
&&f_{N-1}
\end{pmatrix}Q \left[
\begin{matrix}1&&&\text{\rm\Large 0}\\&1&\\
&&\ddots\\
\text{\rm\Large 0}&&&1\end{matrix} \quad 
Q^{-1}\begin{pmatrix}g_1/f_1&&\\&g_2/f_2&\\
&&\ddots\\
&&&g_{N-1}/f_{N-1}\end{pmatrix} 
Q
\right]
\\
&=\overbrace{\frac{1}{2\pi i}\begin{pmatrix}
f_1&&\\
&\ddots&\\
&&f_{N-1}
\end{pmatrix}Q}^{(\ast)}
\left[
 \begin{matrix}
2\pi i &&\text{\rm\Large 0}\\
&\ddots&\\
\text{\rm\Large 0}&&2\pi i 
\end{matrix} \quad 
\overbrace{\iota(P)\begin{pmatrix}2\pi i \, g_1/f_1&&\\
&\ddots\\
&&2\pi i \, g_{N-1}/f_{N-1}\end{pmatrix}\iota(P)^{-1}}^{(\ast\ast)}
\right].
\end{align*}
Using the expansion of $F_n(\l)$ at $\l=1$ (\cite{NIST} 15.8.10) together with a formula
of Gauss on digamma function
(\cite{Bateman} 1.7.3. (29), p.19), one has
\[
2\pi i\,\frac{g_n}{f_n}=\left(\frac{2\pi i}{B(a_n,1-a_n)}\right)^2\tau_n(\l)
=(\zeta_N^n+\zeta_N^{-n}-2)\tau_n(\l).
\]
Hence the matrix $(\ast\ast)$
coincides with $\underline{\tau}$ 
in \eqref{gm-jac-sect-prop1-eq0} via the embedding $\iota:W\hra \C$.
Since $\langle\delta_i,\gamma_j\rangle\subset H_1(X_t,\Z)$ is a sub $\Z$-module,
one has a surjective homomorphism 
$(\C^\times)^g/\underline{q}^\Z\to J(X^{\an}_\l)$ locally around $\l=0$.
This means 
that there is a surjective homomorphism
\begin{equation}\label{gm-jac-sect-prop1-eq4}
\bG_{m,\C((\l))}^{N-1}/\underline{q}^\Z\lra \C((\l))\times_{W((\l))}J(\cX/W((\l)))
\end{equation} 
as abelian schemes over $\C((\l))$.
Since $2\pi i/B(a_n,1-a_n)=e^{\pi i n/N}-e^{-\pi i n/N}\in \iota(W^\times)$ 
and $Q\in\mathrm{GL}(\iota(W))$, the matrix
$(\ast)$ is an invertible matrix with coefficients in $\iota(W)[[\l]]$.
Since $\omega_1,\ldots,\omega_{N-1}$ forms a free basis of $H^0(\Omega^1_{\cX/W((\l))})$
over $W((\l))$, 
this implies that there is a uniformization $\bG_{m,W((\l))}^{N-1}\to J(\cX/W((\l)))$ 
over $W((\l))$ which sits into a commutative diagram
\[
\xymatrix{
\bG_{m,\C((\l))}^{N-1}\ar[r]\ar[d]&\bG_{m,W((\l))}^{N-1}\ar[d]\\
\bG_{m,\C((\l))}^{N-1}/\underline{q}^\Z\ar[r]&J(\cX/W((\l)))
}
\]
where the top arrow is the base change by $\iota$
and the bottom arrow is the composition of \eqref{gm-jac-sect-prop1-eq4} and 
the base change by $\iota$.
This implies $q_{ij}\in W((\l))$ (and hence $q_{jj}\in \Z_p((\l))$) and 
a surjective homomorphism
\[
\bG_{m,W((\l))}^{N-1}/\underline{q}^\Z\to J(\cX/W((\l)))\] 
as abelian schemes over $W((\l))$.
\end{pf}

\subsection{De Rham symplectic basis of $J(\cX/W((\l)))$}\label{gm-dR-sect}
Let $J=J(\cY/W[[\l]])$ be as before.
Let $L=\Frac W[[\l]]$ and denote by $J_\eta$ the generic fiber over $\Spec L$.
We give an explicit description of an de Rham symplectic basis $\{\wh\omega_i,\wh\eta_j\}$
of $H^1_\dR(J_\eta/L)$ in the sense of \S \ref{gm-mt-sect}.
\begin{prop}\label{gm-jac-sect-prop2}
Let $\wt\omega_n$ and $\wt\eta_n$ be as in \eqref{main-thm-wt}.
Put
\[
\wh\omega(\nu)=\sum_{n=1}^{N-1}\nu^{n}\wt\omega_n,\quad
\widehat\eta(\nu)=\sum_{n=1}^{N-1}\nu^{-n}\wt\eta_n
\]
for $\nu\in W^\times$ such that $\nu^N=-1$.
Then $\wh\omega_i$ are $\Q$-linear combinations of $\wh\omega(\nu)$'s,
and $\wh\eta_i$ are $\Q$-linear combinations of $\wh\eta(\nu)$'s.
\end{prop}
\begin{pf}
We again take an embedding $W\hra \C$, and consider the analytic fibration
$\cJ\to S^{\an}$ as in the proof of Theorem \ref{gm-jac-sect-prop1}.
Letting $\delta_i,\gamma_j\in H_1(X_\l,\Z)$ be as in there, it is enough to show
\begin{equation}\label{gm-jac-sect-prop2-eq1}
\frac{1}{2\pi\sqrt{-1}}\int_{\delta_i}\omega(\nu)\in \Q
\end{equation}
and 
\begin{equation}\label{gm-jac-sect-prop2-eq2}
\int_{\delta_i}\eta(\nu)=0,\quad
\int_{\gamma_j}\eta(\nu)\in \Q.
\end{equation}
Recall \eqref{gm-jac-sect-prop1-eq2}
\[
\int_{\delta_i}\omega_n=(1-\zeta_N^{in})B(a_n,1-a_n){}_2F_1(a_n,1-a_n,1;\l)
=-2\pi \sqrt{-1}\cdot\zeta_{2N}^n\frac{1-\zeta_N^{in}}{1-\zeta_N^n}
F_n(\l)
\]
where $\zeta_{2N}:=\exp(\pi\sqrt{-1}n/N)$.
Therefore we have
\[
\frac{1}{2\pi\sqrt{-1}}\int_{\delta_i}\omega(\nu)=-\sum_{n=1}^{N-1}\nu^n\cdot
\zeta_{2N}^n\frac{1-\zeta_N^{in}}{1-\zeta_N^n}\in \Z
\]
the desired assertion \eqref{gm-jac-sect-prop2-eq1}.
To show \eqref{gm-jac-sect-prop2-eq2}, we note 
\begin{equation}\label{gm-jac-sect-prop2-eq3}
\int_e\eta_n=B(2-a_n,a_n){}_2F_1(1+a_n,2-a_n,2;\l)=\frac{N}{n}B(a_n,1-a_n)F'_n(\l).
\end{equation}
Therefore we have
\[
\int_{\delta_i}
\wt\eta_n=(1-\zeta_N^{in})B(a_n,1-a_n)(\l-\l^2)(F'_n(\l)F_n(\l)-F'_n(\l)F_n(\l))=0,
\]
\begin{align*}
\int_{\gamma_i}\wt\eta_n
&=(\l-\l^2)\left(\int_{\gamma_i}F'_n(\l)\omega_n-\frac{n}{N}F_n(\l)\eta_n\right)\\
&=(\l-\l^2)(1-\zeta_N^{in})\cdot 2\pi \sqrt{-1}(F'_n(\l)F_n(1-\l)+F_n(\l)F'_n(1-\l))\\
&=(1-\zeta_N^{in})\frac{2\pi \sqrt{-1}}{B(a_n,1-a_n)}\tag{$\ast$}\\
&=-\zeta_{2N}^{-n}(1-\zeta_N^n)(1-\zeta_N^{in}).
\end{align*}
Here the equality ($\ast$) is proven in the following way.
Let $f(\l):= (\l-\l^2)(F'_n(\l)F_n(1-\l)+F_n(\l)F'_n(1-\l))$. It follows from the hypergeometric
differential equation for $F_n(\l)$ (\cite{NIST} 15.10.1) that one has $f'(\l)=0$, namely $f(\l)$
is a constant function. Using the expansion of $F_n(\l)$ at $\l=1$ (\cite{NIST} 15.8.10),
one has $\lim_{\l\to0}f(\l)=B(a_n,1-a_n)^{-1}$.

Now the above implies
\[
\int_{\delta_i}\eta(\nu)=0,\quad 
\int_{\gamma_j}\eta(\nu)=-\sum_{n=1}^{N-1} \nu^{-n}
\zeta_{2N}^{-n}(1-\zeta_N^n)(1-\zeta_N^{in})\in\Z. 
\]
This completes the proof.
\end{pf}
\begin{prop}\label{gm-jac-sect-prop3}
Let $\nabla:H^1_\dR(X/S)\to \Omega^1_S\ot H^1_\dR(X/S)$ be the Gauss-Manin connection. 
Then
\[
\begin{pmatrix}
\nabla(\omega_n)&\nabla(\eta_n)
\end{pmatrix}
=\begin{pmatrix}
d\l\ot\omega_n&d\l\ot\eta_n
\end{pmatrix}
\begin{pmatrix}
0&(1-\frac{n}{N})\frac{1}{\l-\l^2}\\
\frac{n}{N}&\frac{2\l-1}{\l-\l^2}
\end{pmatrix},
\]
\[
\begin{pmatrix}
\nabla(\wt\omega_n)&\nabla(\wt\eta_n)
\end{pmatrix}
=\begin{pmatrix}
d\l\ot\wt\omega_n&d\l\ot\wt\eta_n
\end{pmatrix}
\begin{pmatrix}
0&0\\
-(\l-\l^2)^{-1}F_n(\l)^{-1}&0
\end{pmatrix}.
\]
\end{prop}
\begin{pf}
The former is straightforward from \eqref{gm-jac-sect-prop1-eq1} and
\eqref{gm-jac-sect-prop2-eq3} together with the hypergeometric differential equation
for $F_n(\l)$ (\cite{NIST} 15.10.1).
The latter follows from this and a computation
\begin{align*}
\nabla(F_n(\l)^{-1}\omega_n)&=-\frac{F'_n(\l)}{F_n(\l)^2}d\l\ot\omega_n
+\frac{n}{N}F_n(\l)^{-1}d\l\ot\eta_n\\
&=-(\l-\l^2)^{-1}F_n(\l)^{-2}d\l\ot \wt\eta_n.
\end{align*}
\end{pf}
\begin{lem}\label{gm-jac-sect-lem1}
Let $\tau_n(\l)$ be as in Theorem \ref{gm-jac-sect-prop1}. Then
\[
\l\frac{d}{d\l}\tau_n(\l)=-(1-\l)^{-1}F_n(\l)^{-2}.
\]
Hence $\tau_n(\l)=-\log\l+\kappa_n+\wt\tau_n(\l)$
where  $\wt\tau_n(\l)$ is as in \eqref{main-sect-1-eq5}.
\end{lem}
\begin{pf}
Apply the projector 
$H^1_\dR(J_\eta/L)\to H^1_\dR(J_\eta/L)^{(n)}$ on the equality
\eqref{gm-mt-sect-eq3} with respect to the basis in Proposition \ref{gm-jac-sect-prop2}.
The multiplicative periods are computed in Theorem \ref{gm-jac-sect-prop1} 
so that we get
\[
\nabla(\wt\omega_n)=c\cdot d\tau_n\ot \wt\eta_n.
\]
for some $c\in \Q(\zeta)^\times$. 
Comparing this with Proposition \ref{gm-jac-sect-prop3}, 
 we have $-(1-\l)^{-1}F_n(\l)^{-2}=c\l\tau'_n(\l)$.
Now one concludes $c=1$ by looking at the constant terms in both side.
\end{pf}

\section{Frobenius on $H^1_\dR(X/S)$}\label{frob-sect}
\subsection{Frobenius on De Rham symplectic basis}
Let $V$ be a complete discrete valuation ring such that
the residue field $k$ is perfect and of characteristic $p$, and the fractional field
$K:=\Frac V$ is of characteristic zero.
Let $\sigma_K$ be
a $p$-th Frobenius endomorphism on $V$.

Let $A$ be an integral flat noetherian $V$-algebra
equipped with
a $p$-th Frobenius endomorphism $\sigma$ on $A$ which is
compatible with $\sigma_K$.
Assume that $A$ is $p$-adically complete and separated
and there is a family $(a_i)_{i\in I}$ of elements of $A$
such that it forms a $p$-basis of $A_n:=A/p^{n+1}A$ over 
$V_n:=V/p^{n+1}V$ for all $n\geq 0$
in the sense of \cite{Ka1} 
Definition 1.3. Notice that the latter assumption is equivalent to that
$(a_i)_{i\in I}$ forms a $p$-basis of $A_0$ over $V_0$ (loc.cit. Lemma 1.6).
Write $A_K:=A\ot_VK$.
We denote by $\FMIC(A_K)=\FMIC(A_K,\sigma)$ the category 
of triplets $(M,\nabla,\Phi)$ where
\begin{itemize}
\item
$M$ is a locally free $A_K$-module of finite rank, 
\item
$\nabla:M\lra \wh\Omega_{A/V}\ot_A M$ is an integrable connection where 
$\wh\Omega^1_{A/V}:=\varprojlim_n \Omega^1_{A_n/V_n}$ is a free $A$-module
(\cite{Ka1} Lemma (1.8)),
\item
$\Phi:\sigma^*M\to M$ is a horizontal $A_K$-linear map.
\end{itemize}
Letting $L:=\Frac(A)$, the category $\FMIC(L)$ is defined in the same way
by replacing $A$ with $L$.

Let $f:X\to\Spec A$ be a projective smooth morphism.
Write $X_n:=X\times_VV/p^{n+1}V$.
Then one has an object
\[
H^i(X/A):=(H^i_\dR(X/A)\ot_V K,\nabla,\Phi)\in \FMIC(A_K)
\]
where $\Phi$ is induced from the Frobenius on crystalline cohomology
via the comparison (\cite{BO} 7.4)
\[H^\bullet_\crys(X_0/A)\cong \varprojlim_n H^\bullet_\dR(X_n/A_n)\cong
H^\bullet_\dR(X/A).\]
Assume that there is a smooth $V$-algebra $A^a$ and a smooth projective morphism
$f^a:X^a\to \Spec A^a$ with a Cartesian diagram
\[
\xymatrix{
X\ar[r]\ar[d]\ar@{}[rd]|{\square}&\Spec A\ar[d]\\
X^a\ar[r]&\Spec A^a
}
\]
such that $\Spec A\to \Spec A^a$ is flat.
One has the overconvergent $F$-isocrystal
$R^if^a_{\rig,*}\O_{X^a}$ on $\Spec A^a_0$ (\cite{Et} 3.4.8.2).
Let
\[
H^i(X^a/A^a):=(H^i_\rig(X^a_0/A^a_0),\nabla,\Phi)\in \FMIC^\dag(A^a_K)
\] 
denote the associated object via the natural equivalence
$F\text{-}\mathrm{Isoc}^\dag(A^a_0)\cong\FMIC^\dag(A^a_K)$ (\cite{LS} 8.3.10).
Then the comparison
$H^i_\rig(X^a_0/A^a_0)\cong H_\dR^i(X^a/A^a)\ot_V(A^a)^\dag_K$ 
(cf. \eqref{eq:comparisonisom2})
induces an isomorphism
 $H^i(X/A)\cong A_K\ot_{(A^a)^{\dag}_K}H^i(X^a/A^a)$ in $\FMIC(A_K)$.
\begin{defn}[Tate objects]\label{frob-sect-def1}
For an integer $r$, a Tate object $A_K(r)$ is defined to be the triplet $(A_K,\nabla,\Phi)$ 
scuh that $\nabla=d$
is the usual differential operator, and $\Phi$ is a multiplication by $p^{-r}$.
\end{defn}
We define for $f\in A\setminus\{0\}$
\[
\log^{(\sigma)}(f):=p^{-1}\log\left(\frac{f^p}{f^\sigma}\right)=-\sum_{n=1}^\infty
\frac{p^{n-1}g^n}{n},
\quad
\frac{f^p}{f^\sigma}=1-pg
\]
which belongs to the $p$-adic completion of the ring $A[g]$.
In particular, if $f\in A^\times$, then $\log^{(\sigma)}(f)\in A$.
\begin{defn}[Log objects]\label{frob-sect-def2}
Let $\uq=(q_{ij})$ be a $g\times g$-symmetric matrix with $q_{ij}\in A^\times$.
We define a log object $\Log(\uq)=(M,\nabla,\Phi)$ to be the following.
Let 
\[
M=\bigoplus_{i=1}^gA_Ke_i\op\bigoplus_{i=1}^gA_Kf_i
\]
be a free $A_K$-module with a basis $e_i,f_j$.
The connection is defined by
\[
\nabla(e_i)=\sum_{j=1}^g \frac{dq_{ij}}{q_{ij}}\ot f_j,\quad \nabla(f_j)=0
\]
and the Frobenius $\Phi$ is defined by
\[
\Phi(e_i)=e_i-\sum_{j=1}^g \log^{(\sigma)}(q_{ij})f_j,\quad
\Phi(f_j)=p^{-1}f_j.
\]
\end{defn}
By definition there is an exact sequence
\[
0\lra \bigoplus_{j=1}^gA_K(1)f_j\lra \Log(\uq)\lra \bigoplus_{i=1}^gA_K(0)e_i\lra0.
\]
\begin{thm}\label{frob-sect-thm1}
Let $R$ be a flat $V$-algebra which is 
a regular noetherian domain complete with respect to a reduced ideal $I$.
Suppose that $R$ has a $p$-th Frobenius $\sigma$.
Let $J$ be a totally degenerating abelian scheme with a principal polarization over $(R,I)$
in the sense of \S \ref{gm-mt-sect}.
Let $\Spec R[h^{-1}]\hra \Spec R$ be an affine open set such that $J$ is proper
over $\Spec R[h^{-1}]$ and 
$q_{ij}\in R[h^{-1}]^\times$ where 
$\uq=(q_{ij})$ is the multiplicative periods as in \eqref{gm-sect-3-eq1}.
Suppose that $R/pR[h^{-1}]$ has a $p$-basis over $V/pV$. 
Let $A=R[h^{-1}]^\wedge$ be the $p$-adic completion of $R[h^{-1}]$.
Put $L:=\Frac(A)$ and $J_A:=J\ot_RA$.
We denote by $J_\eta$ the generic fiber of $J_A$. 
Then there is an isomorphism
\begin{equation}\label{frob-sect-eq1}
(H^1_\dR(J_\eta/L),\nabla,\Phi)\ot_{A_K} A_K(1)\cong \Log(\underline{q})\in 
\FMIC(L)
\end{equation}
which sends the de Rham symplectic basis $\wh\omega_i,\wh\eta_j\in H^1_\dR(J_\eta/L)$ 
to $e_i$, $f_j$ respectively.
\end{thm}
\begin{pf}
Let $q_{ij}$ be indeterminates with $q_{ij}=q_{ji}$, and 
$t_1,\ldots,t_r$ ($r=g(g+1)/2$) are products $\prod q_{ij}^{n_{ij}}$ such that 
$\sum n_{ij}x_ix_j$ is positive semi-definite and they give a $\Z$-basis of 
the group of the symmetric pairings.
Let $J_q=\bG_m^g/\uq^\Z$ be Mumford's construction of tbe quotient group scheme
over a ring $\Z_p[[t_1,\ldots,t_r]]$ (\cite{FC} Chapter III, 4.4).    
Then there is a Cartesian square
\[
\xymatrix{
J\ar@{}[rd]|{\square}\ar[r]\ar[d]&J_q\ar[d]\\
\Spec R\ar[r]&\Spec \Z_p[[t_1,\ldots,t_r]]}
\]
such that the bottom arrow sends $t_i$ to an element of $I$
by the functoriality of Mumford's construction (\cite{FC} Chapter III, 5.5).
Thus we may reduce the assertion to the case of 
$J=J_q$, $R=\Z_p[[t_1,\ldots,t_r]]$, $I=(t_1,\ldots,t_r)$ and $h=\prod q_{ij}$.
Moreover we may replace the Frobenius $\sigma$ on $R$ with an arbitrary one
(e.g. \cite{EM} 6.1), so
that we may assume that it is given as $\sigma(q_{ij})=q_{ij}^p$
and $\sigma(a)=a$ for $a\in \Z_p$. 
Under this assumption $\log^{(\sigma)}(q_{ij})=0$ by definition.
Therefore our goal is to show
\begin{equation}\label{frob-sect-eq2}
\nabla(\wh\omega_i)=\sum_{j=1}^g \frac{dq_{ij}}{q_{ij}}\ot \wh\eta_j,\quad 
\nabla(\wh\eta_j)=0,
\end{equation}
\begin{equation}\label{frob-sect-eq3}
\Phi(\wh\omega_i)=p\wh\omega_i,\quad
\Phi(\wh\eta_j)=\wh\eta_j.
\end{equation}
Since the condition {\bf (C)} in Proposition \ref{gm-mt-sect-prop1} is satisfied,
\eqref{frob-sect-eq2} is nothing other than \eqref{gm-mt-sect-eq3}.
We show \eqref{frob-sect-eq3}.
Let $\uq^{(p)}:=(q_{ij}^p)$ and $J_{q^p}:=\bG_m^g/(\uq^{(p)})^\Z$.
Then there is the natural morphism $\sigma_J:J_{q^p}\to J_q$ such that the following diagram
is commutative.
\[
\xymatrix{
\bG_m^g\ar[d]\ar@{=}[r]&\bG_m^g\ar[d]\\
J_{q^p}\ar[d]\ar[r]^{\sigma_J}&J_q\ar[d]\\
\Spec R\ar[r]^\sigma&\Spec R}
\]
Let $[p]:J_{q^p}\to J_{q^p}$ denotes the multiplication by $p$ with respect to the
commutative group scheme structure of $J_q$.
It factors through the canonical surjective morphism $J_{q^p}\to J_q$ so that we have
$[p]':J_q\to J_{q^p}$. Define $\varphi:=\sigma_J \circ[p]'$. 
Under the uniformization $\rho:\bG_m^g\to J_q$,
this is compatible with
a morphism $\bG_m^g\to \bG_m^g$ given by $u_i\to u^p_i$ and $a\to \sigma(a)$ for
$a\in R$, 
which we also write $\Phi$. 
Therefore
\[
\Phi=\varphi^*:H^1_\dR(J_\eta/L)\lra H^1_\dR(J_\eta/L).
\]
In particular $\Phi$ preserves the Hodge filtration, so that 
$\Phi(\wh\omega_i)$ is again a linear combination of $\wh\omega_i$'s.
Since 
\[
\rho^*\Phi(\wh\omega_i)=
\Phi\rho^*(\wh\omega_i)=\Phi\left(\frac{du_i}{u_i}\right)=p\frac{du_i}{u_i}
\]
one concludes $\Phi(\wh\omega_i)=p\wh\omega_i$.
On the other hand, since $\Phi(\ker\nabla)\subset \ker\nabla$ and
$\ker\nabla$ is generated by $\wh\eta_j$'s by \eqref{frob-sect-eq2},
$\Phi(\wh\eta_j)$ is again
a linear combination of $\wh\eta_i$'s.
Note
\[
\langle \Phi(x),\Phi(y)\rangle=p\langle x,y\rangle.
\]
Therefore 
\[
\langle \Phi(\wh\eta_j),p\wh\omega_j\rangle=
\langle \Phi(\wh\eta_j),\Phi(\wh\omega_j)\rangle=
p\langle \wh\eta_j,\wh\omega_j\rangle.
\]
This implies $\Phi(\wh\eta_j)=\wh\eta_j$, so we are done.
\end{pf}
\subsection{Frobenius on $J(\cX/W((\l)))$}
\begin{thm}\label{frob-thm}
Let $X/S$ be as in \S \ref{main-sect-1}.
Let $W=W(\ol\F_p)$ and
let $A=W((\l))^\wedge$ be the $p$-adic completion endowed with a $p$-th Frobenius
$\sigma$ compatible with the canonical Frobenius on $W$. Put $K:=\Frac(W)$, $A_K:=A\ot_WK$
and $L:=\Frac(A)$.
Put $X_A:=X\times_S\Spec A$ and $X_0:=X_A\times_A\Spec \ol\F_p((\l))$.
Then the $p$-th Frobenius map $\Phi:\sigma^*H^1_\rig(X_0/\ol\F_p((\l)))
\to H^1_\rig(X_0/\ol\F_p((\l)))$
is given by
\begin{equation}\label{frob-sect-thm2-eq0}
\begin{pmatrix}
\Phi(\wt\omega_n)&\Phi(\wt\eta_n)
\end{pmatrix}
=\begin{pmatrix}
\wt\omega_n&\wt\eta_n
\end{pmatrix}
\begin{pmatrix}
(-1)^{\frac{n(p-1)}{N}}p&0\\
-(-1)^\frac{n(p-1)}{N}p\tau^{(\sigma)}_n(\l)&
(-1)^{\frac{n(p-1)}{N}}
\end{pmatrix}
\end{equation}
where $\tau_n^{(\sigma)}(\l)$ is as in \eqref{main-sect-1-eq7} and
$\wt\omega_n,\wt\eta_n$ are as in \eqref{main-thm-wt}.
\end{thm}
\begin{pf}
Let $\uq$ be as in Theorem \ref{gm-jac-sect-prop1}, and
$J(\cY/W[[\l]])$ which is isogenous to $\bG_m^g/\uq^\Z$ over $W((\l))$.
Put $J:=A\ot_{W((\l))}J(\cY/W[[\l]])$.
By Theorem \ref{frob-sect-thm1}, there is an isomorphism
\begin{equation}\label{frob-sect-thm2-eq1}
 (H^1_\dR(J_\eta/L),\nabla,\Phi)\ot_{A_K} A_K(1)\cong \Log(\underline{q}).
\end{equation}
Unlike the isomorphism \eqref{frob-sect-eq1}, the de Rham symplectic basis
does not exactly correspond to the basis $e_i,f_j$ of $\Log(\uq)$
because $J$ is not isomorphic but isogenous to $\bG_m^g/\uq^\Z$.
Let $[\zeta]:X_\l\to X_\l$ denotes the automorphism given by $[\zeta](x,y)=(x,\zeta_N^{-1}y)$.
The action of $[\zeta]$ on $\Log(\uq)$
is given in the following way.
Let $Q$ be the matrix \eqref{gm-jac-sect-prop1-eq00}, and 
$Q':={}^tQ$.
Then
\begin{equation}\label{frob-sect-thm2-eq11}
\begin{pmatrix}
[\zeta]f_1,\ldots,[\zeta]f_{N-1}
\end{pmatrix}
=\begin{pmatrix}
f_1,\ldots,f_{N-1}
\end{pmatrix}Q,\quad
\begin{pmatrix}
[\zeta]e_1,\ldots,[\zeta]e_{N-1}
\end{pmatrix}
=\begin{pmatrix}
e_1,\ldots,e_{N-1}
\end{pmatrix}Q'
\end{equation}
where $e_i,f_j$ are the basis of $\Log(\uq)$ as in Definition \ref{frob-sect-def2}.
Applying the projector 
$P_n=\sum_{\ve^N=1}\ve^{-n}[\ve]$
to the isomorphism \eqref{frob-sect-thm2-eq1}, we have
\begin{equation}\label{frob-sect-thm2-eq2}
H^1_\dR(J_\eta/L)^{(n)}\ot_{A_K}A_K(1)\cong P_n\Log(\uq).
\end{equation} 
Let us put $e:=P_ne_1$ and $f=P_nf_1$.
Then
\[
0\lra A_K(1)f\lra P_n\Log(\uq)\lra A_K(0)e\lra0.
\]
The connection $\nabla$ and the Frobenius $\Phi$ are 
given by
\begin{equation}\label{frob-sect-thm2-eq3}
\nabla(e)=\xi\ot f,\quad \nabla(f)=0,
\end{equation} 
\begin{equation}\label{frob-sect-thm2-eq4}
\Phi(f)=p^{-1}f,\quad \Phi(e)=e-Gf,
\end{equation} 
where
\[
\xi=\sum a_{ij}\frac{dq_{ij}}{q_{ij}}=\sum a_{ij}d\tau_{ij}\in\wh\Omega_A^1,
\quad G=\sum a_{ij}\log^{(\sigma)}q_{ij}\in B^\dag
\]
for certain $a_{ij}\in \Q(\mu_N)$.
It is straightforward from \eqref{gm-jac-sect-prop1-eq0}
and \eqref{frob-sect-thm2-eq11}
to see that
$\xi$ is proportional to $d\tau_n$ defined in \eqref{gm-jac-sect-prop1-tau} and 
that $G$ is proportional to
\begin{align*}
&-\log^{(\sigma)}(\l)+\kappa^{(p)}_n+
(1-p^{-1}\sigma)\left[\frac{1}{F_n(\l)}\left(\sum_{i=1}^\infty K_{n,i}
\frac{\left(\frac{n}{N}\right)_i\left(1-\frac{n}{N}\right)_i}{i!^2}\l^i
\right)\right]\\
&=-\log^{(\sigma)}(\l)+\kappa_n^{(p)}+
\wt\tau_n(\l)-p^{-1}\wt\tau_n(\l^\sigma)\quad\mbox{(by Lemma \ref{gm-jac-sect-lem1})}\\
&=:\tau^{(\sigma)}_n(\l).
\end{align*}
Replace $f$ with $cf$ for some $c\in \Q(\mu_N)^\times$
such that $\xi=d\tau_n$.
Then it turns out that
$G=\tau^{(\sigma)}_n(\l)$.
Let $\nu$ be a fixed $2N$-th root of unity.
It follows from Theorem \ref{gm-jac-sect-prop1} that
there is a free basis $\nu^n\wt\omega_n,\nu^n\wt\eta_n$ 
of the left hand side of \eqref{frob-sect-thm2-eq2},
and they corresponds to $e,f$ up to multiplication by $\Q_p^\times$ 
by Theorem \ref{frob-sect-thm1}.
Replacing $e,f$ with $ce,cf$ for some $c\in \Q_p^\times$, 
we may assume that 
the isomorphism \eqref{frob-sect-thm2-eq2} sends
$\nu^n\wt\omega_n$ to $e$.
Then it follows from Proposition \ref{gm-jac-sect-prop3} and Lemma \ref{gm-jac-sect-lem1}
that we have 
\[
\nabla(\nu^n\wt\omega_n)=\tau'_nd\l\ot  \nu^{n}\wt\eta_n.
\]
Comparing this with \eqref{frob-sect-thm2-eq3}, one sees that
$\nu^{n}\wt\eta_n$ exactly corresponds to $f$ via \eqref{frob-sect-thm2-eq2}.
Hence, by \eqref{frob-sect-thm2-eq4} and taking account into the Tate twist $A_K(1)$,
we have
\[
\Phi(\nu^{n}\wt\eta_n)=\nu^{np}\Phi(\wt\eta_n)=\nu^{n}\wt\eta_n,
\]
\[
\Phi(\nu^n\wt\omega_n)=\nu^{np}\Phi(\wt\omega_n)=\nu^n(p\wt\omega_n-p\tau^{(\sigma)}_n(\l)
\wt\eta_n).
\]
This completes the proof.
\end{pf}
Applying $\Phi$ to 
\[
\begin{pmatrix}
\wt\omega_n&
\wt\eta_n
\end{pmatrix}=
\begin{pmatrix}
\omega_n&
\eta_n
\end{pmatrix}
\begin{pmatrix}
F_n(\l)^{-1}&(\l-\l^2)F'_n(\l)\\
0&-\frac{n}{N}(\l-\l^2)F_n(\l)
\end{pmatrix}
\]
we have
\begin{align*}
&\begin{pmatrix}
\Phi(\omega_n)&
\Phi(\eta_n)
\end{pmatrix}
\begin{pmatrix}
F_n(\l^\sigma)^{-1}&(\l^\sigma-(\l^\sigma)^2)F'_n(\l^\sigma)\\
0&-\frac{n}{N}(\l^\sigma-(\l^\sigma)^2)F_n(\l^\sigma)
\end{pmatrix}\\
&=\begin{pmatrix}
\wt\omega_n&\wt\eta_n
\end{pmatrix}
\begin{pmatrix}
(-1)^{\frac{n(p-1)}{N}}p&0\\
-(-1)^\frac{n(p-1)}{N}p\tau^{(\sigma)}_n(\l)&
(-1)^{\frac{n(p-1)}{N}}
\end{pmatrix} \quad\mbox{by \eqref{frob-sect-thm2-eq0}}\\
&=
\begin{pmatrix}
\omega_n&
\eta_n
\end{pmatrix}
\begin{pmatrix}
F_n(\l)^{-1}&(\l-\l^2)F'_n(\l)\\
0&-\frac{n}{N}(\l-\l^2)F_n(\l)
\end{pmatrix}
\begin{pmatrix}
(-1)^{\frac{n(p-1)}{N}}p&0\\
-(-1)^\frac{n(p-1)}{N}p\tau^{(\sigma)}_n(\l)&
(-1)^{\frac{n(p-1)}{N}}
\end{pmatrix} .
\end{align*}
Put
\begin{multline*}
\begin{pmatrix}
pF_{11}&F_{12}\\
pF_{21}&F_{22}
\end{pmatrix}
:=
\begin{pmatrix}
F_n(\l)^{-1}&(\l-\l^2)F'_n(\l)\\
0&-\frac{n}{N}(\l-\l^2)F_n(\l)
\end{pmatrix}
\begin{pmatrix}
(-1)^{\frac{n(p-1)}{N}}p&0\\
-(-1)^\frac{n(p-1)}{N}p\tau^{(\sigma)}_n(\l)&
(-1)^{\frac{n(p-1)}{N}}
\end{pmatrix} \\
\begin{pmatrix}
F_n(\l^\sigma)^{-1}&(\l^\sigma-(\l^\sigma)^2)F'_n(\l^\sigma)\\
0&-\frac{n}{N}(\l^\sigma-(\l^\sigma)^2)F_n(\l^\sigma)
\end{pmatrix}^{-1}
\end{multline*}
or more explicitly
\begin{align*}
F_{11}(\l)&:=(-1)^\frac{n(p-1)}{N}\left(
\frac{F_n(\l^\sigma)}{F_n(\l)}-(\l-\l^2)F'_n(\l)F_n(\l^\sigma)\tau_n^{(\sigma)}(\l)
\right)\\
F_{21}(\l)&:=(-1)^\frac{n(p-1)}{N}\frac{n}{N}(\l-\l^2)F_n(\l)F_n(\l^\sigma)
\tau_n^{(\sigma)}(\l)\\
F_{12}(\l)&:=\frac{N}{n}\left(p\frac{F'_n(\l^\sigma)}{F_n(\l^\sigma)}
F_{11}(\l)-(-1)^\frac{n(p-1)}{N}
\frac{\l-\l^2}{\l^\sigma-(\l^\sigma)^2}\frac{F'_n(\l)}{F_n(\l^\sigma)}\right)\\
F_{22}(\l)&:=\frac{N}{n}\left(p\frac{F'_n(\l^\sigma)}{F_n(\l^\sigma)}
F_{21}(\l)+(-1)^\frac{n(p-1)}{N}\frac{n}{N}
\frac{\l-\l^2}{\l^\sigma-(\l^\sigma)^2}\frac{F_n(\l)}{F_n(\l^\sigma)}\right).
\end{align*}
We have 
\begin{equation}\label{frob-eigen-eq0}
\begin{pmatrix}
\Phi(\omega_n)&\Phi(\eta_n)
\end{pmatrix}
=\begin{pmatrix}
\omega_n&\eta_n
\end{pmatrix}
\begin{pmatrix}
pF_{11}(\l)&F_{12}(\l)\\
pF_{21}(\l)&F_{22}(\l)\\
\end{pmatrix}.
\end{equation}
Note
\begin{equation}
\label{frob-eigen-det}
F_{11}(\l)F_{22}(\l)-F_{12}(\l)F_{21}(\l)=\frac{\l-\l^2}{\l^\sigma-(\l^\sigma)^2}
\end{equation}
since $\Phi$ acts on 
$\wt\omega_n\cup\wt\eta_n=-\frac{n}{N}(\l-\l^2)\omega_n\cup\eta_n$
by multiplication by $p$.
\begin{cor}\label{frob-eigen}
Each $F_{ij}(\l)$ is overconvergent, namely
\[
    F_{ij}(\l)\in K[\l,(\l-\l^2)^{-1}]^\dag.
\]
Let $\alpha\in W$ be an element such that 
$\l^\sigma|_{\l=\alpha}=\alpha^\sigma$
and $\alpha\not\equiv0,1$ mod $p$.
Let $X_\alpha\to\Spec W$ be the fiber over $\l=\alpha$, and
put $X_{\alpha,0}:=\Spec \ol\F_p\times_WX_\alpha$.
Then the $p$-th Frobenius $\phi$ on $H^1_\rig(X_{\alpha,0}/\ol\F_p)$ is given 
as the $\sigma$-linear endomorphism such that
\begin{equation}\label{frob-eigen-eq1}
\begin{pmatrix}
\phi(\omega_n)&\phi(\eta_n)
\end{pmatrix}
=\begin{pmatrix}
\omega_n&\eta_n
\end{pmatrix}
\begin{pmatrix}
pF_{11}(\l)&F_{12}(\l)\\
pF_{21}(\l)&F_{22}(\l)
\end{pmatrix}\bigg|_{\lambda=\alpha}
\end{equation}
under the identification
$H^1_\rig(X_{\alpha,0}/\ol\F_p)\cong H^1_\dR(X_\alpha/K)$.
\end{cor}
\begin{pf}
There remains only to show that $F_{ij}(\l)$ are overconvergent.
However this is immediate from the fact that
    $\Phi$ is the $p$-th Frobenius on $H^1_\rig(X_{\F_p}/S_{\F_p})$ which is a free
    $\Q_p[\l,(\l-\l^2)^{-1}]^\dag$-module with a basis
$\{\omega_n,\eta_n\}_{1\leq n\leq N-1}$.
\end{pf}
We also have a unit root formula of Dwork type for our 
hypergeometric fibration.
\begin{cor}[Unit root formula]
Let
\[
F_{n,<p}(\l):=\sum_{i=0}^{p-1} \frac{(n/N)_i(1-n/N)_i}{i!^2}\l^i\in \Z_p[\l]
 \]
denotes a truncated polynomial, and let
 \[
\frac{F_n(\l)}{F_n(\l^\sigma)}
    \in K[\l,(\l-\l^2)^{-1},F_{n,<p}(\l)^{-1}]^\dag
\]
be Dwork's $p$-adic hypergeometric function (\cite{Dwork-p-cycle} \S 3).
Let $\alpha\in W(\F_{p^r})$ be an element such that 
$\l^\sigma|_{\l=\alpha}=\alpha^\sigma$
and $\alpha\not\equiv0,1$ mod $p$.
Suppose that 
$H^1_\rig(X_{\alpha,0}/\F_{p^r})^{(n)}$ is ordinary, or equivalently
$F_{n,<p}(\alpha)\not\equiv 0$ mod $p$.
Then 
the unit root of the $p^r$-th Frobenius on $H^1_\rig(X_{\alpha,0}/\F_{p^r})^{(n)}$ 
is the evaluation 
\[
(-1)^\frac{rn(p-1)}{N}\frac{F_n(\l)}{F_n(\l^{\sigma^r})}\bigg|_{\l=\alpha}.
\]
\end{cor}
\begin{pf}
Let $\eta'_n:=F_n(\l)^{-1}\wt\eta_n$.
Then one can rewrite \eqref{frob-sect-thm2-eq0} as follows
\[
\begin{pmatrix}
\Phi(\omega_n)&\Phi(\eta'_n)
\end{pmatrix}
=\begin{pmatrix}
\omega_n&\eta'_n
\end{pmatrix}
\begin{pmatrix}
(-1)^{\frac{n(p-1)}{N}}\frac{pF_n(\l^\sigma)}{F_n(\l)}&0\\
\ast&(-1)^{\frac{n(p-1)}{N}}\frac{F_n(\l)}{F_n(\l^\sigma)}
\end{pmatrix},
\]
which gives the eigenvalues of the $p^r$-th Frobenius $\Phi^r$.
\end{pf}
\section{Syntomic regulator as 1-extension in $\FilFMIC(S)$}\label{1-ext-sect}
Let $f:Y\to \P^1_{\Z_p}$ be the HG fibration constructed in \S \ref{main-sect-3}.
Let $h_1,h_2$ be as in \eqref{main-thm-eqfg}.
Let $\ol U$ be the complement of the support of $h_1h_2$ in $Y$, so that
$h_i\in \O(\ol U)^\times$.
Let $S:=\Spec \Z_p[\l,(\l-\l^2)^{-1}]$ and $X:=f^{-1}(S)$ be as in \S \ref{main-sect-1}.
Put $U:=\ol U\cap X$. Then $U$ is a smooth affine variety over $\Z_p$ 
and the complement 
$D:=Y\setminus U$ is a relative NCD over $\Z_p$.
Throughout this section let $W=W(\ol\F_p)$ be the Witt ring of $\ol\F_p$ and
$K=\Frac(W)$ the fraction field.
\subsection{Computing 1-extension}
Put $B:=W[\l,(\l-\l^2)^{-1}]$, $B_K:=B\ot_WK$ and $A:=\vg(U,\O_U)\ot_{\Z_p}W$, 
$A_K:=A\ot_{\Z_p}\Q_p$.
Let $A^\dag$ and $B^\dag$ be the weak completions.
We write $A^\dag_K:=A^\dag\ot_{\Z_p}\Q_p$
and $B^\dag_K:=B^\dag\ot_{\Z_p}\Q_p$.
In terms of overconvergent sheaves, one has
\[
A^\dag_K=\vg(Y^\rig_K,j^\dag_U\O_{Y^\rig_K}),\quad
B^\dag_K=\vg(\P^{1,\rig}_K,j^\dag_S\O_{\P^{1,\rig}_K}).
\]
Let $\sigma$ be a $p$-th Frobenius on $A^\dag$ compatible with the Frobenius on $W$
such that $\sigma(B^\dag)\subset B^\dag$.
Define
\[
\log^{(\sigma)}(x):=p^{-1}\log\left(\frac{x^p}{x^{\sigma}}\right)\in \wh A:=\varprojlim_n
A/p^nA\]
for $x\in \wh A^\times$. 
\begin{lem}\label{log-lem}
If $x\in A^\times$ then $\log^{(\sigma)}(x)\in A^\dag$.
\end{lem}
\begin{pf}
Let $x\in A^\times$. Let $z\in A^\dag$ be an element such that
$x^p/x^{\sigma}=1-pz$. Then
\begin{equation}\label{log-lem-eq}
\log^{(\sigma)}(x)=-\sum_{n=1}^\infty \frac{p^{n-1}z^n}{n}.
\end{equation}
Fix a surjection $\pi\colon W[t_1,\dots,t_n]\to A$.
This extends to a surjection 
$\pi^{\dag}\colon W[t_1,\dots,t_n]^{\dag}\to A^{\dag}$.
If we denote by $S(r)\subset W[[t_1,\dots,t_n]]$ 
the subring of power series of radius of convergence $r$,
we have $\cup_{r>1}S(r)=W[t_1,\dots,t_n]^{\dag}$;
therefore, we may choose $r_0>1$ such that $z\in\pi^{\dag}(S(r_0))$.
Since $S(r_0)$ is the completion of $W[t_1,\dots,t_n]$ 
with respect to the $r_0$-Gauss norm,
the right hand side of \eqref{log-lem-eq} belongs to the image of $S(r_0)$. 
This means $\log^{(\sigma)}(x)\in A^{\dag}$.
\end{pf}
Let $\Log(h_i)$ be the log object with a free basis $e_{h_i,0}$ and
$e_{h_i,-2}$ defined in \S \ref{poly-obj-sect}. Recall from \eqref{log-obj-gp} the exact sequence
\[
0\lra A^\dag_{K}(1)\os{s}{\lra} \Log(h_i)\os{\rho}{\lra} A^\dag_{K} \lra0,
\]
where the maps
$s,\rho$ are given by $s(1)=e_{h_i,-2}$ and $\rho(e_{h_i,0})=1$.
Consider a distinguished triangle
\begin{equation}\label{1-ext-eq1}
0\lra A^\dag_{K}(2)\lra [\Log(h_1)(1)\os{\rho}{\to}\Log(h_2)]\lra A_{K}^\dag[-1]\lra0
\end{equation}
in the derived category of $\FilFMIC(U)$.
Each term gives rise to the de Rham complex and then we have a commutative diagram 
\begin{equation}\label{1-ext-eq-dr}
\xymatrix{
& A_{K}^\dag\ar[r]^d&\Omega^1_{A_{K}^\dag/B_{K}^\dag}\\
\Log(h_1)\ar[r]^{a\hspace{1.9cm}}&\Omega^1_{A_{K}^\dag/B_{K}^\dag}\ot \Log(h_1)\op \Log(h_2)
\ar[r]^{b\hspace{1cm}}\ar[u]_{0\op \rho}&
\Omega^2_{A_{K}^\dag/B_{K}^\dag}
\ot \Log(h_1)\op
\Omega^1_{A^\dag_{K}/B^\dag_{K}}\ot \Log(h_2)\ar[u]_{0\op(1\ot\rho)}
\\
A_{K}^\dag\ar[r]^{d\qquad}\ar[u]^s&\Omega^1_{A_{K}^\dag/B_{K}^\dag}
\ar[u]_{(1\ot s)\op0}
}
\end{equation}
where $\ot=\ot_{A_{K}^\dag}$, and 
$\Omega^i_{A_{K}^{\dag}/B_{K}^{\dag}}$ denotes 
the sheaf of continuous $i$-th differentials of $A_{K}$ over $B_{K}$,
and the maps $a$ and $b$ are given by
\[
a(x)=(\nabla(x),-\rho(x)),\quad b(\omega\ot x,y)=\omega\ot\rho(x)+\nabla(y).
\]
Recall from Proposition \ref{prop:from_K2} the 1-extension
\begin{equation}\label{1-ext-eq3}
0\lra H^1(U/S)(2)\lra M_{h_1,h_2}(U/S)\lra B_{K}\lra0
\end{equation}
in $\FilFMIC(S)$.
Here we note that $\dlog\{h_1,h_2\}=0$ since $\{h_1,h_2\}$ belongs to $K_2(X)$.
We simply write $M(U/S)=M_{h_1,h_2}(U/S)$, and denote by $M_\dR(U/S)$ (resp.
$M_\rig(U/S)$) the de Rham cohomology (resp. the rigid cohomology).
Then, in terms of the diagram \eqref{1-ext-eq-dr},
one has
\[
H^j_\rig(U_0/S_0)\cong H^j(\Omega^\bullet_{A^\dag_{K}/B^\dag_{K}}),\quad 
M_{h_1,h_2}^\rig\cong\ker(b)/\Image(a)
\]
where $U_0:=U\times_{\Z_p}\F_p$.
Put
\begin{equation}\label{lifting}
e_\dR:=-\frac{dh_2}{h_2}\ot e_{h_1,0}+e_{h_2,0}
\in \Omega^1_{Y/S}(\log D) e_{h_1,0}\op \Q_p\, e_{h_2,0}
\subset \Omega^1_{A^\dag} e_{h_1,0}\op A^\dag e_{h_2,0}.
\end{equation}
Then a direct computation yields
\begin{equation}\label{1-ext-eq4}
(0\op\rho)(e_\dR)=1,\quad b(e_\dR)=0.
\end{equation}
This shows that $e_\dR\in \Fil^0M_\dR(U/S)$ is the unique lifting
of $1\in B_{K}$ in \eqref{1-ext-eq3}.
Let $e_\rig\in M_\rig(U_0/S_0)$ be the corresponding element via the comparison.

\begin{thm}\label{1-ext-thm1}
Let $\Phi$ be the Frobenius on $M_\rig(U/S)$ induced by $\sigma$.
Then
under the inclusion
\[
H^1_\rig(U_0/S_0)=H^1(\Omega^\bullet_{A_{K}^\dag/B_{K}^\dag})
\lra M_\rig(U_0/S_0),\quad \omega\longmapsto \omega\ot e_{h_1,-2}
\]
we have
\begin{equation}\label{1-ext-thm1-eq1}
\Phi(e_\rig)-e_\rig=p^{-1}\log^{(\sigma)}(h_1)\frac{dh_2^\sigma}{h_2^\sigma}
-\log^{(\sigma)}(h_2)\frac{dh_1}{h_1}\in 
H^1_\rig(U_0/S_0).
\end{equation}
where $\log^{(\sigma)}(x)$ is the function as in Lemma \ref{log-lem}.
\end{thm}
\begin{pf}
By \eqref{lifting} we have
\[
\Phi(e_\rig)=
-\frac{dh_2^\sigma}{h_2^\sigma}\ot \left(p^{-1}e_{h_1,0}-p^{-1}
\log^{(\sigma)}(h_1)e_{h_1,-2}\right)
+e_{h_2,0}-\log^{(\sigma)}(h_2)e_{h_2,-2}.
\]
Hence
\begin{align*}
\Phi(e_\rig)-e_\rig
&=\frac{dh_2}{h_2}\ot e_{h_1,0}
-\frac{dh_2^\sigma}{h_2^\sigma}\ot 
\left(p^{-1}e_{h_1,0}-p^{-1}\log^{(\sigma)}h_1)e_{h_1,-2}\right)
-\log^{(\sigma)}(h_2)e_{h_2,-2}\\
&=
\overbrace{p^{-1}\left(p\frac{dh_2}{h_2}-
\frac{dh_2^\sigma}{h_2^\sigma}\right)}^{d\log^{(\sigma)}(h_2)}e_{h_1,0}
+
p^{-1}\log^{(\sigma)}(h_1)\frac{dh_2^\sigma}{h_2^\sigma}e_{h_1,-2}
-\log^{(\sigma)}(h_2)e_{h_2,-2}\\
&\equiv
-\left(-
p^{-1}\log^{(\sigma)}(h_1)\frac{dh_2^\sigma}{h_2^\sigma}
+\log^{(\sigma)}(h_2)\frac{dh_1}{h_1}
\right)e_{h_1,-2}\mod\Image(a)
\end{align*}
where the last equality follows from
\begin{align*}
0\equiv a(\log^{(\sigma)}(h_2)e_{h_1,0})
&=d\log^{(\sigma)}(h_2)e_{h_1,0}+\log^{(\sigma)}(h_2)\nabla(e_{h_1,0})-
\log^{(\sigma)}(h_2)e_{h_2,-2}\\
&=d\log^{(\sigma)}(h_2)e_{h_1,0}+\log^{(\sigma)}(h_2)\frac{dh_1}{h_1}e_{h_1,-2}-
\log^{(\sigma)}(h_2)e_{h_2,-2}.
\end{align*}
This completes the proof.
\end{pf}
\begin{thm}\label{1-ext-thm2}
Let
\[
0\lra H^1(X/S)(2)\lra M(X/S)\lra B_{K}\lra 0
\]
be the 1-extension in $\FilFMIC(S)$ associated to the $K_2$-symbol $\{h_1,h_2\}$
in \eqref{main-thm-eqfg} (see Proposition \ref{prop:from_K2X}).
Let $e_\dR\in \Fil^0M(X/S)$ be the unique lifting of $1\in B$, and $e_\rig\in M_\rig(X/S)$
the corresponding element via the comparison.
Then
\begin{equation}\label{1-ext-thm1-eq2}
\Phi(e_\rig)-e_\rig=p^{-1}\log^{(\sigma)}(h_1)\frac{dh_2^\sigma}{h_2^\sigma}
-\log^{(\sigma)}(h_2)\frac{dh_1}{h_1}\in 
H^1_\rig(X_0/S_0).
\end{equation}
\end{thm}
\begin{pf}
There is a commutative diagram
\[
\xymatrix{
0\ar[r] &H^1(X/S)(2)\ar[d]\ar[r]&M(X/S)(2)\ar[d]\ar[r]&B_{K}\ar@{=}[d]\ar[r]&0\\
0\ar[r] &H^1(U/S)(2)\ar[r]&M(U/S)(2)\ar[r]&B_{K}\ar[r]&0
}
\]
in $\FilFMIC(S)$. Since the vertical arrows are injective, the assertion is
reduced to Theorem \ref{1-ext-thm1}.
\end{pf}
\subsection{Comparison with the rigid syntomic regultor}
Let
$H^\bullet_{\rig\text{-}\syn}(V,\Q_p(j))$ be the {\it rigid syntomic cohomology}
of a smooth scheme $V$ over $W$ (\cite{Be1} 6.1).
There is the exact sequence 
\[
0\lra H^1_\rig(V_0/\ol\F_p)\lra H^2_{\rig\text{-}\syn}(V,\Q_p(2))
\lra \Fil^2H^2_\rig(V_0/\ol\F_p)
\]
where $V_0:=V\times_W\ol\F_p$ (loc.cit. 6.3).
If the relative dimension of $V/W$ is one, then the last term vanishes and hence
we have an isomorphism
\[
\iota:H^2_{\rig\text{-}\syn}(V,\Q_p(2))\os{\cong}{\lra} H^1_\rig(V_0/\ol\F_p).
\]
Suppose that $V=\Spec A$ is affine and 
let $\sigma$ be a $p$-th Frobenius on $A^\dag$ compatible with the Frobenius
on $W$.
Let
\[
\{-\}_{\rig\text{-}\syn}:A^\times \lra H^1_{\rig\text{-}\syn}(V,\Q_p(1)),\quad
f\longmapsto \mbox{class of }\left(\frac{df}{f},\log^{(\sigma)}f\right).
\]
be the symbol map (\cite{Be1} 10.3) which induces
the {\it rigid syntomic symbol map} on the Milnor $K$-groups
\[
K^M_r(A) \lra H^r_{\rig\text{-}\syn}(V,\Q_p(r)),\quad
\{f_1,\cdots,f_r\}\longmapsto 
\{f_1\}_{\rig\text{-}\syn}\cup\cdots\cup
\{f_1\}_{\rig\text{-}\syn}.
\]

\begin{prop}\label{1-ext-thm3}
The following diagram
\[
\xymatrix{
K^M_2(A)\ar[r]\ar[d]
&\Ext^1_{\Fil\text{-}F\text{-}\mathrm{mod}}(\Q_p,
H^1(V/K)(2))\ar[r]^{\hspace{1.5cm}\rho}
&H^1_\rig(V_0/\ol\F_p)
\ar@{=}[d]\\
H^2_{\rig\text{-}\syn}(V,\Q_p(2))\ar[rr]^\iota_\cong&&H^1_\rig(V_0/\ol\F_p)\\
}
\]
is commutative up to sign.
Here the left vertical arrow denotes
the symbol map on Milnor $K$-groups to the rigid syntomic cohomology groups.
The map $\rho$ is given in the following way.
Let 
\[
0\lra H^1(V/K)(2)\lra M\lra K(0)\lra 0
\]
be an 1-extension of filtered $\Phi$-modules. Let $e_\dR \in \Fil^0M_\dR$
be the unique lifting of $1\in K$. Let $e_\rig\in M_\rig$ be the corresponding element
via the comparison.
Then $\rho$ associates $\Phi(e_\rig)-e_\rig\in H^1_\rig(V/K)$ to the 
above extension.
\end{prop}
\begin{pf}
By definition of the cup product of syntomic cohomology, we have
\[
\{f\}_{\rig\text{-}\syn}\cup\{g\}_{\rig\text{-}\syn}
=\mbox{(the cohomology class of $R_{f,g}$)}
\]
for $\{f,g\}\in K^M_2(A)$ where
\[
R_{f,g}:=\left(0,p^{-1}\log^{(\sigma)}(f)\frac{dg^\sigma}{g^\sigma}
-\log^{(\sigma)}(g)\frac{df}{f}\right)\in \{0\}\op \Omega^1_{A^\dag}
\subset \R\Gamma_{\rig\text{-}\syn}(V,2)
\] 
(see also \cite{Be2} (3.2)).
This means
\[
\iota\{f,g\}_{\rig\text{-}\syn}=p^{-1}\log^{(\sigma)}(f)\frac{dg^\sigma}{g^\sigma}
-\log^{(\sigma)}(g)\frac{df}{f}\in H^1_\rig(V_0/\ol\F_p),
\] 
and this agrees with \eqref{1-ext-thm1-eq1} up to sign.
\end{pf}
\section{Proof of Main Theorem : Part 1}\label{mainpf1-sect}
We are now in a position to prove Theorem \ref{main-thm}.
Let 
\[
B=W[\l,(\l-\l^2)^{-1}]\subset B_K=K[\l,(\l-\l^2)^{-1}]
\subset B^\dag_{K}=B^\dag\ot_{\Z_p}\Q_p=K[\l,(\l-\l^2)^{-1}]^\dag
\]
be as in the beginning of \S \ref{1-ext-sect}.
Recall from Theorem \ref{1-ext-thm2} the 1-extension
\[
0\lra H^1(X/S)(2)\lra M(X/S)\lra B_{K}\lra 0
\]
in $\FilFMIC(S)$.
Let $e_\dR\in \Fil^0M_\dR(X/S)$ be the unique lifting of $1\in B$, and
$e_\rig\in M_\rig(X/S)$ the element corresponding to $e_\dR$
via the comparison
\[
B^\dag_{K}\ot_{B_{K}} H^1_\dR(X/S)\cong 
H^1_\rig(X_0/S_0)
\]
where $X_0:=X\times_{\Z_p}\F_p$ and $S_0:=\Spec\ol\F_p[\l,(\l-\l^2)^{-1}]$.

\begin{lem}\label{gm-lem1}
\begin{equation}\label{gm-prop-2}
\nabla(e_\dR)=\frac{dh_1}{h_1}\frac{dh_2}{h_2}
=\sum_{n=1}^{N-1}\frac{\zeta_1^n-\zeta_2^n}{N}\cdot
\frac{d\l}{1-\l}\omega_n.
\end{equation}
\end{lem}
\begin{pf}
Exercise.
\end{pf}

We turn to the proof of Theorem \ref{main-thm}.
Define overconvergent functions
$\wt\ve^{(n)}_1(\l),\,\wt\ve^{(n)}_2(\l)\in B^\dag_{K}$
by
\begin{equation}\label{mainpf1-eq1}
e_\rig-\Phi(e_\rig)=
\sum_{n=1}^{N-1}
\wt\ve^{(n)}_1(\l)\omega_n+\wt\ve^{(n)}_2(\l)\eta_n
\in B^\dag_{K}\ot_{B_{K}} H^1_\dR(X/S).
\end{equation}
Then
\[
\reg_\syn(\{h_1,h_2\}|_{\l=\alpha})=\sum_{n=1}^{N-1}
\wt\ve^{(n)}_1(\alpha)\omega_n+\wt\ve^{(n)}_2(\alpha)\eta_n
\]
by Theorem \ref{1-ext-thm2} and Proposition \ref{1-ext-thm3}.
Rewrite
\[
\wt\ve^{(n)}_1(\l)\omega_n+\wt\ve^{(n)}_2(\l)\eta_n
=\wt E^{(n)}_1(\l)\wt\omega_n+\wt E^{(n)}_2(\l)\wt\eta_n
\]
where $\{\wt\omega_n,\wt\eta_n\}$ is the basis \eqref{main-thm-wt}.
Our goal is to show
\[
\wt E^{(n)}_i(\l)=\frac{\zeta^n_1-\zeta^n_2}{N}E^{(n)}_i(\l).
\]
To do this it is enough to check the differential equations
\eqref{main-diffeq-1-p} and \eqref{main-diffeq-2-p} and the initial conditions
\eqref{main-diffeq-3-p} for $\wt E^{(n)}_i(\l)$. 
In this section we check the differential equations, namely
\begin{equation}\label{main-diffeq-1}
\frac{d}{d\l}\wt E_1^{(n)}(\l)=
\frac{\zeta_1^n-\zeta_2^n}{N}\left(\frac{F_n(\l)}{1-\l}-(-1)^{\frac{(p-1)n}{N}}p^{-1}
\frac{F_n(\l^\sigma)}{1-\l^\sigma}\frac{d\l^\sigma}{d\l}\right),
\end{equation}
\begin{equation}\label{main-diffeq-2}
\frac{d}{d\l}\wt E_2^{(n)}(\l)=
\frac{\zeta_1^n-\zeta_2^n}{N}\left(
\frac{E_1(\l)F_n(\l)^{-2}}{\l-\l^2}+
(-1)^{\frac{(p-1)n}{N}}p^{-1}\tau^{(\sigma)}_n(\l)
\frac{F_n(\l^\sigma)}{1-\l^\sigma}\frac{d\l^\sigma}{d\l}\right).
\end{equation}

Notice that $\wt E^{(n)}_i(\l)$ belong to an affinoid ring
\[
C:=\Q\ot\varprojlim_n \left(W/p^nW[\l,(\l-\l^2)^{-1},F_{n,<p}(\l)^{-1}]\right)\subset 
K[[\l]]
\]
where 
\[
F_{n,<p}(\l):=\sum_{i=0}^{p-1} \frac{(n/N)_i(1-n/N)_i}{i!^2}\l^i\in \Z_p[\l]
 \]
 is a truncated polynomial.
Applying $\nabla$ on \eqref{mainpf1-eq1}, we have
\begin{align*}
\nabla(e_\dR)-\Phi\nabla(e_\dR)
&=\sum_{n=1}^{N-1}d\,\wt E^{(n)}_1(\l)\wt\omega_n+\wt E^{(n)}_1(\l)\nabla\wt\omega_n
+d\,\wt E^{(n)}_2(\l)\wt\eta_n+\wt E^{(n)}_2(\l)\nabla \wt\eta_n\\
&=\sum_{n=1}^{N-1}d\,\wt E^{(n)}_1(\l)\ot\wt\omega_n-\wt E^{(n)}_1(\l)
\frac{F_n(\l)^{-1}}{\l-\l^2}d\l \ot\wt\eta_n
+d\,\wt E^{(n)}_2(\l)\ot\wt\eta_n\\
&\in\wh\Omega^1_{C}\ot_B H^1_\dR(X_{\Q_p}/S_{\Q_p})
\end{align*}
where the second equality follows from Proposition \ref{gm-jac-sect-prop3}.
On the other hand, applying $1-\Phi$ on \eqref{gm-prop-2}, we have
\begin{align*}
\nabla(e_\dR)-\Phi\nabla(e_\dR)
&=\sum_{n=1}^{N-1}\frac{\zeta_1^n-\zeta_2^n}{N}
\left(\frac{d\l}{1-\l}\omega_n-\frac{d\l^\sigma}{1-\l^\sigma}\Phi(\omega_n)\right)\\
&=\sum_{n=1}^{N-1}\frac{\zeta_1^n-\zeta_2^n}{N}
\left(F_n(\l)\frac{d\l}{1-\l}\wt\omega_n-F_n(\l^\sigma)
\frac{d\l^\sigma}{1-\l^\sigma}\Phi(\wt\omega_n)\right)\\
&=\sum_{n=1}^{N-1}\frac{\zeta_1^n-\zeta_2^n}{N}
\left(F_n(\l)\frac{d\l}{1-\l}\wt\omega_n-(-1)^{\frac{n(p-1)}{N}}
p^{-1}F_n(\l^\sigma)\frac{d\l^\sigma}{1-\l^\sigma}
(\wt\omega_n-\tau^{(p)}_n(\l)\wt\eta_n)\right)\\
&\in\wh\Omega^1_{C}\ot_B H^1_\dR(X_{\Q_p}/S_{\Q_p})
\end{align*}
where the third equality follows from Theorem \ref{frob-thm}
and taking account into the Tate twist in ``$H^1(X/S)(2)$''. 
Comparing the two results,
 \eqref{main-diffeq-1} and \eqref{main-diffeq-2} are immediate.

\section{Proof of Main Theorem : Part 2}\label{mainpf2-sect}
There remains to prove 
\begin{equation}\label{main-diffeq-3}
\wt E^{(n)}_1(0)=0,\quad \wt E^{(n)}_2(0)=2(\zeta^n_1-\zeta^n_2)\sum_{\nu^N=-1}
\nu^{-n}\mathrm{ln}^{(p)}_2(\nu).
\end{equation}

Let $W=W(\ol\F_p)$ be the Witt ring.
We use the notation in \eqref{sing-point} and \eqref{gm-sect-eq1}.
Moreover we put $Z_K:=Z\times_WK$, $\cY_K:=\cY\times _{W[[\l]]}K[[\l]]$
and $\cX_K:=\cY\times _{W[[\l]]}K((\l))$:
\[
\xymatrix{
\cY\ar[d]&\cX\ar[l]\ar[d]\\
\Spec W[[\l]]&\Spec W((\l))\ar[l]
}
\quad \xymatrix{
\cY_K\ar[d]&\cX_K\ar[l]\ar[d]\\
\Spec K[[\l]]&\Spec K((\l)).\ar[l]
}
\]

\subsection{Syntomic cohomology of log schemes}
In this section we use the syntomic cohomology by Kato \cite{Ka1}
rather than the rigid syntomic cohomology by Besser, 
as we shall discuss schemes such as $\cY$ which are 
regular but are not of finite type over $W$.

\medskip

Let $X$ be a scheme over $W$. We write $W_n:=W/p^n$ and 
$X_n:=X\times_W\Spec W_n$ for an integer $n\geq 1$.
Let $0\leq r <p$ be an integer.
Suppose that 
$X$ satisfies the assumption in \cite{Ka1} 2.4.
The syntomic complex $\cS_n(r)_X$ is defined to be the mapping fiber
of the homomorphism
\begin{equation}\label{syn-cpx-eq0}
1-p^{-r}f:
J_{D_n}^{[r-\bullet]}\ot_{\O_{P_n}}\Omega^\bullet_{P_n}
\lra \O_{D_n}\ot_{\O_{P_n}}\Omega^\bullet_{P_n}
\end{equation}
of sheaves on the etale site of $X_1$
(\cite{Ka1} 2.5, \cite{Ka2} (1.6)), where $X\hra P$ an immersion such that
$P$ is endowed with a $p$-th Frobenius $f$ and $D_n$ denotes
the DP-envelope of $X_n\subset P_n$ with respect to the canonical
divided power on $pW_n$, and $J^{[i]}_{D_n}$ is the $i$-th DP ideal.
The term of $\cS_n(r)_X$ of degree $q$ is
\begin{equation}\label{syn-cpx-eq1}
(J_{D_n}^{[r-q]}\ot_{\O_{P_n}}\Omega^q_{P_n})
\op(\O_{D_n}\ot_{\O_{P_n}}\Omega^{q-1}_{P_n}).
\end{equation}
We write $\cS_n(r)_{X,P,f}$ if we make the choice of $(P,f)$ clear.
Although the syntomic complex apparently depends on the choice of 
$(P,f)$, there is a functorial isomorphism between $\cS_n(r)_{X,P,f}$ and
$\cS_n(r)_{X,P',f'}$
in the derived category $D^b(X_1^\et)$ of bounded complexes of sheaves on $(X_1)_\et$
(\cite{Ka1} p. 412).
The syntomic cohomology is defined to be the etale cohomology of the syntomic complex
\[
H^i_\syn(X,\Z/p^n(r)):=H^i_{\et}(X_1,\cS_n(r)_X),\quad H^i_\syn(X,\Z_p(r)):=
\varprojlim_n H^i_\syn(X,\Z/p^n(r)).
\]
Let $U=\coprod U_i\to X$ be an etale covering.
Then we may replace $\cS_n(r)_X$ with $\cS_n(r)_{U^\bullet}$
which is defined to be the total complex of 
\[
 p_{1*}\cS_n(r)_{U^{(1)},P_1,f_1}\lra p_{2*}\cS_n(r)_{U^{(2)},
P_2,f_2}\lra \cdots
\]
where $p_k:U^{(k)}:=U\times_X\cdots\times_X U\to X$
and we choose $(P_i,f_i)$'s such that
all arrows are homomorphism of complexes (not only homomorphisms
in the derived category).
Let $\cS_n(r)'_X$ be the mapping fiber of \eqref{syn-cpx-eq0}
replacing $1-p^{-r}f$ with $p^r-f$
 \begin{equation}\label{syn-cpx-eq3}
\xymatrix{
\cS_n(r)_X\ar[r]\ar[d]&J_{D_n}^{[r-\bullet]}\ot_{\O_{P_n}}\Omega^\bullet_{P_n}
\ar@{=}[d]\ar[r]^{1-p^{-r}f}& \O_{D_n}\ot_{\O_{P_n}}\Omega^\bullet_{P_n}\ar[d]^{p^r}\\
\cS_n(r)'_X\ar[r]&J_{D_n}^{[r-\bullet]}\ot_{\O_{P_n}}\Omega^\bullet_{P_n}
\ar[r]^{p^r-f}& \O_{D_n}\ot_{\O_{P_n}}\Omega^\bullet_{P_n}.
}
\end{equation}
Let $H_\syn(X,\Z/p^n(r))':=H_\et(X_1,\cS_n(r)'_X)$ and
$H_\syn(X,\Z_p(r))':=\varprojlim_n H_\et(X_1,\cS_n(r)'_X)$.
The diagram \eqref{syn-cpx-eq3} implies that
$H_\syn(X,\Z_p(r))\to H_\syn(X,\Z_p(r))'$ is bijective modulo torsion.
We claim that there is a commutative diagram 
 \begin{equation}\label{syn-cpx-eq4}
\xymatrix{
J_{D_n}^{[r-\bullet]}\ot_{\O_{P_n}}\Omega^\bullet_{P_n}\ar[r]^{f}\ar[d]
&\O_{D_n}\ot_{\O_{P_n}}\Omega^\bullet_{P_n}\ar[d]\\
\Omega^{\bullet\geq r}_{X_n/W_n}\ar[r]^{f_{X_n}}& \Omega^{\bullet}_{X_n/W_n}
}
\end{equation}
in $D^b(X_1^\et)$.
Note first that the right vertical arrow is a quasi-isomorphism (\cite{Ka1} Theorem 1.7).
Then $f_{X_n}$ is defined to be the composition of
$\Omega^{\geq r}_{Z_n/W_n}\hra \Omega^\bullet_{X_n/W_n}\os{\sim}{\leftarrow}
\O_{D_n}\ot_{\O_{P_n}}\Omega^\bullet_{P_n}
\os{f}{\to} \O_{D_n}\ot_{\O_{P_n}}\Omega^\bullet_{P_n}\os{\sim}{\to}
\Omega^\bullet_{X_n/W_n}$. 
The commutative diagram \eqref{syn-cpx-eq4} is immediate from the definition of $f_X$.
This gives rise to a homomorphism
 \begin{equation}\label{syn-cpx-hom}
\phi_{X_n}:\cS_n(r)'_X
\lra 
\text{Mapping Fiber of }(p^r-f_{X_n})
 \end{equation}
 in $D^b(X_1^\et)$.
 One can also define the syntomic complexes
\[
\cS_n(r)_{(X,M)}
,\quad
\cS_n(r)'_{(X,M)}
\] 
of a log scheme $(X,M)$ over the base $\Spec W$ with trivial log structure
which satisfies \cite{AS-JAG} Condition 2.2.3 (see also Prop. 2.2.6).
In a down-to-earth manner, the syntomic complex $\cS_n(r)_{(X,M)}$
is given as the mapping fiber of
\[
1-p^{-r}f:J^{[r-\bullet]}_{D_n}\ot_{\O_{P_n}}\omega^\bullet_{(P_n,M_{P_n})}
\lra \O_{D_n}\ot_{\O_{P_n}}\omega^\bullet_{(P_n,M_{P_n})}
\]
and $\cS_n(r)'_{(X,M)}$
is the mapping fiber of $p^r-f$
where $(X,M)\hra (P,M_P)$ is as in \cite{AS-JAG} Condition 2.2.3, and
$\omega_{(P_n,M_{P_n})}$ denotes the differential module of the log scheme 
$(P_n,M_{P_n})=(P,M_P)\ot\Z/p^n$ (\cite{Ka3} 1.7).
If $M$ is the log structure associated to a normal crossing divisor $D$,
we write $(X,M)=X(D)$ and 
\[
H^i_\syn(X(D),\Z/p^n(r))=H^i_\et(X_1,\cS_n(r)_{X(D)}),
\quad
H^i_\syn(X(D),\Z/p^n(r))'=H^i_\et(X_1,\cS_n(r)'_{X(D)}).
\]
In the same way as before, one can define a homomorphism
\[
f_{X(D)_n/W_n}:
\Omega_{X_n/W_n}^{\bullet\geq r}(\log D_n)\lra 
\Omega_{X_n/W_n}^{\bullet}(\log D_n)
\]
and one has a homomorphism
 \begin{equation}\label{syn-cpx-hom-log}
\phi_{X(D)_n}:\cS_n(r)'_{X(D)}
\lra 
\text{Mapping Fiber of }(p^r-f_{X(D)_n/W_n})
 \end{equation}
 in $D^b(X_1^\et)$.
Let $\cY$ be a regular scheme.
Let $f:\cY\to S:=\Spec W[[\l]]$ be a flat and quasi-projective 
morphism which is smooth over 
$\Spec W((\l))=\Spec W[[\l]][\l^{-1}]$.
Let $D:=f^{-1}(0)$ be the fiber over $\l=0$.
Suppose that the reduced part $D_{\mathrm{red}}$ is a relative NCD over $W$, 
and the multiplicity of
each component of $D$ is prime to $p$.
We endow the log structure $N$ on $S$ associated to the divisor $\l=0$, and
the canonical DP structure on $pW[[\l]]$.
Then $\cY(D)/(S,N)$ sits in the situation in \cite{Ka3} \S 6.
Let $u^{\log}:((\cY_n/S_n)^{\log}_\crys)^\sim\to (\cY_n^\et)^\sim$ 
be the canonical morphism
from the category of log crystals to the category of etale sheaves. 
By \cite{Ka3} Theorem (6.4),
$Ru_*^{\log}\O_{\cY_n/S_n}$ is canonically
quasi-isomorphic to the log de Rham complex
\[
\omega_{\cY_n/S_n}^{\bullet}:=\Coker\left[
\frac{d\l}{\l}\ot\Omega_{\cY_n/W_n}^{\bullet-1}(\log D_n)\to
\Omega_{\cY_n/W_n}^{\bullet}(\log D_n)\right].\]
We fix a $p$-th
Frobenius $\sigma$ on $W[[\l]]$ such that $\sigma(\l)=a\l^p$ 
for some $a\in 1+pW$.
In the same way as before one can define a homomorphism
\begin{equation}\label{syn-cpx-hom-log-rel-f}
f_{\cY_n/S_n}:
\omega_{\cY_n/S_n}^{\bullet\geq r}\lra 
\omega_{\cY_n/S_n}^{\bullet}
\end{equation}
compatible with $\sigma$,
and
\[
\phi_{\cY_n/S_n}:\cS_n(r)'_{\cY(D)}
\lra 
\text{Mapping Fiber of }(p^r-f_{\cY_n/S_n})
\]
which obviously factors through $\phi_{\cY(D)_n}$ in \eqref{syn-cpx-hom-log}.
We sum up the above in the following proposition.
\begin{prop}\label{syn-cpx-hom-prop}
Let $f:\cY\to S:=\Spec W[[\l]]$ be a flat and quasi-projective 
morphism such that it is smooth over 
$\Spec W((\l))=\Spec W[[\l]][\l^{-1}]$, and letting
$D:=f^{-1}(0)$ be the fiber over $\l=0$, 
the reduced part $D_{\mathrm{red}}$ is a relative NCD over $W$, 
and the multiplicity of
each component of $D$ is prime to $p$. Fix 
a $p$-th Frobenius $\sigma$ on $W[[\l]]$ such that 
$\sigma(\l)=a\l^p$ for some $a\in 1+pW$.
Then there is a natural homomorphism
\begin{equation}\label{syn-cpx-hom-log-rel}
\phi_{\cY_n/S_n}:\cS_n(r)'_{\cY(D)}
\lra 
\text{Mapping Fiber of }(p^r-f_{\cY_n/S_n}).
\end{equation}
\end{prop}
Note that $\phi_{\cY/S}$ depends on the choice of Frobenius $\sigma$ on $W[[\l]]$.
\begin{cor}
Suppose $\dim(\cY/S)<r$.
Then 
\eqref{syn-cpx-hom-log-rel} induces
\begin{equation}\label{syn-cpx-hom-log-rel-1}
\phi_{\cY_n/S_n}:\cS_n(r)'_{\cY(D)}
\lra 
\omega_{\cY_n/S_n}^{\bullet}[-1].
\end{equation}
Suppose further that $f$ is projective, then this induces
\begin{equation}\label{syn-cpx-hom-log-rel-2}
\phi_{\cY/S}:H^i_\syn(\cY(D),\Z_p(r))':=\varprojlim_nH^i_\syn(\cY(D),\Z/p^n(r))'
\lra 
H^{i-1}(\cY,\omega_{\cY/S}^{\bullet}).
\end{equation}
\end{cor}
\begin{proof}
The first assertion is obvious.
For the proof of \eqref{syn-cpx-hom-log-rel-2} it is enough to show the following lemma.
\end{proof}
\begin{lem}\label{syn-cpx-hom-lem}
\begin{enumerate}
\item
$\omega^1_{\cY/S}$ is a locally free $\O_\cY$-module of rank $d=\dim(\cY/S)$,
and $\omega^1_{\cY_n/S_n}\cong \omega^1_{\cY/S}\ot\Z/p^n$.
\item 
Let $\cE^\bullet$ 
be a complex of locally free $\O_\cY$-modules of finite rank. 
If $f$ is projective, then
\[
H^i(\cY,\cE^\bullet)\os{\cong}{\lra}\varprojlim_n H^i(\cY,\cE^\bullet\ot\Z/p^n).
\]
\end{enumerate}
\end{lem}
\begin{proof}
The first assertion can be checked locally in the etale topology on
noticing that $f$ is locally given by 
$(x_0,\cdots,x_d)\mapsto \l=x_0^{a_0}\cdots x_s^{a_s}$ with $a_i>0$ prime to $p$.
We show the latter. We first note that $H^i(\cY,\cE^\bullet)$ and $H^i
(\cY,\cE^\bullet\ot\Z/p^n)$
are finitely generated $W[[\l]]$-modules since $f$ is proper.
By the exact sequence
\[
0\lra \cE^\bullet\lra \cE^\bullet\lra \cE^\bullet\ot\Z/p^n\lra0
\]
one has an exact sequence
\[
0\lra H^i(\cY,\cE)/p^n\lra H^i(\cY,\cE\ot\Z/p^n)\lra H^{i+1}(\cY,\cE)[p^n]\lra0
\]
where $M[p^n]$ denotes the $p^n$-torsion part.
Therefore the assertion is reduced to show that, for any 
finitely generated $W[[\l]]$-module $M$, 
\begin{equation}\label{M-lem-1}
M\os{\cong}{\lra}\varprojlim_n M/p^nM
\end{equation}
where the transition map $M/p^{n+1}\to M/p^n$ is the natural surjection, and
\begin{equation}\label{M-lem-2}
\varprojlim_n M[p^n]=0
\end{equation}
where the transition map $M[p^{n+1}]\to M[p^n]$ is multiplication by $p$.
The former assertion \eqref{M-lem-1} 
is a special case of the general fact $\hat{A}\ot_AM\cong \hat{M}$
for a finitely generated $A$-module $M$ with $A$ noetherian
(\cite{atiyah-mac} Prop.10.3).
We show \eqref{M-lem-2}.
Let
\[
0\lra M_1\lra M_2\lra M_3\lra0
\]
be any exact sequence of finitely generated $W[[\l]]$-modules. Then
\[
0\to M_1[p^n]\to M_2[p^n]\to M_3[p^n]\to M_1/p^n\to M_2/p^n\to M_3/p^n\to0.
\]
Suppose that $M_2[p^n]=0$ for all $n$. Then, by taking the projective limit, we have
\[
\xymatrix{
0\ar[r]&\varprojlim_n M_3[p^n]\ar[r]&
\varprojlim_n(M_1/p^n)\ar[r]& \varprojlim_n(M_2/p^n)\\
&&M_1\ar[r]\ar[u]^\cong&M_2\ar[u]_\cong}
\]
and hence $\varprojlim_n M_3[p^n]=0$ follows.
For the proof of \eqref{M-lem-2}, apply $M_3=M$ and $M_2$ a free $W[[\l]]$-module of 
finite rank.
\end{proof}
\begin{cor}\label{syn-cpx-hom-cor}
Let $f:\cY\to S$ be as in Proposition \ref{syn-cpx-hom-prop}.
If $\dim(\cY/S)<r$ then we have a map
\begin{equation}\label{syn-cpx-hom-log-rel-3}
\phi_{\cY_n/S_n}:\cS_n(r)_{\cY(D)}
\lra 
\omega_{\cY_n/S_n}^{\bullet}[-1]
\end{equation}
which makes the following commutative
\begin{equation}\label{syn-cpx-hom-log-rel-5}
\xymatrix{
\cS_n(r)_{\cY(D)}\ar[r]\ar[d]_{\eqref{syn-cpx-eq3}}&
\omega_{\cY_n/S_n}^{\bullet}[-1]\ar[d]^{p^r}\\
\cS_n(r)'_{\cY(D)}\ar[r]^{\eqref{syn-cpx-hom-log-rel-1}}&
\omega_{\cY_n/S_n}^{\bullet}[-1].
}
\end{equation}
If $f$ is projective, then we have
\begin{equation}\label{syn-cpx-hom-log-rel-4}
\phi_{\cY_n/S_n}:H^i_\syn(\cY(D),\Z_p(r))
\lra 
H^{i-1}(\cY,\omega^\bullet_{\cY/S})
\end{equation}
by taking the projective limit of the cohomology of \eqref{syn-cpx-hom-log-rel-3}
(note Lemma \ref{syn-cpx-hom-lem} (2) for the rght hand side).
\end{cor}
\begin{proof}
Note first that 
the map \eqref{syn-cpx-hom-log-rel-f} is {\it not} a morphism of complexes.
In a precise sense,
there are a quasi isomorphism $\omega_{\cY_n/S_n}^\bullet\to
\wt\omega_{\cY_n/S_n}^\bullet$ and a morphism $\wt{f}_{\cY_n/S_n}:\omega_{\cY_n/S_n}^{\bullet\geq r}\to
\wt\omega_{\cY_n/S_n}^\bullet$ of complexes which sits into a commutative diagram
\[
\xymatrix{
J^{[\bullet-r]}_{D_n}
\ot \omega^\bullet_{(P_n,M_{P_n})}\ar[d]\ar[r]^{p^r-f}
&\O_{D_n}\ot \omega^\bullet_{(P_n,M_{P_n})}\ar[d]^h\\
\omega_{\cY_n/S_n}^{\bullet\geq r}\ar[r]^{\wt{f}_{\cY_n/S_n}}
&\wt\omega_{\cY_n/S_n}^\bullet.
}
\]
Note that $h$ is a quasi isomorphism.
Suppose that $\dim(\cY/S)<r$ and hence $\omega_{\cY_n/S_n}^{\bullet\geq r}=0$.
Now one immediately has a commutative diagram
\[
\xymatrix{
\cS_n(r)_{\cY(D)}\ar[r]\ar[d]_{\eqref{syn-cpx-eq3}}&
\wt\omega_{\cY_n/S_n}^{\bullet}[-1]\ar[d]^{p^r}\\
\cS_n(r)'_{\cY(D)}\ar[r]&
\wt\omega_{\cY_n/S_n}^{\bullet}[-1]
}
\]
of complexes and hence the diagram \eqref{syn-cpx-hom-log-rel-5} 
in the derived category. 
\end{proof}
\subsection{Vanishing of $\wt E^{(n)}_1(0)$}\label{mainpf2-1-sect}
We first show $\wt E^{(n)}_1(0)=0$.
Let $P_\nu\in Z_K$ be the singular point \eqref{sing-point} for $\nu\in W$ such that
$\nu^N=-1$. Let $(u,v)$ be a local coordinate of $\cY_K$ at $P_\nu$
such that $f(u,v)=\l$ where $f:\cY_K\to \Spec K[[\l]]$.
Then 
\[
\Res_{P_\nu}:\Omega^1_{\cX_K/K((\l))}\to \Omega^1_{K((\l))((u))/K((\l))}\to K((\l)),
\]
be the residue map where the second arrow is given by
\[
\sum_n c_n(\l)u^n\frac{du}{u}\longmapsto c_0(\l).
\]
This does not depend on the choice of the coordinates $(u,v)$ up to sign.
Obviously
\[
\Res_{P_\nu}(\Omega^1_{\cY_K/K[[\l]]})\subset K[[\l]]. 
\]
Thus the above residue map induces a commutative diagram
\[
\xymatrix{
H^1_\zar(\Omega^{\bullet\leq 1}_{\cY_K/K[[\l]]})\ar[r]\ar[d]&H^1_\dR(\cX_K/K((\l)))
\ar[d]^{\Res_{P_\nu}}\\
K[[\l]]\ar[r]^\subset&K((\l))
}
\]
where $H^i_\zar(\Omega^{\bullet\leq 1}_{\cY_K/K[[\l]]})
=H^i_\zar(\cY_K,\Omega^{\bullet\leq 1}_{\cY_K/K[[\l]]})$ 
denotes the Zariski cohomology of the complex
$\O_{\cY_K}\to \Omega^1_{\cY_K/K[[\l]]}$.

Let $Z\subset\cY$ be the central fiber over $\l=0$, which is a relative NCD
over $W$ of relative dimension $1$. 
Let $P_\nu$ be the intersection points as in \eqref{sing-point}.
and let $i_Z:Z_\bullet\to Z\subset \cY$ 
be the simplicial nerve of the normalization $\wt{Z}\to Z$.
Let $Z^{\mathrm{reg}}:=Z\setminus\{P_\nu\}$
and $Z_K^{\mathrm{reg}}:=Z_K\setminus\{P_\nu\}$.
There is the residue map $\Res_{P_\nu}^Z:H^1_\dR(Z^\reg_K/K)\to K$ at $P_\nu$
which is compatible with the above.
We consider a commutative diagram
\begin{equation}\label{mainpf2-1-eq1}
\xymatrix{
&H^2_\syn(Z_\bullet,\Z_p(2))\ar[d]_{\phi_{Z_\bullet}}
&H^2_\syn(\cY,\Z_p(2))\ar[d]^\phi\ar[l]\ar[r]&
H^2_\syn(\cY(Z),\Z_p(2))\ar[d]^{\phi_{\cY/W[[\l]]}}\\
H^1_\dR(Z^\reg_K/K)\ar[rd]_{\Res^Z_{P_\nu}}
&H^1_\dR(Z_{\bullet}/W)\ar[l]\ar[d]^{\text{0 map}}&
H^1_\zar(\Omega^{\bullet\leq 1}_{\cY/W[[\l]]})\ar[r]^\subset\ar[d]\ar[l]_{i_Z^*}
&H^1_\zar(\omega^{\bullet}_{\cY/W[[\l]]})\ar[d]^{\Res_{P_\nu}}\\
&K&K[[\l]]\ar[r]^\subset\ar[l]_{\mod \l}&K((\l)).
}
\end{equation}
Here $\phi_{Z_\bullet}$ is as in \eqref{syn-cpx-hom},
and $\phi_{\cY/W[[\l]]}$ is as in 
Corollary \ref{syn-cpx-hom-cor} \eqref{syn-cpx-hom-log-rel-4}, and
$\phi$ is the induced homomorphism from $\phi_{\cY/W[[\l]]}$.
\begin{lem}\label{mainpf2-1-lem2}
The symbol $\{h_1,h_2\}\in K_2(X)$ has no boundary at $Z=f^{-1}(0)$.
In particular $\{h_1,h_2\}|_\cX\in K_2(\cX)$ lies in the image of $K_2(\cY)$.
\end{lem}
\begin{pf}
We already know the explicit description of $Z$ in \S \ref{main-sect-3}.
The assertion is now an easy exercise.
\end{pf}
We turn to the proof of $\wt E^{(n)}_1(0)=0$.
By Lemma \ref{mainpf2-1-lem2},
an element 
\begin{equation}\label{mainpf2-1-eq2}
R:=\reg_\syn(\{h_1,h_2\})|_{\cY}\in H^2_\syn(\cY,\Q_p(2))
\end{equation}
 is defined.
Let $c:H^1_\zar(\Omega^{\leq 1}_{\cY/W[[\l]]})\to H^1_\dR(\cX/K((\l)))$ be
the natural map.
Then one easily sees
\begin{equation}\label{mainpf2-1-eq3}
c\circ \phi\circ\text{can}(R)=\sum_{n=1}^{N-1}\wt E^{(n)}_1(\l)\wt\omega_n
+\wt E^{(n)}_2(\l)\wt\eta_n\in H^1_\dR(\cX_K/K((\l))).
\end{equation}
A direct computation yields $\Res_{P_\nu}(\wt\eta_n)=0$ and
\[
\Res_{P_\nu}(\omega_n)=(-1)^s\nu^{-n}+\mbox{(higher terms)}\in K[[\l]]
\]
where $s\in \Z$ does not depend on $n$.
By a diagram chase of \eqref{mainpf2-1-eq1}, we have
\[
\sum_{n=1}^{N-1}\nu^{-n}\cdot\wt E^{(n)}_1(0)=0
\]
for all $\nu$ such that $\nu^N=-1$.
This implies $\wt E^{(n)}_1(0)=0$ for all $n$ as required.

\subsection{Computing $\wt E^{(n)}_2(0)$}\label{mainpf2-2-sect}
We keep the notation in the diagram \eqref{mainpf2-1-eq1}.
Consider a commutative diagram
\begin{equation}\label{mainpf2-2-eq1}
\xymatrix{
K_2(Z_\bullet)\ar[d]_{\reg^{Z_\bullet}_\syn}&K_2(\cY)\ar[l]\ar[d]^{\reg^\cY_\syn}\\
H^2_\syn(Z_\bullet,\Z_p(2))\ar[d]_\cong&H^2_\syn(\cY,\Z_p(2))\ar[d]^\phi\ar[l]&\\
H^1_\dR(Z_\bullet/W)&
H^1_\zar(\Omega^\bullet_{\cY/W[[\l]]})\ar[r]\ar[l]_{i_Z^*}\ar[d]
&H^1_\dR(\cX_K/K((\l)))\ar[d]\\
H^1(\O_{Z_\bullet})\ar[u]^\cong&H^1(\O_{\cY})\ar[l]_{i'_Z}\ar[r]^j&H^1(\O_{\cX_K})
}
\end{equation}
where $\reg_\syn$ are the syntomic regulator maps.
 Since $H^1(Z,\O_{Z})\cong H^1(Z_\bullet,\O_{Z_\bullet})$ is a free $W$-module of rank 
 $N-1$ (=the genus of the generic fiber $\cX_K$),
it follows from \cite{Ha} III, 12.9 
that $H^1(\O_{\cY})$ is a free $W[[\l]]$-module of rank $N-1$ and
$H^1(\O_{Z})\cong  H^1(\O_{\cY})\ot_{W[\l]}W[\l]/(\l)$.
In particular $j:H^1(\O_\cY)\to H^1(\O_{\cX_K})\cong H^1(\O_{\cY_K})\ot_{K[\l]} K(\l)$
is an inclusion 
and $i'_Z$ is the reduction modulo $\l$.

\medskip

We prove that $\wt E^{(n)}_2(0)$ satisfies
the initial conditions \eqref{main-diffeq-3}.
Recall the element \eqref{mainpf2-1-eq2} in $H^2_\syn(\cY,\Q_p(2))$.
By the diagram \eqref{mainpf2-2-eq1}, we have an element
\[
\sum_{n=1}^{N-1}\wt E^{(n)}_1(\l)\wt\omega_n
+\wt E^{(n)}_2(\l)\wt\eta_n\in H^1_\dR(\cX_K/K((\l))),
\]
and an element
\[
\beta\in H^1_\dR(\O_{Z_\bullet}).
\]
and they correspond as follows
\[
\sum_{n=1}^{N-1}\wt E^{(n)}_2(0)\cdot i'_Z(\wt\eta_n)=\pm\beta\in H^1_\dR(\O_{Z_\bullet}).
\]
Note that $\{i'_Z(\wt\eta_n)\mid n=1,\ldots,N-1\}$ 
is a basis of $H^1(\O_{Z_K})\cong K^{N-1}$.
Thanks to the compatibility of the syntomic regulator maps,
$\beta$ coincides with
the syntomic regulator of $\{h_1,h_2\}|_{Z_\bullet}\in K_2(Z_\bullet)$ up to sign.
Recall  the definition of $\wt\ve^{(n)}_i(\l)$ from \eqref{mainpf1-eq1}.
By Lemma \ref{gm-lem1} it turns out
\[
\sum_{n=1}^{N-1}\wt E^{(n)}_1(\l)\wt\omega_n
+\wt E^{(n)}_2(\l)\wt\eta_n
=p^{-1}\log^{(\sigma)}(h_1)\frac{dh_2^\sigma}{h_2^\sigma}
-\log^{(\sigma)}(h_2)\frac{dh_1}{h_1}\in H^1_\dR(\cX_K/K((\l)))
\]
and hence we have
\begin{equation}\label{mainpf2-2-eq2}
\left(p^{-1}\log^{(\sigma)}(h_1)\frac{dh_2^\sigma}{h_2^\sigma}
-\log^{(\sigma)}(h_2)\frac{dh_1}{h_1}\right)\bigg|_{Z_\bullet}=\sum_{n=1}^{N-1}
\wt E^{(n)}_2(0)\cdot i'_Z(\wt\eta_n)\in H^1(\O_{Z_\bullet})
\end{equation}
without ambiguity of sign.
We compute the both sides of \eqref{mainpf2-2-eq2} explicitly. 
Let $\{P_\nu\}=Z_1\cap Z_2$ be the set of intersection points \eqref{sing-point}
and $\phi_i:\{P_\nu\}\hra Z_{i,K}$ the embeddings.
We fix an isomorphism
\begin{equation}\label{mainpf2-2-eq3}
\left(\bigoplus_{\nu^N=-1}H^0(\O_{P_\nu})\right)/\Image \phi\os{\cong}{\lra}
H^1(\O_{Z_K})\os{\cong}{\lra}
H^1(\O_{Z_\bullet})
\end{equation}
which arises from an exact sequence
\[
0\lra \O_{Z_K}\lra \O_{Z_{1,K}}\op \O_{Z_{2,K}}\os{\phi}{\lra} \bigoplus_\nu\O_{P_\nu}\lra 0
\]
where we put $\phi:=\phi_1^*-\phi^*_2$.
We denote an element $(c_\nu)_\nu\in (\op_\nu H^0(\O_{P_\nu}))/\Image\phi$
by $\sum_\nu c_\nu [P_\nu]$.
Then it is not hard to show
\[
i'_Z(\wt\eta_n)=-\sum_\nu\nu^{-n}[P_\nu]
\]
and hence we have
\begin{equation}\label{mainpf2-2-eq5}
\mbox{RHS of \eqref{mainpf2-2-eq2}}=-\sum_{\nu^N=-1}
\left(\sum_{n=1}^{N-1}
\wt E^{(n)}_2(0)\nu^{-n}\right)[P_\nu]
\end{equation}
under the identification \eqref{mainpf2-2-eq3}.
Next we compute the left hand side of \eqref{mainpf2-2-eq2}.
Recall the affine open set \eqref{int-locus3}   
\[
W_2=\Spec \Z_p[z,w_1,\l,(z-1)^{-1}]/(zw_1^N-(z-1)^{N-1}(z-\l))\subset Y.
\]
Then $Z_1\cap W_2=\{z=\l=0\}$, $Z_2\cap W_2=\{w_1^N-(z-1)^{N-1}=\l=0\}$
and $P_\nu=\{z=w_1+\nu=\l=0\}$.
The symbol $\{h_1,h_2\}|_{W_2}$ is defined by 
\[
h_1|_{W_2}=\frac{w_1-\zeta_1(z-1)}{w_1-\zeta_2(z-1)},\quad
h_2|_{W_2}=\frac{1-\l}{(z-1)^2}.
\]
Let $u$ be the inhomogeneous coordinate of $Z_2\cong \P^1_{\Z_p}$ such that
$u^{N-1}=w_1$ and $u^N=z-1$.
Define $p$-th Frobenius maps $\varphi_i$ on $Z_i\cong\P^1_{\Z_p}$
by $\varphi_1(w_1)=w_1^p$ and $Z_1$ and $\varphi_2(u)=u^p$ on $Z_2$.
Since the LHS of \eqref{mainpf2-2-eq2} coincides with the syntomic regulator
$\reg_\syn\{h_1,h_2\}|_{Z_\bullet}$ up to sign, and it does not depends on the choice
of the Frobenius, we may assume that $\sigma$ is given by the Frobenius
$\varphi=(\varphi_1,\varphi_2)$ on $Z_\bullet$.
Let
\[
{\mathscr S}_{\wt Z}(2):\O_{Z_1}\op\O_{Z_2}\os{d}{\lra}
\Omega^1_{Z_1}\op\Omega^1_{Z_2},
\]
\[
{\mathscr S}_{Z_1\cap Z_2}(2):\bigoplus_\nu\O_{P_\nu}\lra 0
\]
be the syntomic complexes on $\wt Z=Z_1\coprod Z_2$ and
$Z_1\cap Z_2$ respectively where the first terms are placed in degree $0$ (\cite{Ka1} 2.5).
The syntomic complex ${\mathscr S}_{Z_\bullet}(2)$ of $Z_\bullet$ is given 
as the mapping fiber of $\phi=\phi_1^*-\phi_2^*:{\mathscr S}_{\wt Z}(2)\to 
{\mathscr S}_{Z_1\cap Z_2}(2)$.
Note $h_2|_{Z_1}=1$ and
\[
h_1|_{Z_2}=\frac{1-\zeta_1 u}{1-\zeta_2 u},\quad h_2|_{Z_2}=u^{-2N}.
\]
Therefore the LHS of \eqref{mainpf2-2-eq2} is given as the cohomology class of
\[
R_n:=\left(0,-2N\log^{(p)}\left(\frac{1-\zeta_1 u}{1-\zeta_2 u}\right)\frac{du}{u}\right)
\in \Omega^1_{Z_{1,n}}\op\Omega^1_{Z_{2,n}},\quad Z_{i,n}:=Z_i\ot\Z/p^{n+1}\Z
\]
where $\log^{(p)}h:=p^{-1}\log h(u)^p/h(u^p)$, namely
\begin{align*}
\mbox{LHS of \eqref{mainpf2-2-eq2}}
=(\mbox{the cohomology class of }(R_n)_{n\geq 0})
&\in 
\Q\ot\left(\varprojlim_n H^1_\syn(Z_\bullet,\Z/p^{n+1}(2))\right)\\
&=
H^1_\syn(Z_\bullet,\Q_p(2))\cong H^1(\O_{Z_\bullet}).
\end{align*}
We write down this in terms of the $p$-adic dilog function
\[
\mathrm{ln}_2^{(p)}(u):=\sum_{p\not|n}\frac{u^n}{n^2}\in \Z_p\left\{u,\frac{1}{1-u}\right\}^\dag.
\]
Note
\[
-\log^{(p)}(1-u)\frac{du}{u}
=d\mathrm{ln}^{(p)}_2(u).
\]
Therefore
the isomorphism
\[
H^2_\syn(Z_\bullet,\Z_p(2))\os{\cong}{\lra}H^1(\O_{Z_\bullet}) \os{\cong}{\lra}
\left(\bigoplus_{\nu^N=-1}H^0(P_\nu)\right)/\Image \phi
\]
sends the cohomology class of $R$ to an element
\[
-2N\phi[\mathrm{ln}^{(p)}_2(\zeta_1u)-\mathrm{ln}^{(p)}_2(\zeta_2u)]
=
-2N\sum_\nu(\mathrm{ln}^{(p)}_2(\zeta_1\nu^{-1})-\mathrm{ln}^{(p)}_2(\zeta_2\nu^{-1}))
\cdot[P_\nu].
\]
Here we note $P_\nu=\{w_1=-\nu\}=\{u=\nu^{-1}\}$.
We thus have
\begin{equation}\label{mainpf2-2-eq6}
\mbox{LHS of \eqref{mainpf2-2-eq2}}
=-2N\sum_{\nu^N=-1}(\mathrm{ln}^{(p)}_2(\zeta_1\nu^{-1})
-\mathrm{ln}^{(p)}_2(\zeta_2\nu^{-1}))\cdot[P_\nu]
\end{equation}
under the identification \eqref{mainpf2-2-eq3}.
Comparing \eqref{mainpf2-2-eq6} with \eqref{mainpf2-2-eq5}, we have
\[
-\sum_{n=1}^{N-1}
\wt E^{(n)}_2(0)\nu^{-n}
=-2N(\mathrm{ln}^{(p)}_2(\zeta_1\nu^{-1})
-\mathrm{ln}^{(p)}_2(\zeta_2\nu^{-1}))\]
for all $\nu$, which uniquely determines $\wt E^{(n)}_2(0)$.
One easily finds
\begin{align*}
\wt E^{(n)}_2(0)&=2\sum_{\nu^N=-1}
\nu^n(\mathrm{ln}^{(p)}_2(\zeta_1\nu^{-1})-\mathrm{ln}^{(p)}_2(\zeta_2\nu^{-1}))\\
&=2(\zeta^n_1-\zeta^n_2)\sum_{\nu^N=-1}
\nu^{-n}\mathrm{ln}^{(p)}_2(\nu).
\end{align*}
This competes the proof of Theorem \ref{main-thm}.
\section{Coleman integrals and Syntomic regulators}\label{coleman-sect}
Let $W=W(\F_q)$ be the Witt ring of a finite field $\F_q$.
Put $K:=\Frac(W)$.
Let $C$ be a smooth projective curve over $K$.
In the seminal paper \cite{CdS}, Coleman and de Shalit defined
a $p$-adic regulator map
\[
r_p:K_2(C)\lra \Hom(\vg(C,\Omega^1_{C/K}),K)
\]
using the {\it Coleman integrals}
\[
\int_{(f)}\log(g)\omega.
\]
We refer to the reader a very nice exposition \cite{Be3} for Coleman integrals.
A. Besser proved that $r_p$ agrees with the syntomic regulator map.
\begin{thm}[\cite{Be2} Theorem 3]\label{besser-thm}
Suppose that $C$ has a good reduction.
Let $\phi:H^1_\dR(C/K)\to H^1_\dR(C/K)$ be the $p$-th Frobenius compatible with
the canonical Frobenius on $K$.
Let
\[
\Theta:H^2_\syn(C,\Q_p(2))\cong H^1_\dR(C/K)
\os{1-p^{-2}\phi}{\us{\sim}{\longleftarrow}}
H^1_\dR(C/K)
\]
be the composition of isomorphisms where the first one is the canonical one,
and in the second arrow ``$1$'' denotes the identity.
Then
\[
\mathrm{Tr}_C(\Theta(\reg_\syn\{f,g\})\cup[\omega])=(r_p\{f,g\})(\omega)
\]
for $\omega\in H^0(C,\Omega^1_{C/K})$.
\end{thm}
Recall from \S \ref{main-sect-2} the HG fibration $f:X\to S$ and
the $K_2$-symbol $\{h_1,h_2\}\in K_2(X)$.
The main theorem (Theorem \ref{main-thm})
describes the syntomic regulator
\[
\reg_\syn\{h_1,h_2\}|_{\l=\alpha}\in H^1_\dR(X_\alpha/K),\quad X_\alpha:=f^{-1}(\alpha)
\]
for $\alpha\in W$ such that $\l^\sigma|_{\l=\alpha}=\alpha^\sigma$
and $\alpha\not\equiv 0,1$ mod $p$,
in terms of overconvergent functions $\ve^{(n)}_i(\l)$.
Define $s_{i,\alpha}^{(n)}\in K$ to be the elements satisfying
\begin{align*}
&
\Theta(\reg_\syn\{h_1,h_2\}|_{\l=\alpha})=
\sum_{n=1}^{N-1}s^{(n)}_{1,\alpha}\,\omega_n+s^{(n)}_{2,\alpha}\,\eta_n\\
\Longleftrightarrow\quad
&\reg_\syn\{h_1,h_2\}|_{\l=\alpha}
=\sum_{n=1}^{N-1}(1-p^{-2}\phi)(s^{(n)}_{1,\alpha}\,\omega_n+s^{(n)}_{2,\alpha}\,\eta_n)
\end{align*}
where $\phi$ is the $p$-th Frobenius on $H^1_\dR(X_\alpha/K)$.
\begin{thm}\label{coleman-thm}
Let $s_{i,\alpha}^{(n)}\in K$ be characterized by
\begin{equation}\label{besser-thm-eq2}
\begin{pmatrix}
(s^{(n)}_{1,\alpha})^\sigma\\
(s^{(n)}_{2,\alpha})^\sigma
\end{pmatrix}
=
p\frac{\alpha^\sigma-(\alpha^\sigma)^2}{\alpha-\alpha^2}
\begin{pmatrix}
F_{22}(\alpha)&-F_{12}(\alpha)\\
-pF_{21}(\alpha)&pF_{11}(\alpha)
\end{pmatrix}
\left[
\begin{pmatrix}
s^{(n)}_{1,\alpha}\\
s^{(n)}_{2,\alpha}
\end{pmatrix}
-
\begin{pmatrix}
\ve^{(n)}_1(\alpha)\\
\ve^{(n)}_2(\alpha)
\end{pmatrix}\right]
\end{equation}
where $F_{ij}(\l)$ are the overconvergent functions in Corollary \ref{frob-eigen}.

Let $r_p:K_2(X_\alpha)\to \Hom(\vg(X_\alpha,\Omega^1_{X_\alpha}),K)$
be the $p$-adic regulator due to Coleman and de Shalit \cite{CdS}.
Then
\begin{equation}\label{besser-thm-eq1}
r_p\{h_1,h_2\}|_{\l=\alpha}(\omega_n)=\int_{(h_1)}\log(h_2)\omega_n|_{\l=\alpha}=\pm
(\alpha-\alpha^2)^{-1}s^{(n)}_{2,\alpha}\in K
\end{equation}
for $\alpha\in W$ such that $\l^\sigma|_{\l=\alpha}=\alpha^\sigma$
and $\alpha\not\equiv 0,1$ mod $p$.
\end{thm}
\begin{pf}
\eqref{besser-thm-eq1} is immediate from Theorem \ref{besser-thm}  
toegther with the fact
\begin{align*}
\mathrm{Tr}_{X_\alpha}([\omega_n]\cup[\eta_n])
&=
-\frac{N}{n}(\alpha-\alpha^2)^{-1}\mathrm{Tr}_{X_\alpha}([\wt\omega_n]\cup[\wt\eta_n])\\
&=
\pm
\frac{N}{n}(\alpha-\alpha^2)^{-1}.
\end{align*}
By Theorem \ref{main-thm} we have
\[
\ve_1^{(n)}(\alpha)\omega_n
+\ve_2^{(n)}(\alpha)\eta_n
=(1-p^{-2}\phi)(s^{(n)}_{1,\alpha}\,\omega_n+s^{(n)}_{2,\alpha}\,\eta_n).
\]
The $p$-th Frobenius $\phi$ on $H^1_\rig(X_\alpha/K)$ is explicitly given 
in Corollary \ref{frob-eigen} \eqref{frob-eigen-eq1}. 
Hence we have
\[
\begin{cases}
\ve_1^{(n)}(\alpha)
=s^{(n)}_{1,\alpha}-p^{-1}(s^{(n)}_{1,\alpha})^\sigma F_{11}(\alpha)
-p^{-2}(s^{(n)}_{2,\alpha})^\sigma F_{12}(\alpha)\\
\ve_1^{(n)}(\alpha)
=s^{(n)}_{2,\alpha}-p^{-1}(s^{(n)}_{1,\alpha})^\sigma F_{21}(\alpha)
-p^{-2}(s^{(n)}_{2,\alpha})^\sigma F_{22}(\alpha)
\end{cases}
\]
which is equivalent to \eqref{besser-thm-eq2} noting \eqref{frob-eigen-det}.
\end{pf}
{\bf Complement.} Put
\[
A:=\frac{\alpha^\sigma-(\alpha^\sigma)^2}{\alpha-\alpha^2}
\begin{pmatrix}
F_{22}(\alpha)&-F_{12}(\alpha)\\
-pF_{21}(\alpha)&pF_{11}(\alpha)
\end{pmatrix},\quad
{\mathbf e}:=\begin{pmatrix}
\ve^{(n)}_1(\alpha)\\
\ve^{(n)}_2(\alpha)
\end{pmatrix}.
\]
\[
A=(PD(P^{-1})^\sigma)^{-1}=P^\sigma D^{-1}P^{-1}
\]
\[
A^{\sigma^{-1}}=PE^{\sigma^{-1}}(P^{-1})^{\sigma^{-1}},
A^{\sigma^{-2}}=P^{\sigma^{-1}}E^{\sigma^{-2}}(P^{-1})^{\sigma^{-2}},
\]
\[
A^{\sigma^{-1}}A^{\sigma^{-2}}\cdots A^{\sigma^{-n}}
=PE^{\sigma^{-1}}\cdots E^{\sigma^{-n}}(P^{-1})^{\sigma^{-n}}
\]
A more explicit description of $s^{(n)}_{i,\alpha}$ satisfying 
\eqref{besser-thm-eq2} is
\begin{equation}\label{besser-thm-eq3}
p^{-1}\begin{pmatrix}
s^{(n)}_{1,\alpha}\\
s^{(n)}_{2,\alpha}
\end{pmatrix}=-A^{\sigma^{-1}}{\mathbf e}^{\sigma^{-1}}
-pA^{\sigma^{-1}}A^{\sigma^{-2}}{\mathbf e}^{\sigma^{-2}}
-p^2A^{\sigma^{-1}} A^{\sigma^{-2}}A^{\sigma^{-3}}{\mathbf e}^{\sigma^{-3}}
-\cdots.
\end{equation}
If $r=[K:\Q_p]$, then $\sigma^{-1}=\sigma^{r-1}$, \ldots. Hence one has overconvergent
functions
\[
s^{(n)}_i(\l)\in \Q_p\{\l,(1-\l)^{-1}\}^\dag
\]
which may depend on $r$,  such that
\[
s_i^{(n)}(\alpha)=s^{(n)}_{i,\alpha}
\]
for all $\alpha\in W$ such that $\l^\sigma|_{\l=\alpha}=\alpha^\sigma$
and $\alpha\not\equiv 0,1$ mod $p$.

\end{document}